\documentclass[a4paper,twoside,10pt]{article}
\usepackage[a4paper,left=2.5cm,right=2.5cm, top=2.5cm, bottom=2.5cm]{geometry}
\usepackage[latin1]{inputenc}
\usepackage{cuted}
\usepackage{orcidlink}
\definecolor{darkblue}{rgb}{0.0, 0.0, 0.55}

\usepackage{tcolorbox}

\usepackage{array}
\usepackage{amssymb}
\usepackage{mathrsfs}
\usepackage{mathtools}
\usepackage{enumitem}
\usepackage{accents}
\usepackage{multirow}
\usepackage{graphicx}
\usepackage{comment}
\usepackage{algorithm2e}
\usepackage{amscd,amsmath,amssymb,mathrsfs,bbm,listings}
\usepackage{cancel}
\usepackage{subcaption}
\allowdisplaybreaks
\graphicspath{{figs/}}
\usepackage{amsmath}
\usepackage{orcidlink}
\usepackage{footmisc}
\usepackage{amsthm}
\usepackage{amssymb}
\usepackage{stmaryrd}
\SetSymbolFont{stmry}{bold}{U}{stmry}{m}{n}\usepackage{bigints}
\usepackage{cite}
\usepackage{color}
\usepackage[abs]{overpic}
\usepackage[font=footnotesize,labelfont=bf]{caption}
\usepackage{cases}
\usepackage{tikz}
\usepackage{rotating}
\usepackage{blkarray}
\usetikzlibrary{matrix,calc,arrows,cd}
\usepackage{soul,xcolor}
\usepackage{enumitem}
\usepackage{verbatim}
\usepackage{graphicx}
\usepackage{subcaption}
\usepackage{mwe}
\usepackage{tikz}
\usepackage{hyperref}
\usepackage{amsthm}

\usepackage{pifont}

\newcommand{\jump}[1]{\llbracket #1 \rrbracket}
\newcommand{\av}[1]{\{\!\!\{#1\}\!\!\}}

\numberwithin{equation}{section}
\hypersetup{
    colorlinks,
    linktoc=section,
    citecolor=red,
    linkcolor=darkblue,
    urlcolor=magenta,
    pdfborder={0 0 0}
}

\newcommand{\cmark}{\ding{51}} 
\newcommand{\xmark}{\ding{55}} 
\newtheorem{theorem}{Theorem}[section]
\newtheorem{lemma}[theorem]{Lemma}
\newtheorem{proposition}[theorem]{Proposition}
\newtheorem{assumption}{Assumption}[section]
\newtheorem{corollary}[theorem]{Corollary}
\newtheorem{remark}[theorem]{Remark}

\newcommand{\eremk}{\hbox{}\hfill\rule{0.8ex}{0.8ex}}

\newcommand{\bO}{\boldsymbol{0}}
\newcommand{\R}{\mathbb{R}}
\newcommand{\N}{\mathbb{N}}

\renewcommand{\H}{\boldsymbol{H}}
\renewcommand{\L}{\boldsymbol{L}}
\newcommand{\WW}{\boldsymbol{W}}
\newcommand{\bW}{\boldsymbol{W}}

\newcommand{\Z}{\boldsymbol{\mathcal{Z}}}

\newcommand{\dS}{\,\mathrm{d}S}
\newcommand{\dt}{\,\mathrm{d}t}
\newcommand{\ds}{\,\mathrm{d}s}

\newcommand{\Norm}[2]{\|#1\|_{#2}}
\newcommand{\semiNorm}[2]{| #1 |_{#2}}

\newcommand{\dn}{\#}  

\renewcommand{\dh}{d_h}
\newcommand{\sh}{s_h}

\newcommand{\pus}{{\mu_s}}
\newcommand{\pds}{{\mu_s}}
\newcommand{\pdsb}{{\mu_b}}
\newcommand{\pts}{{\mu_s}}
\newcommand{\ptsi}{{\mu_{\sigma}}}
\newcommand{\ptt}{{\mu_{\tau}}}

\newcommand{\calD}{\mathcal{D}}
\newcommand{\QT}{Q_T}
\newcommand{\nOmega}{\boldsymbol{n}_{\Omega}}
\newcommand{\nf}{\boldsymbol{n}_f}

\newcommand{\nb}{\boldsymbol{n}}
\newcommand{\f}{\boldsymbol{f}}

\renewcommand{\u}{\boldsymbol{u}}
\renewcommand{\v}{\boldsymbol{v}}
\newcommand{\z}{\boldsymbol{z}}
\newcommand{\w}{\boldsymbol{w}}
\newcommand{\B}{\boldsymbol{B}}
\newcommand{\C}{\boldsymbol{C}}

\newcommand{\nS}{\nu_S}
\newcommand{\nM}{\nu_M}
\newcommand{\dpt}{\partial_t}
\newcommand{\ddt}{\frac{\mathrm{d}}{\mathrm{d}t}}
\newcommand{\Nabla}{\boldsymbol{\nabla}}
\newcommand{\e}{\boldsymbol{\varepsilon}}
\DeclareMathOperator{\curl}{\mathbf{curl}}
\DeclareMathOperator{\ddiv}{div}
\DeclareMathOperator{\Div}{\bf{div}}
\DeclareMathOperator{\esssup}{ess\,sup}

\newcommand{\Th}{\mathcal{T}_h}

\newcommand{\hK}{h_K}
\newcommand{\Fh}{\mathcal{F}_h}
\newcommand{\Fho}{\mathcal{F}_h^{\mathcal{I}}}
\newcommand{\Fhb}{\mathcal{F}_h^{\partial}}

\newcommand{\cont}{\mathrm{cont}}
\newcommand{\uh}{\boldsymbol{u}_h}
\newcommand{\vh}{\boldsymbol{v}_h}
\newcommand{\ph}{p_h}
\newcommand{\qh}{q_h}

\newcommand{\Bh}{\boldsymbol{B}_h}
\newcommand{\Ch}{\boldsymbol{C}_h}
\newcommand{\Pp}[2]{\mathbb{P}^{#1}(#2)}

\newcommand{\Vgrad}{\mathcal{V}_h^{{\rm gr},k+1}}
\newcommand{\bVgrad}{\overline{\mathcal{V}}_h^{{\rm gr},k+1}}
\newcommand{\Vcurl}{\boldsymbol{\mathcal{V}}_h^{\curl,k}}

\newcommand{\Vdiv}{\boldsymbol{\mathcal{V}}_h^{\ddiv,k-1}}

\newcommand{\Xh}{\boldsymbol{\mathcal{Z}}_h^{\curl,k}}

\newcommand{\Id}{\mathrm{Id}}
\newcommand{\Ihcurl}[1]{\mathcal{I}_h^{\curl, #1}}
\newcommand{\Ihgrad}[1]{\mathcal{I}_h^{\cont, #1}}
\newcommand{\Ihdiv}[1]{\mathcal{I}_h^{\ddiv, #1}}
\newcommand{\Picurl}{\Pi_h^{\curl,k}}
\newcommand{\Jhcurl}{\mathcal{J}_h^{\curl,k}}
\newcommand{\Jhdiv}{\mathcal{J}_h^{\ddiv,k}}
\newcommand{\Jhav}{\mathcal{J}^{\mathrm{av,g}}_h}
\newcommand{\Jav}{\mathcal{J}^{\mathrm{av,c}}_h}
\newcommand{\Jhavg}{\mathcal{J}_h^{\mathrm{av},\mathrm{g}}}

\newcommand{\eIu}{\boldsymbol{e}_{\mathcal{I}}^{\u}}
\newcommand{\eIB}{\boldsymbol{e}_{\mathcal{I}}^{\B}}
\newcommand{\ehu}{\boldsymbol{e}_{h}^{\u}}
\newcommand{\ehB}{\boldsymbol{e}_{h}^{\B}}

\newcommand{\eu}{\boldsymbol{e}_{\boldsymbol{u}}}
\newcommand{\eB}{\boldsymbol{e}_{\boldsymbol{B}}}

\newcommand{\pressure}{p_{\sf isotr}}
\newcommand{\gammaone}{\ensuremath{\alpha}}

\title{Robust H(curl)-based finite element methods for the incompressible MHD system} 
\author{L. Beir\~ao da Veiga\thanks{Department of Mathematics and Applications, University of Milano-Bicocca, Via Cozzi 55, 20125 Milan, Italy (\href{mailto: lourenco.beirao@unimib.it}{lourenco.beirao@unimib.it}, \href{mailto:sergio.gomezmacias@unimib.it}{sergio.gomezmacias@unimib.it})} \thanks{IMATI-CNR ``E. Magenes", Via Ferrata 5, 27100 Pavia, Italy}\ \orcidlink{0000-0001-5895-469X} \and S. G\'omez\footnotemark[1] \footnotemark[2]\ \orcidlink{0000-0001-9156-5135} \and I. Perugia\thanks{Faculty of Mathematics, University of Vienna, Oskar-Morgenstern-Platz 1, 1090 Vienna, Austria (\href{mailto:ilaria.perugia@univie.ac.at}{ilaria.perugia@univie.ac.at}, \href{mailto:enrico.zampa@univie.ac.at}{enrico.zampa@univie.ac.at})}\ \orcidlink{0000-0003-1368-2883} \and
E. Zampa\footnotemark[3]\ \orcidlink{0000-0002-1351-8656}
}
\date{}

\begin{document}

\maketitle

\begin{abstract}
\noindent We propose and analyze a class of finite element methods for the time-dependent incompressible magnetohydrodynamics system based on $\H(\curl)$-conforming discretizations for both the velocity and the magnetic field. This choice is guided by the aim of developing methods that are also suitable for the types of solutions arising in problems posed on nonconvex domains.
Within this framework, we introduce three stabilized formulations, and study how the stabilization mechanisms employed influence their structural properties. In particular, we focus on suitability for nonconvex polyhedral domains, the need for Lagrange multipliers for the magnetic field, pressure-robustness, and quasi-robustness with respect to both the fluid and magnetic Reynolds numbers. The proposed formulations are further assessed through numerical experiments, highlighting their practical performance.
\end{abstract}



\section{Introduction}

The magnetohydrodynamic (MHD) model, which combines equations from electromagnetism and fluid dynamics, 
is highly relevant in the study of plasmas and liquid metals, and has applications across numerous fields, including geophysics, astrophysics, and engineering.
The finite element (FE) discretization of MHD systems is a rich and active area of research, presenting a wide range of challenges due to the complexity of the model equations and their underlying topological structure. 

\paragraph{Previous literature.} The existing FE  literature can be broadly classified into three main lines of research:
\begin{itemize}[noitemsep]
\item \textbf{\textit{Structure-preserving methods.}}  
    These approaches aim to preserve, at the discrete level, as many of the invariants of the continuous system as possible, such as the total energy, Gauss' law, magnetic helicity, and cross-helicity. Representative contributions in this direction include, without claiming completeness, those in~\cite{HuLeeXu21, GawlikGB22, LaakmannHuFarrell23,MaoXi25,AndrewsFarrell25}. The preservation of invariants is not merely a theoretical concern: it was shown in~\cite{AndrewsFarrell25} and~\cite{MaoXi25} that numerical schemes that fail to preserve magnetic helicity may exhibit unphysical dissipation.  
    On the other hand, in practice, magnetic-helicity-preserving methods often suffer from spurious oscillations when applied to low-viscosity or low-magnetic-resistivity regimes, thereby severely limiting their applicability.
    Moreover, insights such as Onsager's conjecture suggest that exact preservation of invariants may be incompatible with the approximation of low-regularity solutions.
    Works focusing on the design of structure-preserving discretizations are often characterized by a lack of \emph{a priori} error estimates or a complete convergence analysis.
    A notable exception is the recent work~\cite{BdVHuMascotto25}, where error estimates are established for the method proposed in~\cite{HuLeeXu21}.

\item 
\textbf{\textit{Convergent methods for low-regularity solutions.}}  
    It is well known that, even in magnetostatics, $\H^1$-conforming approximations may lead to spurious modes and incorrect solutions \cite{Costabel91}, thus motivating the study of alternative discretizations for MHD systems. In this direction, Sch\"otzau \cite{Schotzau04} and Prohl \cite{Prohl08} proposed and analyzed convergent finite element methods (FEM) for the stationary and time-dependent MHD problems, respectively. In both works, the magnetic field is approximated using $\H(\curl)$-conforming FE spaces, enabling convergence under minimal regularity assumptions. 

\item
\textbf{\textit{Robust methods with respect to the fluid and magnetic Reynolds numbers.}}  
    In the finite volume literature, it is well established that stabilization or upwinding strategies beyond classical convective stabilization are required to obtain schemes that are robust with respect to the magnetic Reynolds number; see, for instance, \cite{BalsaraDumbser2015, FambriEtAl23}. 
    While similar ideas have been introduced in the FE context~\cite{HiptmairPagliantini18, Fu19, WimmerTang24, ZampaBustoDumbser24}, most of these works lack a rigorous theoretical analysis. To date, 
    initial theoretical studies have primarily focused on the linearized problem; see, for instance, \cite{gerbeau2000stabilized,wacker2016nodal,beirao2024robust}.  
    To the best of the authors' knowledge, the only works providing a proof of robustness in the fully nonlinear case are \cite{RobustMHD, DiPietroDroniuPatierno26}, which are, however, restricted to convex domains and require a Lagrange multiplier to enforce the divergence-free constraint.
\end{itemize}

\paragraph{Motivation.} A natural approach to developing FE schemes that are suitable also for nonconvex polyhedral domains is to make use of~$\H(\curl)$-conforming elements for the magnetic field. 
Although such elements have also been used for fluid flows, as discussed below, the literature on robust~$\H(\curl)$ elements for high Reynolds numbers is very limited compared to that on~$\H^1$- or $\H(\ddiv)$-conforming elements. 

This work aims at contributing towards the design of~$\H(\curl)$-conforming FE schemes for MHD systems with the following properties: \emph{i)} suitability for general domains, \emph{ii)} pressure-robustness,
and \emph{iii)} quasi-robustness with respect to both fluid and magnetic Reynolds numbers. 
By pressure-robustness, we refer to schemes in which the discrete velocity solution 
is independent of modifications of the data that affect only the pressure at the continuous level, such as 
the gradient part of the load terms.
By Reynolds quasi-robustess, we refer to schemes for which the error, measured in suitable norms including convective stability terms, remains bounded as the fluid and/or magnetic Reynolds numbers increase. 

\paragraph{Novelty.}
For the methods proposed in this work, the pressure-robustness property is naturally enforced through the adopted~$\H(\curl)$-based variational formulation. 
Ensuring Reynolds quasi-robustness is more involved and requires the addition of several stabilization terms to the discrete problem. 
In the present framework, we 
analyze how these stabilization mechanisms affect the aforementioned structural properties, and how they impact the minimal solution regularity required for convergence estimates. 

Our assessment of the various advantages and limitations associated with each specific stabilization approach has led to the development of the three methods proposed in this contribution. 
In all cases, we adopt $\H(\curl)$-conforming elements not only for the magnetic field but also for the fluid velocity. 
The motivation for the choice of using such discrete spaces also for the fluid velocity is twofold: the practical advantage of a simplified implementation, and the theoretical interest in investigating~$\H(\curl)$-conforming elements, which have been less explored in the fluid-dynamics setting.
Although unconventional, this choice can be traced back to the seminal work of Girault on the Navier--Stokes equations~\cite{Girault88, Girault90}, and has since been employed in dual-field formulations~\cite{DualFieldNS, MaoXi25}, 
cross-helicity-preserving discretizations of MHD~\cite{HuLeeXu21, LaakmannHuFarrell23}, 
analysis of the Stokes problem with a focus on boundary conditions~\cite{BoonHiptmairTonnonZampa24, BoonTonnonZampa26}, and more recently in the context of polytopal discretizations~\cite{BdVDassiDiPietroDroniou22, DiPietroDroniuQian24}.

Below, we provide a detailed description of the three proposed methods and outline their features.
\begin{itemize}[noitemsep]
\item
Method~1 is fluid Reynolds quasi-robust, pressure robust, and only requires minimal regularity for $\mathcal{O}(h^s)$ convergence (with~$s > 1/2$),  which is consistent with the expected regularity in nonconvex polyhedral domains.
Furthermore, the solenoidal constraint for the magnetic field is enforced naturally in the discrete formulation,thus avoiding the introduction of an additional Lagrange multiplier.
However, a notable drawback of this scheme is its lack of magnetic Reynolds quasi-robustness. 

\item Method~2 can be viewed as a variant of the first scheme which, by introducing additional stabilization terms, also achieves magnetic Reynolds quasi-robustness. The price to pay is \emph{i)} the necessity of a Lagrange multiplier to enforce the divergence-free condition on the magnetic field, and \emph{ii)} the introduction of a strong stabilization term that, in the limit~$h\to 0$, essentially enforces~$H^1$ regularity on the magnetic field. 

\item Method~3 represents a potentially ``ideal'' scheme, since it combines the desirable properties of both approaches above and, in addition, appears capable of achieving faster pre-asymptotic error reduction rates in convection-dominated regimes. Unlike the first two schemes, a complete error analysis for this method is not yet available. Specifically, in this work, we prove that most error terms can be estimated optimally, whereas two terms require an additional, plausible assumption that has not yet been rigorously justified.
\end{itemize}

\noindent Table~\ref{table1} summarizes the properties of the three proposed schemes. From this table, it is clear that the quest for the provably perfect FEM for MHD systems on general domains remains open. Nevertheless, we believe this work provides some promising methods and establishes a solid foundation for further investigation.

\begin{table}[!ht]
\centering
\begin{tabular}{|l|c|c|c|}
\hline
& Method 1 (\S\ref{sec:method1}--\ref{sec:4}) & Method 2 (\S\ref{sec:nm-rob}) & Method 3 (\S\ref{sec:ideal}) \\
\hline 
fluid Reynolds quasi-robust & \cmark & \cmark & \cmark \\
\hline
magnetic Reynolds quasi-robust & \xmark & \cmark & \cmark \scriptsize{(numerical)} \\
\hline
pressure-robust & \cmark & \cmark & \cmark \\
\hline
compatible with nonconvex & \multirow{2}{*}{\cmark} & \multirow{2}{*}{\xmark} & \multirow{2}{*}{\cmark} \\
polyhedral domains  & & & \\
 \hline
 Lagrange multiplier for $\B$ not required & \cmark & \xmark & \cmark \\
\hline
theoretical error analysis & \small{complete} & \small{complete} & \small{partial} \\
\hline
pre-asymptotic error reduction & \small{${\mathcal O}(h^k)$} & \small{${\mathcal O}(h^k)$} & \small{${\mathcal O}(h^{k+\frac12})$} \scriptsize{(numerical)}\\
\hline
\end{tabular}
\caption{Summary of the main properties of the three proposed methods. Symbols: \cmark\ indicates the property holds, \xmark\ indicates it does not hold. The exponent~$k$ is the degree of the polynomials used in the approximation. For Method 3, some properties are observed numerically (proven under plausible assumptions).}
\label{table1}
\end{table}

In the final part of the contribution, we present a set of computational tests comparing the practical performance of the proposed methods in terms of both convergence to manufactured solutions and behavior on benchmark problems. In particular, for Method~3, the results obtained provide numerical evidence supporting the expected properties.

\paragraph{Outline.} The article is organized as follows. At the end of the present section, we introduce the model equations and express them in an equivalent form, which will be useful in the sequel. After presenting the mesh, the discrete spaces and related assumptions in Section~\ref{sec:mesh}, we describe the first proposed method in Section~\ref{sec:method1}. The theoretical analysis of such scheme is given in Section~\ref{sec:4}. The two aforementioned variants are introduced in Section~\ref{sec:nm-rob} and Section~\ref{sec:ideal}, respectively, together with the associated theoretical developments. Finally, numerical tests in two space dimensions are presented in Section~\ref{sec:num}.

\paragraph{Governing equations.}
Let the space--time cylinder~$\QT := \Omega \times I$, where~$\Omega \subset \R^3$ is a spatial domain with  Lipschitz boundary~$\partial \Omega$, and~$I = (0, T)$ is a time interval with final time~$T > 0$. We denote by~$\nOmega$ the unit normal vector pointing outward~$\Omega$.

Given an external force~$\f : \QT \to \R^3$, initial data~$\u_0 : \Omega \to \R^3$ and~$\B_0 : \Omega \to \R^3$, and positive fluid~$(\nS)$ and magnetic~$(\nM)$ scaled diffusivity parameters, we consider the following time-dependent magnetohydrodynamics (MHD) system: find the velocity~$\u : \QT \to \R^3$, the isotropic pressure~$\pressure : \QT \to \R$, and the magnetic induction~$\B : \QT \to \R^3$ such that
\begin{subequations}
\label{eq:MHD-system}
\begin{alignat}{3}
\label{eq:MHD-system-1}
\dpt \u - 2\nS \Div \e(\u) + (\Nabla \u) \u + \B \times \curl \B - \nabla \pressure & = \f & & \quad \text{ in~$\QT$}, \\
\label{eq:MHD-system-2}
\dpt \B + \nM \curl (\curl \B) - \curl(\u \times \B) & = \bO & & \quad \text{ in~$\QT$}, \\ 
\label{eq:MHD-system-3}
\ddiv \u = 0 \quad \text{ and }\quad  \ddiv \B = 0 & & & \quad \text{ in~$\QT$}, \\
\label{eq:MHD-system-4}
\u  = \bO, \quad \curl \B \times \nOmega = \bO, \quad \B \cdot \nOmega = 0  & & & \quad \text{ on~$\partial \Omega \times I$}, \\
\label{eq:MHD-system-5}
\u = \u_0 \quad \text{ and } \quad \B = \B_0 & & & \quad \text{ on~$\Omega \times \{0\}$}.
\end{alignat}
\end{subequations}
Recalling the identities (to be intended in the distributional sense)
\begin{subequations}
\begin{alignat}{3}
\label{eq:diff-identities-1}
-2 \Div \e(\u) & = \curl (\curl \u) - 2 \Nabla (\ddiv \u),\\
\label{eq:diff-identities-2}
(\Nabla \u) \u & = (\curl \u) \times \u + \frac12 \nabla |\u|^2 ,
\end{alignat}
\end{subequations}
the solution~$(\u, \pressure,\B)$ to the MHD system~\eqref{eq:MHD-system} satisfies also the following equations: 
\begin{subequations}
\label{eq:MHD-system-curl}
\begin{alignat}{3}
\label{eq:MHD-system-curl-1}
\dpt \u + \nS \curl(\curl \u) + (\curl \u) \times \u + \B \times \curl \B - \nabla p & = \f & & \quad \text{ in~$\QT$}, \\
\label{eq:MHD-system-curl-2}
\dpt \B + \nM \curl (\curl \B) - \curl(\u \times \B) & = \bO & & \quad \text{ in~$\QT$}, \\ 
\label{eq:MHD-system-curl-3}
\ddiv \u = 0 \quad \text{ and }\quad  \ddiv \B = 0 & & & \quad \text{ in~$\QT$}, \\
\label{eq:MHD-system-curl-4}
\u = \bO , \quad \curl \B 
\times \nOmega = \bO, \quad \B \cdot \nOmega = 0  & & & \quad \text{ on~$\partial \Omega \times I$}, \\
\label{eq:MHD-system-curl-5}
\u = \u_0 \quad \text{ and } \quad \B = \B_0 & & & \quad \text{ on~$\Omega \times \{0\}$},
\end{alignat}
\end{subequations}
since the last term in~\eqref{eq:diff-identities-1} vanishes due to the incompressibility condition on~$\u$ in~\eqref{eq:MHD-system-3}, and the last term in~\eqref{eq:diff-identities-2} is ``absorbed" into the modified pressure~$p := \pressure + |\u|^2/2$. 

\paragraph{Function spaces.}
In the following, we use standard notation for~Sobolev, $L^q$, and Bochner spaces. 
For instance, given~$q \in [1, \infty]$, $s \geq 0$, and an open, bounded set~$\calD \subset \R^d$ ($d \in \{2, 3\}$) with Lipschitz boundary~$\partial \calD$, we denote by~$W^{s, q}(\calD)$ the corresponding Sobolev space with seminorm~$\semiNorm{\cdot}{W^{s, q}(\calD)}$ and norm~$\Norm{\cdot}{W^{s, q}(\calD)}$. In particular, we have~$L^q(\calD) := W^{0, q}(\calD)$,  and~$L^2(\calD)$ is the space of square integrable functions in~$\calD$ with inner product~$(\cdot, \cdot)_{\calD}$ and norm~$\Norm{\cdot}{L^2(\calD)}$. Moreover, we denote by~$H^s(\calD) := W^{s, 2}(\calD)$ with seminorm~$\semiNorm{\cdot}{H^s(\calD)}$ and norm~$\Norm{\cdot}{H^s(\calD)}$. We use boldface to denote spaces of vector-valued functions with three components. 

In addition, we define the following spaces:
\begin{alignat*}{3}
\H(\ddiv; \calD) & := \{\v \in \L^2(\Omega) \ : \ \ddiv \v \in L^2(\calD)\}, \\
\H(\curl; \calD) & := \{\v \in \L^2(\Omega) \ : \ \curl \v \in \L^2(\calD)\},
\end{alignat*}
as well as the kernel space
\begin{equation}\label{eq:def:Z}
\Z := \big\{ \v \in \H(\curl; \Omega) \, : \, (\v, \nabla q)_{\Omega} = 0 
\text{ for all~$q \in H^1(\Omega)$} \big\},
\end{equation}
where~$\Omega$ is the spatial domain in the MHD system~\eqref{eq:MHD-system}.

Finally, given a time interval~$(a, b)$ and a Banach space~$(Z, \Norm{\cdot}{Z})$, we define the corresponding Bochner space as
\begin{equation*}
L^q(a, b; Z) := \big\{v : (a, b) \to Z \ : \ \Norm{v}{L^q(a, b; Z)} < \infty \big\},
\end{equation*}
where
\begin{alignat*}{3}
\Norm{v}{L^q(a, b; Z)} := \begin{cases}
\displaystyle \Big(\int_{a}^b \Norm{v(t)}{Z}^q\Big)^{1/q} \dt & \text{ if~$q \in [1, \infty)$}, \\
\displaystyle \esssup_{t \in (a, b)} \Norm{v(t)}{Z} & \text{ if~$q = \infty$}.
\end{cases}
\end{alignat*}

\begin{remark}[Existence of continuous weak solutions]
The existence of a solution $(\u, p,\B)$ to problem~\eqref{eq:MHD-system}, written in variational form and posed in suitable Bochner spaces, is discussed, e.g., in \cite[\S2.2]{Gerbeau-etal-book:2006}. It suffices, for instance, to assume that the external force~$\f 
\in\L^2(\QT)$, and that the initial data~$(\u_0, \B_0) \in \H^1(\Omega) \times \H(\curl; \Omega)$ have vanishing divergence and satisfy the respective boundary conditions. In the following, we assume that these conditions on the data hold; in particular, we assume that~$\u_0$ and $\B_0$ belong to $\Z$, which implies that both $\u(\cdot, t)$ and $\B(\cdot, t)$ belong to~$\Z$ for a.e. $t \in  (0,T)$.
\eremk          
\end{remark}


\section{Meshes and discrete spaces}\label{sec:mesh}
Let~$\{\Th\}_{h > 0}$ be a family of conforming tetrahedral meshes of the spatial domain~$\Omega$. We denote the set of all faces of $\Th$ by~$\Fh$, while~$\Fho$ and~$\Fhb$  denote the sets of internal and boundary faces, respectively. For each~$f \in \Fho$, we denote by $\nb_f$ one of its two unit normal vectors, chosen and fixed once and for all. 
For each element~$K \in \Th$ and each face~$f \in {\mathcal{F}}_h$, let~$\hK$ and~$h_f$ be their corresponding diameters.   

We make the following assumptions on the mesh family~$\{\Th\}_{h > 0}$.
\begin{assumption}[Mesh shape-regularity]
\label{asm:mesh}
We assume that $\{\Th\}_{h > 0}$ is uniformly shape-regular, in the sense that the chunkiness parameter is uniformly bounded for all elements in the mesh family; see, e.g., \cite[Def.~4.2.16]{brenner2008mathematical}. 
\end{assumption}

\begin{assumption}[Mesh quasi-uniformity]
\label{asm:mesh:2}
We assume that~$\{\Th\}_{h > 0}$ is quasi-uniform, i.e., there is a constant~$C_{\mathrm{qu}} > 0$ independent of~$h$ such that
\begin{equation*}
h \le C_{\mathrm{qu}} h_{\min},
\end{equation*}
where~$h_{\min}$ and~$h$ are the minimum and maximum element diameters of the mesh~$\Th$, respectively. 
\end{assumption}

Given a polynomial degree~$k \in \N$ with~$k \geq 1$, we denote by~$\Pp{k}{\Th}$ the space of piecewise polynomials of degree~$k$ defined on~$\Th$, and by~$\Vgrad = \Pp{k+1}{\Th} \cap H^1(\Omega)$.
Moreover, we define
\begin{equation*}
\bVgrad := \Big\{\phi_h \in \Vgrad \ : \ (\phi_h, 1)_{\Omega} = 0 \Big\}.
\end{equation*}
Furthermore, we denote by~$\Vcurl \subset \H(\curl; \Omega)$ and~$\Vdiv \subset \H(\ddiv; \Omega)$ the corresponding N\'ed\'elec space of the second kind~\cite[\S1.2]{Nedelec:1980} and the Brezzi--Douglas--Marini (BDM) space \cite[\S2]{Brezzi_Douglas_Duran_Fortin:1987}, respectively, both of degree $k$. 

For all~$\varepsilon > 0$ and any~$k \in \N$ with~$k \geq 1$, we define the standard interpolation operators:
$\Ihcurl{k} : \H^{1 + \varepsilon}(\Omega) \to \Vcurl$, $\Ihdiv{k} : \H^{{\frac12} + \varepsilon}(\Omega) \to \Vdiv$, and~$\Ihgrad{k+1} : H^{3/2 + \varepsilon}(\Omega) \to \Vgrad$ 
(see, e.g., \cite[\S2.5]{Boffi_Brezzi_Fortin:2013} or~\cite[\S2.4]{ErnGuermond:2017}). Note that the required regularities can be significantly weakened by using more advanced approximation operators, following, e.g.,~\cite[Ch. 17]{Ern_Guermond-I:2020}, or more recent literature (see~\cite{chaumont2024stable} and references therein). 

In the next lemma, we recall the important commutativity properties of these operators (see, e.g., \cite[\S2.5.6]{Boffi_Brezzi_Fortin:2013}).
\begin{lemma}[Commutativity of the interpolation operators]
\label{lemma:commutativity-interpolants}
For any~$k \in \N$ with~$k \geq 1$, the following identities hold:
\begin{alignat*}{3}
\Ihcurl{k} (\nabla \phi) & = \nabla \Ihgrad{k+1} \phi & &  \qquad \forall \phi \in H^{2 + \varepsilon}(\Omega), \\
\Ihdiv{k} (\curl \v) & = \curl \Ihcurl{k} \v & & 
\qquad \forall \v \in \H^{1 + \varepsilon}(\Omega) \text{ with } \curl \v \in \H^{1/2 + \varepsilon}(\Omega).
\end{alignat*}
The regularity requirements can be significantly weakened (see the observation above).
\end{lemma}

\section{Finite element method with~\texorpdfstring{$\uh \in \Vcurl$}{uh in Vh(curl)} and~\texorpdfstring{$\Bh \in \Vcurl$}{Bh in Vh(curl)}}
\label{sec:method1}
We consider a semidiscrete-in-space formulation for the MHD system~\eqref{eq:MHD-system-curl}, in which both the fluid velocity~$\u$ and the magnetic field~$\B$ are approximated in~$\Vcurl$.

For sufficiently regular functions, we define the following forms:
$$
\begin{aligned}
& a(\u, \v) := (\curl \u, \curl \v)_{\Omega} \, , \quad 
c(\w; \u, \v) := ((\curl \w) \times \u , \v)_{\Omega} \, , \quad 
b(\v, q) := (\v, \nabla q)_{\Omega}\, , \\
& \dh(\w,\v) :=-  \sum_{f\in\Fhb}  
((\curl \w)\times \nOmega, \v)_{f} 
- \sum_{f\in \Fhb}  
((\curl \v)\times \nOmega, \w)_{f}
+ \gammaone \sum_{f\in\Fhb} 
h_f^{-1} (\w\times\nOmega,\v\times\nOmega)_f\, , \\
& \sh(\w;\u,\v) :=  
\sum_{f\in\Fho} h_f^{-1} \gamma (\w_{|_f}) ( \jump{\u} , \jump{\v} )_f\, ,
\end{aligned}
$$
where $\jump{\cdot}$ in the Continuous Interior Penalty (CIP) stabilization term denotes the standard jump operator, $\gammaone \in {\mathbb R}$ is a positive parameters independent of~$h$, and 
\begin{equation}
\label{eq:def-gamma-3}
\gamma(\w_{|_f}) := \max \{C_S, \Norm{\w}{\L^{\infty}(f)}\},
\end{equation}
for some ``safeguard" positive constant~$C_S$ independent of~$h$. The form~$c(\cdot;\cdot,\cdot)$ is skew-symmetric with respect to its last two arguments.
The semidiscrete problem reads as follows: for all~$t \in (0, T]$, find~$(\uh(\cdot, t), \ph(\cdot, t), \Bh(\cdot, t)) \in \Vcurl \times \bVgrad \times \Vcurl$, with $\uh$ and $\Bh$ differentiable in time, such that
\begin{subequations}
\label{eq:curl-curl-semidiscrete}
\begin{alignat}{3}
\nonumber
(\dpt \uh, \vh)_{\Omega} + \nS a(\uh, \vh) + c(\uh;\uh,\vh) - c(\Bh; \Bh, \vh)  \\
\label{eq:curl-curl-semidiscrete-a}
+ \nS d_h(\uh, \vh) - b(\vh, \ph) + \pus s_h(\uh;\uh,\vh) & = (\Ihcurl{k} (\f), \vh)_{\Omega} & & \quad \forall \vh \in \Vcurl, \\
\label{eq:curl-curl-semidiscrete-b}
b(\uh, \qh) & = 0 & & \quad \forall \qh \in \bVgrad, \\
\label{eq:curl-curl-semidiscrete-c}
(\dpt \Bh, \Ch)_{\Omega} + \nM a(\Bh, \Ch) + c(\Ch; \Bh, \uh) & = 0 & & \quad \forall \Ch \in \Vcurl, \\
\label{eq:curl-curl-semidiscrete-d}
\uh(\cdot, 0) = \Picurl \u_0 \quad \text{ and } \quad \Bh(\cdot, 0) = \Picurl \B_0,&  
\end{alignat}
\end{subequations}
where $\pus$ in ${\mathbb R}$ is a positive parameter
and~$\Picurl: \L^2(\Omega) \to \Vcurl$ denotes the~$\L^2(\Omega)$-orthogonal projection into~$\Vcurl$.

\begin{remark}[Role of the additional terms in~\eqref{eq:curl-curl-semidiscrete}]
The form $\sh(\cdot;\cdot,\cdot)$ is introduced in order to stabilize the scheme for large values of the fluid Reynold number (i.e., in our scaled model, when $0 <\nS \ll 1$). The form $d_h(\cdot,\cdot)$ is introduced to enforce weakly the tangential boundary condition on $\uh$ following a Nitsche-type approach. The motivation for such a weak imposition is that, as shown in~\cite{BoonTonnonZampa26}, imposing this condition directly in~$\Vcurl$ can lead to an ill-posed problem. The normal boundary conditions on~$\uh$ and~$\Bh$ and the tangential boundary condition on~$\curl \Bh$ are satisfied weakly. Finally, the interpolation operator acting on the loading term $\f$ is 
critical for achieving pressure robustness; see Remark~\ref{rem:pressure-robustness} below. 
\eremk
\end{remark}

We now define the following discrete version of the kernel space in~\eqref{eq:def:Z}:
\begin{equation*}
\Xh := \big\{\vh \in \Vcurl \, : \, (\vh, \nabla \qh)_{\Omega} = 0\ 
\text{ for all~$\qh \in \bVgrad$}\big\} .
\end{equation*}
Due to~\eqref{eq:curl-curl-semidiscrete-b}, it is immediate 
that $\uh \in \Xh$ at all times. Moreover, 
since~$\nabla \bVgrad \subset \Vcurl$, it holds
$$
\v \in \Z  \quad \Longrightarrow \quad \Picurl \v \in \Xh \, . 
$$
As a consequence of the above observations, we can restrict the discretization space for~$\uh$ to~$\Xh$ and eliminate the unknown~$\ph$, so that the semidiscrete-in-space formulation~\eqref{eq:curl-curl-semidiscrete} reduces to the following first-order-in-time system of equations: for all~$t \in (0, T]$, find~$(\uh(\cdot, t),\ \Bh(\cdot, t)) \in \Xh \times \Vcurl$, differentiable in time, such that
\begin{subequations}
\label{eq:kernel-semidiscrete}
\begin{alignat}{4}
\label{eq:kernel-semidiscrete-1}
(\dpt \uh, \vh)_{\Omega} + \nS a(\uh, \vh) + c(\uh;\uh,\vh) - c(\Bh; \Bh, \vh) & & & \nonumber \\
+  \nS d_h(\uh, \vh) + \pus s_h(\uh; \uh, \vh) & = (\Ihcurl{k} (\f), \vh)_{\Omega} & & \quad \forall \vh \in \Xh, \\
\label{eq:kernel-semidiscrete-2}
(\dpt \Bh, \Ch)_{\Omega} + \nM a(\Bh, \Ch) + c(\Ch; \Bh, \uh) & = 0 & & \quad \forall \Ch \in \Vcurl, \\
\label{eq:kernel-semidiscrete-3}
\uh(\cdot, 0) = \Picurl \u_0(\cdot) \quad \text{ and } \quad \Bh(\cdot, 0) = \Picurl \B_0(\cdot) && & \quad \text{in~$\Omega$} .
\end{alignat}
\end{subequations}

\begin{remark}[Discrete divergence-free property]
\label{rem:divergence-free-method-I}
By choosing test functions~$\Ch = \nabla \varphi_h$ with~$\varphi_h \in \bVgrad$ in 
equation~\eqref{eq:curl-curl-semidiscrete-c}, we trivially obtain~$(\dpt \Bh, \nabla \varphi_h)_{\Omega}=0$ for all $\varphi_h \in \bVgrad$. This property, combined with $\Picurl \B_0 \in \Xh$ (which follows from $\B_0 \in \Z$), implies that also $\Bh \in \Xh$ at all times. As a consequence, we could also restrict \eqref{eq:kernel-semidiscrete-2} to $\Xh$. However, the error analysis in Section~\ref{sec:a-priori-curl-curl} below requires the use of discrete test functions that do not necessarily belong to $\Xh$. 
For this reason, we prefer to write the method in the form~\eqref{eq:kernel-semidiscrete}. 
\eremk
\end{remark}

\begin{remark}[Other boundary conditions and cross-helicity conservation]
If the term~$d_h(\u_h, \v_h)$ in~\eqref{eq:kernel-semidiscrete} is omitted, the method imposes nonstandard slip boundary conditions: 
\begin{equation*}
    \u \cdot \nOmega = 0, 
    \qquad 
    \nOmega \times \curl \u = 0.
\end{equation*}
These boundary conditions were adopted, e.g., in~\cite{Girault88, Girault90, DiPietroDroniuQian24, BdVDassiDiPietroDroniou22}, and their connection
with the standard slip boundary conditions is discussed in \cite{MitreaMonniaux2009,BoonHiptmairTonnonZampa24}. 
They are particularly noteworthy because, in the inviscid and unforced case, namely, when $\f = 0$ and~$\nu_S = \nu_M = 0$, if the stabilization term $s_h(\uh; \uh, \vh)$ is omitted, one can take $\vh = \Bh$ and $\Ch = \uh$ in \eqref{eq:kernel-semidiscrete} and obtain the following discrete conservation of the cross-helicity:
\begin{equation*}
    \frac{\mathrm{d}}{\mathrm{d}t} (\uh, \Bh)_{\Omega} = 0.
\end{equation*}
Other finite element methods that preserve this property are discussed in \cite{HuLeeXu21,GawlikGB22,BdVHuMascotto25}.
\eremk
\end{remark}

\begin{remark}[Pressure robustness]
\label{rem:pressure-robustness}
Let the loading term~$\f$ be of the form $\f = \nabla \phi$, where~$\phi$ is a scalar function of regularity depending on the specific interpolant~$\Ihcurl{k}$ adopted in \eqref{eq:kernel-semidiscrete}; see Lemma \ref{lemma:commutativity-interpolants} and the text above it.
Due to the first commutativity property in Lemma~\ref{lemma:commutativity-interpolants}, 
the term on the right-hand side of equation~\eqref{eq:kernel-semidiscrete-1} reduces to
\begin{equation*}
(\Ihcurl{k} (\f), \vh)_{\Omega} = (\Ihcurl{k} (\nabla \phi), \vh)_{\Omega}
= (\nabla \Ihgrad{k+1} \phi, \vh) = 0 \quad  \forall \vh \in \Xh.
\end{equation*}
Therefore, gradient perturbations do not affect the approximation of the velocity field~$\u$. 
\eremk
\end{remark}

\subsection{Well-posedness}
Without loss of generality and to simplify the presentation, in the following analysis, we will assume that the positive parameter~$\pus$ is set equal to $1$.
For any given~$\w$ in $\L^\infty(\Omega)$, 
we define the following discrete seminorms for sufficiently regular vector functions~$\v$ defined in $\Omega$:
\begin{align}
\label{eq:curl-char-norm}
\Norm{\v}{\dn}^2 &\coloneq \Norm{\curl \v}{\L^2(\Omega)}^2 + \sum_{f\in\Fhb}h_f^{-1} \Norm{\v\times \nOmega}{\L^2(f)}^2,\\
\label{eq:w-seminorm}
\semiNorm{\v}{\w}^2 & \coloneqq 
\sum_{f\in\Fho} 
h_f^{-1} \gamma(\w_{|_f}) \Norm{\jump{\v}}{\L^2(f)}^2 \, .
\end{align}

\begin{lemma}[Coercivity in~$\Norm{\cdot}\dn$]
\label{lemma:coercivity-curl-uh}
Let Assumption~\ref{asm:mesh} on~$\Th$ hold and let~$\w \in \L^{\infty}(\Omega)$ be given. For~$\gammaone$ sufficiently large, there is a positive constant~$\beta$ independent of~$h$, $\nS$, and~$\nM$ such that
\begin{equation}\label{eq:X1}
a(\vh,\vh) + 
\dh(\vh,\vh) 
\geq \beta 
\Norm{\vh}{\dn}^2
\qquad \forall \vh \in \Xh.
\end{equation}
\end{lemma}
\begin{proof}
By the Young inequality with parameter~$\varepsilon > 0$, and the definition of the form~$\dh(\cdot, \cdot)$,
we have
\begin{alignat}{3}
\nonumber
a(\vh,\vh) & + \dh(\vh,\vh) \\
\nonumber
& = \Norm{\curl \vh}{\L^2(\Omega)}^2 +
2 \sum_{f \in \Fhb} (\curl \vh , \vh \times \nOmega)_{f} 
+ \gammaone  \sum_{f \in \Fhb} h_f^{-1} \Norm{\vh \times \nOmega}{\L^2(f)}^2 \\
\label{eq:aux-coer-stab}
& \geq \Norm{\curl \vh}{\L^2(\Omega)}^2 - \varepsilon  \sum_{f \in \Fhb} h_f \Norm{\curl \vh}{\L^2(f)}^2 + \Big(\gammaone - \frac{1}{\varepsilon}\Big)  \sum_{f \in \Fhb} h_f^{-1} \Norm{\vh \times \nOmega}{\L^2(f)}^2.
\end{alignat}
We estimate the second term on the right-hand side of~\eqref{eq:aux-coer-stab}. To do so, 
we use the standard trace-inverse inequality for polynomials (see, e.g., \cite[Lemma 12.8]{Ern_Guermond-I:2020}) with constant~$C_{\mathrm{inv}}$ independent of~$h$, to obtain 
\begin{alignat}{3}
\label{eq:inverse-Poincare}
\varepsilon \sum_{f \in \Fhb} h_f \Norm{\curl \vh}{\L^2(f)}^2 \le \varepsilon  C_{\mathrm{inv}}\Norm{\curl \vh}{\L^2(\Omega)}^2.
\end{alignat}
Then, the coercivity bound~\eqref{eq:X1} follows by combining~\eqref{eq:inverse-Poincare} with~\eqref{eq:aux-coer-stab}, and taking~$\varepsilon$ sufficiently small and~$\gammaone$ sufficiently large, both depending only on~$C_{\mathrm{inv}}$.
\end{proof}
\begin{theorem}[Well-posedness]\label{theo:well}
Let Assumption \ref{asm:mesh} on~$\Th$ hold.
Assume that~$\f \in C^0([0, T]; {\cal S})$ with the space ${\cal S}$ sufficiently regular for $\Ihcurl{k}(\f)$ to be well defined; see Remark \ref{rem:f-interp} below.
If~$\gammaone$ is sufficiently large as in Lemma~\ref{lemma:coercivity-curl-uh}, there exists a unique solution~$(\uh, \Bh) \in C^1([0, T]; \Xh) \times C^1([0, T]; \Xh)$ to the semidiscrete-in-space formulation~\eqref{eq:kernel-semidiscrete}. Moreover, such a unique solution satisfies the following continuous dependence on the data:
\begin{alignat}{3}
\nonumber
\frac{1}{2} & \big(\Norm{\uh}{L^\infty(0, T; \L^2(\Omega))}^2  + \Norm{\Bh}{L^\infty(0, T; \L^2(\Omega))}^2\big) \\
\nonumber
& + 2 \beta \nS \int_0^T \Norm{\uh(\cdot, t)}{\dn}^2 \dt 
+ 2 \nM \int_0^T \Norm{\curl \Bh(\cdot, t)}{\L^2(\Omega)}^2 \dt + 2 \int_0^T \semiNorm{\uh(\cdot, t)}{\uh}^2 \dt  \\
\label{eq:stab-semidiscrete}
 &  \qquad \le  \Norm{\u_0}{{\L}^2(\Omega)}^2 + \Norm{\B_0}{{\L}^2(\Omega)}^2  + 2 \Norm{\Ihcurl{k}(\f)}{L^1(0, T; {\L}^2(\Omega))}^2  .
\end{alignat}
\end{theorem}
\begin{proof}
The semidiscrete-in-space formulation~\eqref{eq:kernel-semidiscrete} is
a first-order-in-time Cauchy problem in a finite-dimensional space with continuous {(locally Lipschitz) nonlinear} coefficients. By the Picard--Lindel\"of theorem, the local Lipschitz continuity of the nonlinear functional defining the Cauchy problem implies the existence and uniqueness of a solution~$(\uh, \Bh) \in C^1([0, t^*]; \Xh) \times C^1([0, t^*]; \Xh)$ to~\eqref{eq:kernel-semidiscrete}, for some 
time~$t^* \in (0, T]$. To conclude global existence up to the final time~$T$, it only remains to show that solutions to~\eqref{eq:kernel-semidiscrete} cannot blow up in finite time. Taking~$\vh = \uh$ and~$\Ch = \Bh$ in~\eqref{eq:kernel-semidiscrete-1} and~\eqref{eq:kernel-semidiscrete-2}, respectively, integrating in time until~$t \in (0, T]$, summing both equations, using the H\"older inequality, and recalling the coercivity bound in~\eqref{eq:X1}, we get 
\begin{alignat}{3}
\nonumber
\frac12   \Big( & \Norm{\uh(\cdot, t)}{\L^2(\Omega)}^2 + \Norm{\Bh(\cdot, t)}{\L^2(\Omega)}^2\Big) \\
\nonumber
& 
 + \beta \nS \int_0^t 
\Norm{\uh(\cdot, s)}\dn^2 \ds + \nM \int_0^t \Norm{\curl \Bh(\cdot, s)}{{\L}^2(\Omega)}^2 \ds + \int_0^t \semiNorm{\uh(\cdot, s)}{\uh}^2 \ds \\
\label{eq:aux-stab}
& \qquad \le \frac12 \Big(\Norm{\Picurl \u_0}{\L^2(\Omega)}^2 + \Norm{\Picurl \B_0}{\L^2(\Omega)}^2 \Big) + \Norm{\Ihcurl{k}(\f)}{L^1(0, T; \L^2(\Omega))} 
\Norm{\uh}{L^{\infty}(0, t; \L^2(\Omega))} ,
\end{alignat}
where we also used the skew-symmetry property of~$c(\cdot; \cdot, \cdot)$
and the identity~$s_h(\uh;\uh,\uh)=\semiNorm{\uh}{\uh}^2$, with~$\semiNorm{\cdot}{\uh}$ as in~\eqref{eq:w-seminorm}.
Since~$t$ is arbitrary, it is easy to check that the solution to~\eqref{eq:kernel-semidiscrete} cannot blow up in finite time and thus it can be extended up to~$T$.
Applying the Young inequality in the last term on the right-hand side of~\eqref{eq:aux-stab} with $t=T$ and taking the supremum for $t \in [0,T]$ on the left-hand side, with some simple manipulations, we obtain
\begin{equation*}
\begin{split}
\frac14  & \big(\Norm{\uh}{L^\infty(0, T; \L^2(\Omega))}^2 + \Norm{\Bh}{L^\infty(0, T; \L^2(\Omega))}^2\big) \\
& + \beta \nS \int_0^T \Norm{\uh(\cdot, t)}\dn^2 \dt + \nM \int_0^T \Norm{\curl \Bh(\cdot, t)}{\L^2(\Omega)}^2\dt 
 + \int_0^T \semiNorm{\uh(\cdot, t)}{\uh}^2 \dt \\
 & \qquad \le \frac12 \big(\Norm{\Picurl \u_0}{\L^2(\Omega)}^2 + \Norm{\Picurl \B_0}{\L^2(\Omega)}^2 \big) 
 + \Norm{\Ihcurl{k}(\f)}{L^1(0, T; \L^2(\Omega))}^2 .
\end{split}
\end{equation*}
This, together with the continuity of the $\Picurl$ operator in the $\L^2(\Omega)$ norm and some trivial manipulations, gives the stability estimate~\eqref{eq:stab-semidiscrete}, and the proof is complete.
\end{proof}

\begin{remark}
\label{rem:pressrob-f}
In light of Remark~\ref{rem:pressure-robustness}, if the gradient component of the Helmholtz--Hodge decomposition of the loading term~$\f$ has sufficient regularity for the application of~$\Ihcurl{k}$, then~$\f$ on the right-hand side of~\eqref{eq:stab-semidiscrete} can be replaced by its Helmholtz--Hodge projection; see, e.g., \cite{John_etal:2017}.
\eremk
\end{remark}

\begin{remark}\label{rem:f-interp}
The explicit expression of the stability bound in terms of~$\f$ depends on the choice of the commuting interpolation operator~$\Ihcurl{k}$ used in~\eqref{eq:kernel-semidiscrete}. For the simpler choice based on a direct evaluation of the degrees of freedom, we have
\begin{equation*}
\Norm{\Ihcurl{k}(\f)(\cdot, t)}{\L^2(\Omega)} \le \Norm{\f(\cdot, t)}{\L^2(\Omega)} + C' h^{1 + \varepsilon} 
\semiNorm{\f(\cdot, t)}{\H^{1+\varepsilon}(\Omega)}
\end{equation*}
for all~$\varepsilon > 0$.
This bound directly reflects on \eqref{eq:stab-semidiscrete}.
By using more involved approximants (e.g., \cite{chaumont2024stable}), the regularity requirement on~$\f$ and, consequently, the norm appearing on the right-hand side in the above bound can be weakened.
\eremk
\end{remark}

We close this section by noting that there also exists a (unique) discrete pressure solution. 
\begin{corollary}[Existence of a discrete pressure]
Under the assumptions of Theorem \ref{theo:well}, there exists a unique {$\ph \in C^0([0, T]; \bVgrad)$} 
such that the triple $(\uh,\ph,\Bh)$ solves problem~\eqref{eq:curl-curl-semidiscrete}.
\end{corollary}
\begin{proof}
The following inf--sup condition readily follows from~$\nabla \bVgrad \subseteq \Vcurl$ and the Poincar\'e--Wirtinger inequality: there exists a positive constant~$C$ independent of~$h$ (and of~$\nS$ and~$\nM$) such that
$$
\sup_{\mathbf{0}\ne \vh \in \Vcurl} \frac{b(\vh,q_h)}{\Norm{\vh}{\H(\curl; \Omega)}} 
\ge C
\Norm{q_h}{H^1(\Omega)} \qquad \forall \, q_h \in \bVgrad.
$$
Therefore, recalling the classical theory of mixed problems, equation \eqref{eq:kernel-semidiscrete-1} implies that, at every time~$t$, the problem of finding $p_h \in \bVgrad$ such that
$$
\begin{aligned}
b(\vh, \ph) = & 
(\dpt \uh, \vh)_{\Omega} + \nS a(\uh, \vh) + c(\uh;\uh,\vh) - c(\Bh; \Bh, \vh)  
\\
& + \nS d_h(\uh, \vh)  + s_h(\uh;\uh,\vh) - (\Ihcurl{k} (\f), \vh)_{\Omega} 
& \qquad \forall \vh \in \Vcurl
\end{aligned}
$$
has a unique solution.
\end{proof}

\section{Convergence analysis}\label{sec:4}
This section is devoted to deriving \emph{a priori} error estimates for the semidiscrete-in-space formulation~\eqref{eq:kernel-semidiscrete} which are robust for small values of the fluid diffusivity parameter (i.e., for~$0 < \nS \ll 1$). 

Henceforth, we denote by~$\Id$ the identity operator. Moreover, we use~$a \lesssim b$ to denote the existence of a positive constant~$C$ such that~$a \le C b$, where~$C$  
depends only on the mesh regularity parameter, and the method parameters $\alpha$ and~$C_S$, and is, in particular, independent of~$h$, $\nS$, and $\nM$. 

\subsection{Properties of the~\texorpdfstring{$L^2$}{L2} projection operator~\texorpdfstring{$\Picurl$}{Pi-curl}}
We now present some stability and approximation properties of the $\L^2(\Omega)$ projection operator~$\Picurl$ onto $\Vcurl$. 
\begin{lemma}[Estimates in $\L^{\infty}(\Omega)$ for $\Picurl$]
\label{lemma:picurl_linfty_stab}
Under the mesh conditions in Assumptions \ref{asm:mesh} and \ref{asm:mesh:2}, there hold
\begin{subequations}
\begin{alignat}{3}
\label{eq:stab-Picurl-infty}
    \Norm{ \Picurl \v}{\L^{\infty}(\Omega)}  & \lesssim \Norm{\v }{\L^{\infty}(\Omega)} &  & \qquad \forall \v \in \L^{\infty}(\Omega), \\
            \label{eq:Picurl-approx-Linfty}
        \Norm{(\Id - \Picurl) \v}{\L^{\infty}(\Omega)} & \lesssim h^r \semiNorm{\v}{\WW^{r, \infty}(\Omega)} & & \qquad \forall \v \in {\WW}^{r, \infty}(\Omega), \ 0 \le r \le k +1, \\
\label{eq:stab-curl-Picurl-infty}
\Norm{\curl \Picurl \v}{\L^{\infty}(\Omega)} & \lesssim \semiNorm{\v}{\WW^{1,\infty}(\Omega)} & & \qquad \forall \v \in \WW^{1, \infty}(\Omega).
\end{alignat}
\end{subequations}
\end{lemma}
\begin{proof} 
As the discrete space~$\Vcurl$ satisfies the hypotheses in~\cite{douglas1974stability},
the stability bound~\eqref{eq:stab-Picurl-infty} follows from the main theorem therein. 
To prove~\eqref{eq:Picurl-approx-Linfty}, we observe that, for any~$\vh \in \Vcurl$, we have
\begin{align*}
\Norm{(\Id - \Picurl) \v}{\L^{\infty}(\Omega)}&\le \Norm{\v-\vh}{\L^{\infty}(\Omega)}+\Norm{\vh - \Picurl \v}{\L^{\infty}(\Omega)}\\
&=\Norm{\v-\vh}{\L^{\infty}(\Omega)}+\Norm{\Picurl(\vh - \v)}{\L^{\infty}(\Omega)}\lesssim \Norm{\v-\vh}{\L^{\infty}(\Omega)},
\end{align*}
where, in the last step, we have used~\eqref{eq:stab-Picurl-infty}. Then, estimate~\eqref{eq:Picurl-approx-Linfty} follows from~\cite[Cor.~5.4]{ErnGuermond:2017}.
As for~\eqref{eq:stab-curl-Picurl-infty}, we use the triangle inequality, a polynomial inverse estimate, 
the well-known stability of~$\Ihgrad{1}$ in~$\WW^{1, \infty}(\Omega)$ and its approximation properties, and~\eqref{eq:Picurl-approx-Linfty} with~$r=1$, to deduce
\begin{align*}
\Norm{\curl \Picurl \v}{\L^{\infty}(\Omega)} & \leq \Norm{\curl (\Picurl \v-\Ihgrad{1}\v) }{\L^{\infty}(\Omega)} + \Norm{\curl \Ihgrad{1}\v}{\L^{\infty}(\Omega)} \\
& \lesssim h^{-1} \Norm{\Picurl \v - \Ihgrad{1}\v}{\L^{\infty}(\Omega)} + \Norm{\v}{\WW^{1,\infty}(\Omega)}
\\
&\leq h^{-1} \Norm{\Picurl \v - \v}{\L^{\infty}(\Omega)} + h^{-1}\Norm{\v - \Ihgrad{1}\v}{\L^{\infty}(\Omega)} + \Norm{\v}{\WW^{1,\infty}(\Omega)}\\
& \lesssim \Norm{\v}{\WW^{1,\infty}(\Omega)},
\end{align*}
thus obtaining~\eqref{eq:stab-curl-Picurl-infty}.
\end{proof}

\begin{lemma}[Further estimates for~$\Picurl$]
\label{lemma:approx-Picurl}
Under the mesh conditions in Assumptions~\ref{asm:mesh} and \ref{asm:mesh:2}, the following estimates hold for any~$r \in [1, k + 1]$: 
\begin{subequations}
    \begin{alignat}{3}
        \label{eq:Picurl-approx-L2}
        \Norm{(\Id - \Picurl) \v}{\L^2(\Omega)} & \lesssim h^r \semiNorm{\v}{\H^{r}(\Omega)} & & \quad \forall \v \in \H^{r}(\Omega),      \\
        \label{eq:Picurl-approx-curl}
        \Norm{\curl (\Id - \Picurl) \v}{\L^2(\Omega)} & \lesssim h^{r-1} \semiNorm{\v}{\H^{r}(\Omega)} & & \quad \forall \v \in \H^{r}(\Omega), \\
        \label{eq:Picurl-approx-curl-char}
       \Norm{(\Id - \Picurl)\v}{\dn} & \lesssim h^{r-1} \semiNorm{\v}{\H^{r}(\Omega)} & & \quad \forall \v \in \H^{r}(\Omega).
    \end{alignat}
\end{subequations}
\end{lemma}
\begin{proof}
Let~$\v \in \H^{r}(\Omega)$. Estimate~\eqref{eq:Picurl-approx-L2} follows from the approximation properties of the space~$\Vcurl$ (see, e.g., \cite[Cor.~5.3]{ErnGuermond:2017}). 
To prove~\eqref{eq:Picurl-approx-curl}, we employ standard polynomial inverse estimates, the quasi-uniformity of~$\Th$, and the stability properties of~$\Picurl$ to deduce the following bound for any~$\vh \in \Vcurl$:
\begin{alignat*}{3}
\Norm{\curl (\Id - \Picurl) \v}{\L^2(\Omega)} & \le \Norm{\curl (\v - \vh)}{\L^2(\Omega)} + \Norm{\curl (\vh - \Picurl \v)}{\L^2(\Omega)} \\
& \lesssim \Norm{\curl(\v - \vh)}{\L^2(\Omega)} + h^{-1} \Norm{\Picurl(\vh - \v)}{\L^2(\Omega)} \\
& \lesssim \Norm{\curl(\v - \vh)}{\L^2(\Omega)} + h^{-1} \Norm{\v - \vh}{\L^2(\Omega)},
\end{alignat*}
which, combined with the approximation properties of~$\Vcurl$, gives~\eqref{eq:Picurl-approx-curl}.

For any~$f \in \Fhb$, let~$K_f \in \Th$ be the only element such that~$f \subset \partial K_f$. By the definition in~\eqref{eq:curl-char-norm} of~$\Norm{\cdot}\dn$ and the fact that~$\nOmega$ is unitary, we have
\begin{equation}
\label{eq:aux-Picurl-char}
\Norm{(\Id - \Picurl) \v}\dn^2 \le \Norm{\curl(\Id - \Picurl)\v}{\L^2(\Omega)}^2 + \sum_{f \in \Fhb} h_f^{-1} \Norm{(\Id - \Picurl) \v}{\L^2(f)}^2.
\end{equation}
Therefore, it only remains to estimate the second term on the right-hand side of~\eqref{eq:aux-Picurl-char}. Using the trace inequality for continuous functions (see~\cite[Lemma~12.15]{Ern_Guermond-I:2020}) and proceeding as for~\eqref{eq:Picurl-approx-curl}, we obtain
\begin{alignat*}{3}
\sum_{f \in \Fhb} h_f^{-1} \Norm{(\Id - \Picurl) \v}{\L^2(f)}^2 
& \lesssim \sum_{f \in \Fhb} h_f^{-1} \big(h_{K_f}^{-1} \Norm{(\Id - \Picurl)\v}{\L^2(K_f)}^2 + h_{K_f} \semiNorm{(\Id - \Picurl)\v}{\H^1(K_f)}^2 \big)  \\
& \lesssim \sum_{f \in \Fhb} h_f^{-1} h_{K_f}^{2r - 1} \semiNorm{\v}{\H^{r}(K_f)}^2 \lesssim h^{2(r-1)} \semiNorm{\v}{\H^{r}(\Omega)}^2,
\end{alignat*}
which completes the proof of~\eqref{eq:Picurl-approx-curl-char}.
\end{proof}

\subsection{The projection operators~\texorpdfstring{$\Jhav$}{Jh-av-grad}, \texorpdfstring{$\Jav$}{Jh-av-curl}, \texorpdfstring{$\Jhdiv$}{Jh-div},  and~\texorpdfstring{$\Jhcurl$}{Jh-curl} and their properties}
In the error analysis, we also use the projection operators~$\Jhav : \Pp{k + 1}{\Th} \to \Vgrad$, 
$\Jav: \Pp{k}{\Th} \to \Vcurl $, 
$\Jhdiv : \L^1(\Omega) \to \Vdiv$, and~$\Jhcurl : \L^1(\Omega) \to \Vcurl$ defined in~\cite[\S4.2 and~\S5]{ErnGuermond:2017}, for which we recall the following properties.
\begin{lemma}[Properties of~$\Jhav$ and $\Jav$]
\label{lemma:Jhav}
Let Assumption~\ref{asm:mesh} on~$\Th$ be satisfied. Then, for all~$\phi_h \in \Pp{k+1}{\Th}$ and~$\v_h \in \Pp{k}{\Th}$, there hold
\begin{subequations}
\begin{alignat}{3}
\sum_{K \in \Th} \semiNorm{(\Id - \Jhav) \phi_h}{H^1(K)}^2 & \lesssim \sum_{f \in \Fho}  h_f^{-1} \Norm{\jump{\phi_h}}{L^2(f)}^2 
\label{void}
\, , \\
\lVert \v_h - \Jav{\v_h}\rVert_{\L^2(\Omega)}^2 & \lesssim \sum_{f\in\Fho} h_f \lVert \jump{\v_h} \times \nf\rVert_{\L^2(f)}^2 \, .
\label{eq:Jav_prop}
\end{alignat}
\end{subequations}
\end{lemma}

\begin{lemma}[Properties of~$\Jhdiv$]
\label{lemma:properties-Jhdiv}
Let Assumption~\ref{asm:mesh} on~$\Th$ be satisfied. Then, for any~$q \in [1, \infty]$, 
\begin{alignat*}{3}
\Norm{\Jhdiv \v}{\L^q(\Omega)} & \lesssim \Norm{\v}{\L^q(\Omega)} & & \qquad \forall \v \in \L^q(\Omega), \\
\curl (\Jhcurl \v) & = \Jhdiv(\curl \v) & & \qquad \forall \v \in \L^q(\Omega) \text{ with } \curl \v \in \L^q(\Omega). 
\end{alignat*}
\end{lemma}

\begin{lemma}[Properties of~$\Jhcurl$]
\label{lemma:properties-Jhcurl}
Let Assumption~\ref{asm:mesh} on~$\Th$ be satisfied.  Then, for any~$q \in [1, \infty]$,
$r \in [0, k + 1]$, and $s \in [0,k]$, there hold
\begin{alignat*}{3}
\Norm{\Jhcurl \v}{\L^q(\Omega)} & \lesssim \Norm{\v}{\L^q(\Omega)} & & \qquad \forall \v \in \L^q(\Omega),\\
\Norm{(\Id - \Jhcurl) \v}{\L^q(\Omega)} & \lesssim \inf_{\vh \in \Vcurl} \Norm{\v - \vh}{\L^q(\Omega)} \lesssim h^r \semiNorm{\v}{\WW^{r,q}(\Omega)}  & & \qquad \forall \v \in 
\WW^{r, q}(\Omega),\\ 
\nonumber
\Norm{\curl (\Id - \Jhcurl) \v}{\L^q(\Omega)} & \lesssim \!\!\! \inf_{\w_h \in \Vdiv} \! \Norm{\curl \v - \w_h}{\L^q(\Omega)} \\
& \lesssim h^s \semiNorm{\curl\v}{{\bf W}^{s,q}(\Omega)}
 & &  \hspace{-1in}  \forall \v \in \H(\curl; \Omega) \text{ with } \curl \v \in \WW^{s, q}(\Omega).
\end{alignat*}
\end{lemma}
The reason why we do not restrict equation~\eqref{eq:kernel-semidiscrete-2} to the discrete kernel space~$\Xh$ is that~$\Jhcurl \B$ does not necessarily belong to~$\Xh$, even if~$\B \in \Z$.

\subsection{\emph{A priori} error estimates}
\label{sec:a-priori-curl-curl}
We are now in a position to derive \emph{a priori} error estimates for the semidiscrete-in-space formulation~\eqref{eq:kernel-semidiscrete}. 
In this section, we denote by~$(\u, \pressure, \B)$ a continuous weak solution to~\eqref{eq:MHD-system} in the sense of~\cite[Def.~2.1]{Prohl08}.
In order to guarantee the consistency of the scheme and be able to write the error equation, we make the following regularity assumptions. 

\begin{assumption}[Regularity of the weak solution]
\label{asm:regularity}
We assume that a continuous weak solution~$(\u, \pressure, \B)$ to~\eqref{eq:MHD-system} (in the sense of~\cite[Def. 2.1]{Prohl08}) satisfies
\begin{alignat*}{3}
\u & \in 
H^1(0, T; \L^2(\Omega)) \cap L^2(0, T; \H^{3/2 + \varepsilon}(\Omega)), & \qquad \pressure & \in L^2(0, T; H^1(\Omega)), \\
\B & \in H^1(0, T; \L^2(\Omega)) \cap L^2(0, T; \Z), & 
\end{alignat*}
for some~$\varepsilon > 0$. 
\end{assumption}

For convenience, we define the error functions
\begin{alignat*}{5}
    \eu & \coloneqq \u - \uh, & \qquad  \eIu & \coloneqq \u - \Picurl\u,&  \qquad \ehu & \coloneqq \uh - \Picurl\u,\\
    \eB & \coloneqq \B - \Bh, & \qquad \eIB & \coloneqq \B - \Jhcurl \B, & \qquad \ehB & \coloneqq \Bh - \Jhcurl \B.
\end{alignat*}

\begin{remark}[Discrete approximants]
We adopt different approximants for the velocity~$\u$ and the magnetic field~$\B$, 
although both approximants belong to the same discrete space $\Vcurl$. 
The motivation for using~$\Picurl \u$ for the velocity is twofold: (i) $\Picurl \u \in \Xh$, which is important to handle the incompressibility constraint in a pressure-robust way; and (ii) it guarantees the orthogonality properties needed in handling certain convection terms following a CIP approach. 
In contrast, for the magnetic field, we use~$\Jhcurl \B$. 
Although such an approximant does not guarantee~$\Jhcurl \B \in \Xh$, it enjoys commuting diagram properties that are crucial for the analysis.
\eremk
\end{remark}

\begin{proposition}[{Bound on the discrete errors}]
\label{prop:discrete-errors}
Let Assumption~\ref{asm:mesh} on the mesh family~$\{\Th\}_{h > 0}$ hold, and let~$\gammaone$ be sufficiently large as in Lemma~\ref{lemma:coercivity-curl-uh}. Let also $(\u, \pressure, \B) $ be a weak solution to~\eqref{eq:MHD-system} satisfying Assumption~\ref{asm:regularity}, and~$(\uh, \Bh) \in C^1([0, T]; \Xh) \times C^1([0, T]; \Xh)$ be the solution to~\eqref{eq:kernel-semidiscrete}. Then, the following error bound holds for a.e. $t \in (0,T)$: 
\begin{alignat*}{3}
    \frac12 \ddt \Big(\Norm{\ehu}{\L^2(\Omega)}^2 & + \Norm{\ehB}{\L^2(\Omega)}^2\Big) +  \beta\nS \Norm{\ehu}\dn^2 + \nM \Norm{\curl \ehB}{\L^2(\Omega)}^2 + \semiNorm{\ehu}{\uh}^2  \\
    & \le M_0 + M_1 + M_2 + M_3 + M_4 + M_5 + M_6 + M_7,
\end{alignat*}
where
\begin{alignat*}{3}
    M_0 & := (\dpt \eIu, \ehu)_{\Omega} + (\dpt \eIB, \ehB)_{\Omega}, & M_4 & := c(\Jhcurl \B; \Bh, \ehu) - c(\B; \B, \ehu), \\
    M_1 & := \nS a(\eIu, \ehu) + \nM a(\eIB, \ehB), & \qquad M_5 & := c(\ehB; \B, \u) - c(\ehB; \Bh, \Picurl \u), \\
    M_2 & := \nS \dh(\eIu, \ehu), & \qquad M_6 & := \sh(\uh; \eIu, \ehu),  \\
    M_3 & := c(\u; \u, \ehu) - c(\uh; \Picurl \u, \ehu), & M_7 & := (\Ihcurl{k} \f - \f - \nabla p, \ehu)_{\Omega}. 
\end{alignat*}
\end{proposition}
\begin{proof}
Due to the consistency of the semidiscrete-in-space formulation~\eqref{eq:kernel-semidiscrete} and the regularity assumptions on~$(\u, p, \B)$, for a.e. $t \in (0, T)$, we have
\begin{alignat*}{3}
\nonumber
(\dpt \u, \vh)_{\Omega} + \nS a(\u, \vh) + c(\u; \u, \vh)  - c(\B; \B, \vh) \\
 + \nS \dh(\u, \vh) + \sh(\uh; \u, \vh) & = (\f + \nabla p, \vh)_{\Omega} & & \qquad \forall \vh \in \Xh, \\
(\dpt \B, \Ch)_{\Omega} + \nM a(\B, \Ch) + c(\Ch; \B, \u) & = 0 & & \qquad \forall \Ch \in \Vcurl. 
\end{alignat*}
Consequently, the following error equations hold:
\begin{subequations}
\label{eq:error-equations}
\begin{alignat}{3}
\nonumber
(\dpt \eu, \vh)_{\Omega}  + \nS a(\eu, \vh) + \nS \dh(\eu, \vh) 
+ c(\u; \u, \vh)  - c(\uh; \uh, \vh)  \\
\nonumber
 - c(\B; \B, \vh)  + c(\Bh; \Bh, \vh)  + \sh(\uh; \eu, \vh) \\
\label{eq:error-equation-1}
  = \big(\f + \nabla p - \Ihcurl{k}(\f), \vh)_{\Omega} & \qquad \forall \vh \in \Xh, \\
\label{eq:error-equation-2}
 (\dpt \eB, \Ch)_{\Omega} + \nM a(\eB, \Ch) + c(\Ch; \B, \u) - c(\Ch; \Bh, \uh)  = 0 & \qquad \forall \Ch \in \Vcurl.
\end{alignat}
\end{subequations}
By adding and subtracting suitable terms in~\eqref{eq:error-equations}, 
we obtain
\begin{alignat*}{3}
(\dpt \ehu, & \vh)_{\Omega}  + \nS a(\ehu, \vh) + \nS \dh (\ehu, \vh) + c(\uh; \ehu, \vh) - c(\ehB; \Bh, \vh) + \sh(\uh; \ehu, \vh) \\
& = (\dpt \eIu, \vh)_{\Omega} + \nS a(\eIu, \vh) + \nS \dh(\eIu, \vh) + c(\u; \u, \vh) - c(\uh; \Picurl \u, \vh) \\
& \quad - c(\B; \B, \vh) + c(\Jhcurl \B; \Bh, \vh) + \sh(\uh; \eIu, \vh) \\
& \quad + (\Ihcurl{k} \f - \f - \nabla p, \vh)_{\Omega} & & \ \forall \vh \in \Xh,  \\
(\dpt \ehB, &\Ch)_{\Omega}  + \nM a(\ehB, \Ch) + c(\Ch; \Bh, \ehu)\\
& = (\dpt \eIB, \Ch)_{\Omega} + \nM a(\eIB, \Ch) + c(\Ch; \B, \u) - c(\Ch; \Bh, \Picurl \u) & & \ \forall \Ch \in \Vcurl.
\end{alignat*}
We now take~$\vh = \ehu$ and~$\Ch = \ehB$ and sum the resulting equations. Using 
Lemma~\ref{lemma:coercivity-curl-uh}, the identity~$s_h(\uh;\ehu,\ehu)=\semiNorm{\ehu}{\uh}^2$, with~$\semiNorm{\cdot}{\uh}$ as in~\eqref{eq:w-seminorm}, and the skew-symmetry property of~$c(\cdot; \cdot, \cdot)$, we get
\begin{alignat}{3}
\nonumber
\frac12 \ddt \Big(\Norm{\ehu}{\L^2(\Omega)}^2 & + \Norm{\ehB}{\L^2(\Omega)}^2\Big) + \nS \beta \Norm{\ehu}\dn^2 + \nM \Norm{\curl \ehB}{\L^2(\Omega)}^2 + \semiNorm{\ehu}{\uh}^2  \\
\nonumber
& \le \Big[(\dpt \eIu, \ehu)_{\Omega} + (\dpt \eIB, \ehB)_{\Omega}\Big] + \Big[\nS a(\eIu, \ehu) + \nM a(\eIB, \ehB)\Big] + \nS \dh(\eIu, \ehu)  \\
\nonumber
& \quad + \Big[c(\u; \u, \ehu) - c(\uh; \Picurl \u, \ehu)\Big]  + \Big[c(\Jhcurl \B; \Bh, \ehu) - c(\B; \B, \ehu)\Big]\\
\nonumber
& \quad + \Big[c(\ehB; \B, \u) - c(\ehB; \Bh, \Picurl \u)\Big] \\
\nonumber
& \quad  + \sh(\uh; \eIu, \ehu)  + (\Ihcurl{k} \f - \f - \nabla p, \ehu)_{\Omega} \\
\nonumber
& =: M_0 + M_1 + M_2 + M_3 + M_4 + M_5 + M_6 + M_7,
\end{alignat}
which completes the proof. 
\end{proof}

In the next lemmas, we estimate the terms~$M_1,\ldots,M_7$ separately. 
Such results  
hold for a.e.~$t \in (0,T)$; however, for the sake of brevity, we avoid recalling it at every instance.
For the same reason, in each lemma, we only specify the regularity assumptions necessary to estimate each term, and assume that all the hypotheses of Proposition~\ref{prop:discrete-errors}, as well as Assumption \ref{asm:mesh:2}, hold. 
\begin{lemma}[Estimate of~$M_0$]
\label{lemma:M0}
If~$\dpt \u  \in L^2(0, T; \H^k(\Omega))$ and~$\dpt \B \in L^2(0, T; \H^k(\Omega))$, it holds that
\begin{equation*}
M_0 
\lesssim h^{2k} \big(\semiNorm{\dpt \u}{\H^k(\Omega)}^2 + \semiNorm{\dpt \B}{\H^k(\Omega)}^2 \big) + \Norm{\ehu}{\L^2(\Omega)}^2 + \Norm{\ehB}{\L^2(\Omega)}^2.
\end{equation*}
\end{lemma}
\begin{proof}
The result immediately follows from the approximation properties in Lemmas~\ref{lemma:approx-Picurl}  and~\ref{lemma:properties-Jhcurl} of~$\Picurl$ and~$\Jhcurl$, respectively. 
\end{proof}

\begin{lemma}[Estimate of~$M_1$]
\label{lemma:M1}
If~$\u \in 
L^2(0, T; \H^{k+1}(\Omega))$ and~$\curl \B \in 
L^2(0, T; \H^{k}(\Omega))$, for any~$\delta > 0$, it holds that
\begin{equation*}
\begin{split}
M_1
& \lesssim \frac{1}{2\delta} h^{2k} \big(\nS \semiNorm{\u}{\H^{k+1}(\Omega)}^2 + \nM \semiNorm{\curl\B}{\H^{k}(\Omega)}^2 \big) + \frac{\delta}{2}\big(\nS\Norm{\ehu}{\dn}^2 + \nM \Norm{\curl \ehB}{\L^2(\Omega)}^2 \big).
\end{split}
\end{equation*}
\end{lemma}
\begin{proof}
The estimate of~$M_1$ can be obtained using the Young inequality with parameter~$\delta$, and the approximation properties in Lemmas~\ref{lemma:approx-Picurl} and~\ref{lemma:properties-Jhcurl} of~$\Picurl$ and~$\Jhcurl$, respectively.
\end{proof}

\begin{lemma}[Estimate of~$M_2$]
\label{lemma:M2}
If~$\u \in 
L^2(0, T; \H^{k+1}(\Omega))$, for any~$\delta > 0$, it holds that
\begin{alignat*}{3}
M_2 
\lesssim \frac{\nS}{\delta}  h^{2k} \semiNorm{\u}{\H^{k+1}(\Omega)}^2 + \delta \nS \Norm{\ehu}{\dn}^2.
\end{alignat*}
\end{lemma}
\begin{proof}
Using the definition of~$\dh(\cdot, \cdot)$, we write
\begin{equation}\label{eq:M2-decomposition}
\begin{aligned}
M_2  = \nS \dh(\eIu, \ehu) 
& = - \nS \sum_{f \in \Fhb} \big((\curl \eIu) \times \nOmega, \ehu \big)_{f} - \nS \sum_{f \in \Fhb} \big((\curl \ehu) \times \nOmega, \eIu \big)_{f} \\
& \quad + \gammaone \nS \sum_{f \in \Fhb} h_f^{-1} \big(\eIu \times \nOmega, \ehu \times \nOmega \big)_f = : M_2^{(1)} + M_2^{(2)} + M_3^{(3)},
\end{aligned}
\end{equation}
and estimate each term~$M_2^{(i)}$ separately. 

Using the Young inequality, the trace inequality for continuous functions, and the approximation properties of~$\Picurl$ in Lemma~\ref{lemma:approx-Picurl}, the term~$M_2^{(1)}$ can be estimated as
\begin{alignat}{3}
\nonumber
M_2^{(1)} & = \nS \sum_{f \in \Fhb} \big(\curl \eIu , \ehu \times \nOmega \big)_{f} 
\le \frac{\nS}{2\delta} \sum_{f \in \Fhb} h_f \Norm{\curl \eIu}{\L^2(f)}^2  + \frac{\delta \nS}{2} \Norm{\ehu}{\dn}^2 \\
\nonumber
& \lesssim \frac{\nS}{2 \delta} \sum_{f \in \Fhb} h_f \big(h_{K_f}^{-1} \Norm{\curl \eIu}{\L^2(K_f)}^2 + h_{K_f} \semiNorm{\curl \eIu}{\H^1(K_f)}^2  \big) + \frac{\delta \nS}{2} \Norm{\ehu}{\dn}^2 \\
\label{eq:M2-1}
& \lesssim \frac{\nS}{2\delta} h^{2k} \semiNorm{\u}{\H^{k + 1}(\Omega)}^2 + \frac{\delta \nS}{2} \Norm{\ehu}{\dn}^2.
\end{alignat}

As for~$M_2^{(2)}$, we use the Young inequality, the inverse trace inequality for polynomials, the trace inequality for continuous functions, and the approximation properties of~$\Picurl$ in Lemma~\ref{lemma:approx-Picurl} to obtain
\begin{alignat}{3}
\nonumber
M_2^{(2)}  = - \nS \sum_{f\in \Fhb} \big( (\curl \ehu) \times \nOmega, \eIu\big)_f & \le \frac{\nS}{2\delta} \sum_{f \in \Fhb} h_f^{-1} \Norm{\eIu}{\L^2(f)}^2 + \frac{\delta \nS }{2} \sum_{f \in \Fhb} h_f \Norm{\curl \ehu}{\L^2(f)}^2 \\
\label{eq:M2-2}
& \lesssim \frac{\nS}{2 \delta} h^{2k} \semiNorm{\u}{\H^{k+1}(\Omega)}^2 + \frac{\delta\nS}{2} \Norm{\ehu}{\dn}^2.
\end{alignat}

The term~$M_2^{(3)}$ can be estimated similarly:
\begin{alignat}{3}
\nonumber
M_2^{(3)} = \gammaone \nS \sum_{f \in \Fhb} h_f^{-1} (\eIu \times \nOmega, \ehu \times \nOmega)_f 
& \le \frac{\gammaone \nS}{2\delta} \sum_{f \in \Fhb} h_f^{-1} \Norm{\eIu}{\L^2(f)}^2 + \frac{\delta \nS}{2} \Norm{\ehu}{\dn}^2 \\
\label{eq:M2-3}
& \lesssim \frac{\gammaone \nS}{2\delta} h^{2k} \semiNorm{\u}{\H^{k+1}(\Omega)}^2 + \frac{\delta \nS}{2} \Norm{\ehu}{\dn}^2.
\end{alignat}
The desired result then follows combining~\eqref{eq:M2-1}--\eqref{eq:M2-3} with~\eqref{eq:M2-decomposition}.
\end{proof}

\begin{lemma}[Estimate of~$M_3$]
\label{lemma:M3}
If~$\u \in L^2(0, T; \H^{k+1}(\Omega)) \cap L^{\infty}(0, T; \WW^{1, \infty}(\Omega)) 
$, then the following estimate holds for any~$\delta>0$:
\begin{alignat*}{3}
M_3 & \lesssim h^{2k} \left(1+h^2\semiNorm{\u}{\WW^{1, \infty}(\Omega)} \right)\semiNorm{\u}{\H^{k+1}(\Omega)}^2 + \delta 
\semiNorm{\ehu}{\uh}^2 \\
& \quad + \left((1 + \delta^{-1}) \Norm{\u}{\L^{\infty}(\Omega)}^2 + \semiNorm{\u}{\WW^{1, \infty}(\Omega)} \right)\Norm{\ehu}{\L^2(\Omega)}^2.
\end{alignat*}
\end{lemma}
\begin{proof}
Adding and subtracting suitable terms, we split~$M_3$ as follows:
\begin{alignat}{3}
\nonumber
M_3 & = c(\u; \u, \ehu) - c(\uh; \Picurl \u, \ehu) \\
\nonumber
& = c(\eIu; \u, \ehu) 
+ c(\Picurl \u; \eIu, \ehu) + c(\ehu; \eIu, \ehu) - c(\ehu; \u, \ehu) \\
\label{eq:split-M3}
& =: M_3^{(1)} + M_3^{(2)} + M_3^{(3)}  + M_3^{(4)}.
\end{alignat}

Using the H\"older and the Young inequalities, and the approximation properties of~$\Picurl$ in Lemma~\ref{lemma:approx-Picurl}, we get
\begin{alignat}{3}
\nonumber
M_3^{(1)} = c(\eIu; \u, \ehu) = \big((\curl \eIu) \times \u, \ehu \big)_{\Omega} & \le \Norm{\u}{\L^{\infty}(\Omega)} \Norm{\curl \eIu}{\L^2(\Omega)} \Norm{\ehu}{\L^2(\Omega)}^2 \\
\label{eq:M3-1}
& \lesssim h^{2k}  \semiNorm{\u}{\H^{k+1}(\Omega)}^2 + \Norm{\u}{\L^{\infty}(\Omega)}^2 \Norm{\ehu}{\L^2(\Omega)}^2. 
\end{alignat}
The term~$M_3^{(2)}$ can be estimated similarly, using also a polynomial inverse estimate (see~\cite[Lemma 12.1]{Ern_Guermond-I:2020}) and the stability property in Lemma~\ref{lemma:picurl_linfty_stab} of~$\Picurl$:
\begin{alignat}{3}
\nonumber
M_3^{(2)} = c(\Picurl \u; \eIu, \ehu) & = \big((\curl \Picurl \u) \times \eIu, \ehu \big)_{\Omega} \\
\nonumber
& \le \Norm{\curl \Picurl \u}{\L^{\infty}(\Omega)} \Norm{\eIu}{\L^2(\Omega)} \Norm{\ehu}{\L^2(\Omega)} \\
\label{eq:M3-2}
& \lesssim \semiNorm{\u}{\WW^{1, \infty}(\Omega)} \big(h^{2k+2} \semiNorm{\u}{\H^{k+1}(\Omega)}^2 + \Norm{\ehu}{\L^2(\Omega)}^2 \big).
\end{alignat}

As for the term~$M_3^{(3)}$, we use the H\"older inequality, a polynomial inverse estimate, and estimate~\eqref{eq:Picurl-approx-Linfty} for~$\Picurl$ to obtain
\begin{alignat}{3}
\nonumber
M_3^{(3)} = c(\ehu; \eIu, \ehu)  = \big((\curl \ehu) \times \eIu, \ehu \big)_{\Omega} & \le \Norm{\eIu}{\L^{\infty}(\Omega)} \Norm{\curl \ehu}{\L^2(\Omega)} \Norm{\ehu}{\L^2(\Omega)} \\
\nonumber
& \lesssim h^{-1} \Norm{\eIu}{\L^{\infty}(\Omega)} \Norm{\ehu}{\L^2(\Omega)}^2 \\
\label{eq:M3-3}
& \lesssim \semiNorm{\u}{\WW^{1, \infty}(\Omega)} \Norm{\ehu}{\L^2(\Omega)}^2.
\end{alignat}

The estimate of the term~$M_3^{(4)}$ is more involved. 
Using identity~\eqref{eq:diff-identities-2} elementwise, we split~$M_3^{(4)}$ into two terms as follows:
\begin{alignat}{3}
\nonumber
M_3^{(4)} = - c(\ehu; \u, \ehu) & = \big((\curl \ehu) \times \ehu, \u \big)_{\Omega}\\
\label{eq:aux-M3-4}
& = \sum_{K \in \Th} \Big[\big((\Nabla \ehu) \ehu, \u \big)_{K} - \frac12 \big(\nabla |\ehu|^2, \u \big)_{K}\Big] =: M_3^{(4a)} + M_3^{(4b)}.
\end{alignat}

To estimate the first term on the right-hand side of~\eqref{eq:aux-M3-4}, we first observe that~$\ehu \cdot \Ihgrad{1} \u \in \Pp{k+1}{\Th}$, where~$\Ihgrad{1}$ is applied componentwise to~$\u$. Therefore, since~$\ehu \in \Xh$ and the gradient of global constants in~$\Omega$ is zero, the following identity holds: 
\begin{equation}
\label{eq:vanishing-average-operator}
\big(\ehu, \nabla \Jhav(\ehu \cdot \Ihgrad{1} \u) \big)_{\Omega} = 0.
\end{equation}

Adding and subtracting suitable terms, and using the identity in~\eqref{eq:vanishing-average-operator} and the identity
$$\v \cdot \nabla(\v \cdot \w) = \w \cdot \big((\Nabla \v) \v\big) + \v \cdot \big((\Nabla \w)\v\big),$$ 
we have
\begin{alignat}{3}
\nonumber
M_3^{(4a)} = \sum_{K \in \Th} \big((\Nabla \ehu) \ehu, \u \big)_K & = \sum_{K \in \Th} \Big[\big((\Nabla \ehu) \ehu, (\Id - \Ihgrad{1})\u \big)_K + \big((\Nabla \ehu) \ehu, \Ihgrad{1} \u \big)_{K}\Big] \\
\nonumber
& = \sum_{K \in \Th} \Big[\big((\Nabla \ehu) \ehu, (\Id - \Ihgrad{1})\u \big)_K + \big(\nabla (\ehu \cdot \Ihgrad{1} \u), \ehu \big)_{K} \\
\nonumber
& \quad - \big((\Nabla \Ihgrad{1} \u) \ehu, \ehu \big)_{K}\Big] \\
\nonumber
& = \sum_{K \in \Th} \Big[\big((\Nabla \ehu) \ehu, (\Id - \Ihgrad{1})\u \big)_K + \big(\nabla (\Id - \Jhavg) (\ehu \cdot \Ihgrad{1} \u), \ehu \big)_{K} \\
\nonumber
& \quad - \big((\Nabla \Ihgrad{1} \u) \ehu, \ehu \big)_{K}\Big] \\
\label{eq:split-M3-4a}
& =: \mathcal{L}_1 + \mathcal{L}_2 + \mathcal{L}_3.
\end{alignat}

To estimate~$\mathcal{L}_1$, we use the H\"older inequality, a polynomial inverse estimate, and the local approximation properties of~$\Ihgrad{1}$ in~$\L^{\infty}(K)$ to get
\begin{alignat}{3}
\nonumber
\mathcal{L}_1 & = \sum_{K \in \Th} \big((\Nabla \ehu) \ehu, (\Id - \Ihgrad{1}) \u \big)_K \\
\nonumber
& \le \sum_{K \in \Th} \Norm{\Nabla \ehu}{\L^2(K)} \Norm{\ehu}{\L^2(K)} \Norm{(\Id - \Ihgrad{1}) \u}{\L^{\infty}(K)} \\
\label{eq:L1}
& \lesssim \sum_{K \in \Th} \semiNorm{\u}{W^{1, \infty}(K)} \Norm{\ehu}{\L^2(K)}^2  \lesssim \semiNorm{\u}{W^{1, \infty}(\Omega)} \Norm{\ehu}{\L^2(\Omega)}^2.
\end{alignat}

As for~$\mathcal{L}_2$, we use the Young inequality, the approximation properties in Lemma~\ref{lemma:Jhav}, the continuity of~$\Ihgrad{1} \u$, and the definition of~$\gamma(\cdot)$ in~\eqref{eq:def-gamma-3}, and obtain 
\begin{alignat}{3}
\nonumber
\mathcal{L}_2 & = \sum_{K \in \Th} \big(\ehu, \nabla (\Id - \Jhavg) (\ehu \cdot \Ihgrad{1} \u) \big)_K \\
\nonumber
& \le \frac{\delta}{2\Norm{\u}{\L^{\infty}(\Omega)}^{2} } \sum_{K \in \Th} \semiNorm{(\Id - \Jhavg) (\ehu \cdot \Ihgrad{1} \u) }{H^1(K)}^2 + \frac{1}{2\delta} \Norm{\u}{\L^{\infty}(\Omega)}^{2}  \Norm{\ehu}{\L^2(\Omega)}^2 \\
\nonumber
& \lesssim \frac{\delta}{2} \frac{\Norm{\Ihgrad{1} \u}{\L^{\infty}(\Omega)}^{2}}{\Norm{\u}{\L^{\infty}(\Omega)}^{2} } \sum_{f \in \Fho}  h_f^{-1} \Norm{\jump{\ehu}}{\L^2(f)}^2 + \frac{1}{2\delta} \Norm{\u}{\L^{\infty}(\Omega)}^{2}  \Norm{\ehu}{\L^2(\Omega)}^2 \\
\label{eq:L2}
& \lesssim \frac{\delta}{2C_S}
\semiNorm{\ehu}{\uh}^2 + \frac{1}{2\delta} \Norm{\u}{\L^{\infty}(\Omega)}^{2} \Norm{\ehu}{\L^2(\Omega)}^2,
\end{alignat}
where~$C_S > 0$ is the safeguard constant in~\eqref{eq:def-gamma-3}, and we have used the trivial stability of~$\Ihgrad{1} \u$ in the~$\L^{\infty}(\Omega)$ norm. 

We now estimate~$\mathcal{L}_3$ using the H\"older inequality 
and the stability of~$\Ihgrad{1}$ in the~$W^{1,\infty}(\Omega)$ norm as follows:
\begin{alignat}{3}
\nonumber
\mathcal{L}_3  = - \sum_{K \in \Th} \big( (\Nabla \Ihgrad{1} \u) \ehu, \ehu \big)_K 
& \le \Norm{\Nabla \Ihgrad{1} \u}{\L^{\infty}(\Omega)} \Norm{\ehu}{\L^2(\Omega)}^2 \\
\label{eq:L3}
& \lesssim \semiNorm{\u}{\WW^{1, \infty}(\Omega)} \Norm{\ehu}{\L^2(\Omega)}^2.
\end{alignat}

Combining~\eqref{eq:L1}--\eqref{eq:L3} with~\eqref{eq:split-M3-4a}, we get
\begin{alignat}{3}
\label{eq:M3-4a}
M_3^{(4a)} \lesssim \frac{\delta}{2 C_S} \semiNorm{\ehu}{\uh}^2 + \left(\frac{1}{2\delta}\Norm{\u}{\L^{\infty}(\Omega)}^2 +\semiNorm{\u}{\WW^{1, \infty}(\Omega)}\right) \Norm{\ehu}{\L^2(\Omega)}^2.
\end{alignat}

We now consider the term~$M_3^{(4b)}$ on the right-hand side of~\eqref{eq:aux-M3-4}. 
Using integration by parts, the identity~$\jump{|\v|^2} = 2 \av{\v} \cdot \jump{\v}$ (where~$\av{\cdot}$ denotes the standard average operator), the fact that~$\u = \bO$ on~$\partial \Omega$ and $\ddiv \u = 0$ in $\Omega$, the definition in~\eqref{eq:def-gamma-3} of~$\gamma(\cdot)$, an inverse trace inequality, and the Young inequality, we have
\begin{alignat}{3}
\nonumber
M_3^{(4b)} = - \frac12 \sum_{K \in \Th} \big( \nabla |\ehu|^2, \u \big)_{K} & = - \frac12\sum_{f \in \Fho} \int_f (\u \cdot \nf) \jump{|\ehu|^2} \dS \\
\nonumber
& = - \sum_{f \in \Fho} \int_f (\u \cdot \nf) \av{\ehu} \cdot \jump{\ehu} \dS \\
\nonumber
& \le \Norm{\u}{\L^{\infty}(\Omega)} \Big(\sum_{f \in \Fho} \frac{h_f}{\gamma(\uh{}_{|_{f}})} \Norm{\av{\ehu}}{L^2(f)}^2 \Big)^{1/2} \semiNorm{\ehu}{\uh} \\
\nonumber
& \lesssim \frac{1}{\sqrt{C_S}}\Norm{\u}{\L^{\infty}(\Omega)} \Norm{\ehu}{\L^2(\Omega)} \semiNorm{\ehu}{\uh} \\
\label{eq:M3-4b}
& \lesssim \frac{\delta}{2} 
\semiNorm{\ehu}{\uh}^2 + \frac{1}{2 C_S \delta} 
\Norm{\u}{\L^{\infty}(\Omega)}^2 \Norm{\ehu}{\L^2(\Omega)}^2.
\end{alignat}
The proof is then concluded by combining~\eqref{eq:M3-1}, \eqref{eq:M3-2}, \eqref{eq:M3-3}, \eqref{eq:M3-4a}, and~\eqref{eq:M3-4b} with~\eqref{eq:split-M3}.
\end{proof}

\begin{lemma}[Estimate of~$M_4$]
\label{lemma:M4}
If~$\B \in L^2(0, T; \H^{k}(\Omega)) \cap L^{\infty}(0, T; \L^{\infty}(\Omega))$ with~$\curl \B \in L^2(0, T; \H^k(\Omega)) \cap L^{\infty}(0, T; \L^{\infty}(\Omega))$, 
the following estimate holds:
\begin{equation*}
\begin{split}
M_4 & \lesssim h^{2k} \big(  \Norm{\curl \B}{\L^{\infty}(\Omega)} \semiNorm{\B}{\H^{k}(\Omega)}^2 + \semiNorm{\curl \B}{\H^{k}(\Omega)}^2 \big)  \\
& \quad + \Norm{\curl \B}{\L^{\infty}(\Omega)} \Norm{\ehB}{\L^2(\Omega)}^2  + (\Norm{\B}{\boldsymbol{L}^{\infty}(\Omega)}^2 +  \Norm{\curl \B}{\L^{\infty}(\Omega)}) \Norm{\ehu}{\L^2(\Omega)}^2.
\end{split}
\end{equation*}
\end{lemma}
\begin{proof}
Adding and subtracting suitable terms, we can split~$M_4$ as follows:
\begin{alignat}{3}
\nonumber
M_4 & = c(\Jhcurl\B; \Bh, \ehu) - c(\B; \B, \ehu) \\
\nonumber
& = c(\Jhcurl \B; \ehB, \ehu) - c(\Jhcurl \B; \eIB, \ehu) - c(\eIB; \B, \ehu) \\
\label{eq:split-M4}
& =: M_4^{(1)} + M_4^{(2)} + M_4^{(3)}.
\end{alignat}

Using the properties of~$\Jhdiv$ and~$\Jhcurl$ from Lemmas~\ref{lemma:properties-Jhdiv} and~\ref{lemma:properties-Jhcurl}, respectively, together with the H\"older and the Young inequalities, we obtain
\begin{alignat}{3}
\nonumber
M_4^{(1)}  = c(\Jhcurl \B; \ehB, \ehu) & = \big((\curl (\Jhcurl \B)) \times \ehB, \ehu \big)_{\Omega} \\
\nonumber
&  = \big(\Jhdiv(\curl \B) \times \ehB, \ehu \big)_{\Omega} \\
\nonumber
& \le \Norm{\Jhdiv (\curl \B)}{\L^{\infty}(\Omega)} \Norm{\ehB}{\L^2(\Omega)} \Norm{\ehu}{\L^2(\Omega)} \\
\label{eq:M4-1}
& \lesssim \Norm{\curl \B}{\L^{\infty}(\Omega)} \big(\Norm{\ehB}{\L^2(\Omega)}^2 + \Norm{\ehu}{\L^2(\Omega)}^2\big).
\end{alignat}

Proceeding similarly and using the approximation properties of~$\Jhcurl$ in Lemma~\ref{lemma:properties-Jhcurl}, we get the following estimates of~$M_4^{(2)}$ and~$M_4^{(3)}$: 
\begin{alignat}{3}
\nonumber
M_4^{(2)} = -c(\Jhcurl \B; \eIB, \ehu) & = -\big((\curl (\Jhcurl \B)) \times \eIB, \ehu \big)_{\Omega} \\
\nonumber
& \lesssim \Norm{\curl \B}{\L^{\infty}(\Omega)} \Norm{\eIB}{\L^2(\Omega)} \Norm{\ehu}{\L^2(\Omega)} \\
\label{eq:M4-2}
& \lesssim  \Norm{\curl \B}{\L^{\infty}(\Omega)} (h^{2k} \semiNorm{\B}{\H^k(\Omega)}^2 +  \Norm{\ehu}{\L^2(\Omega)}^2 ), \\
\nonumber
\\
\nonumber
M_4^{(3)}  =  -c(\eIB; \B, \ehu) & = \big((\curl \eIB)  \times \B, \ehu\big)_{\Omega} \\
\nonumber
& \le \Norm{\B}{\L^{\infty}(\Omega)} \Norm{\curl \eIB}{\L^2(\Omega)} \Norm{\ehu}{\L^2(\Omega)} \\
\label{eq:M4-3}
& \lesssim h^{2k} 
\semiNorm{\curl \B}{\H^{k}(\Omega)}^2
+ \Norm{\B}{\L^{\infty}(\Omega)}^2 \Norm{\ehu}{\L^2(\Omega)}^2 . 
\end{alignat}
The proof concludes combining~\eqref{eq:M4-1}--\eqref{eq:M4-3} with~\eqref{eq:split-M4}. 
\end{proof}

\begin{lemma}[Estimate of~$M_5$]
\label{lemma:M5}
If~$\u \in L^2(0, T; \H^{k+1}(\Omega)) \cap L^{\infty}(0, T; \L^{\infty}(\Omega))$ and~$\B \in L^2(0, T; \H^k(\Omega)) \cap L^{\infty}(0, T; \L^\infty(\Omega))$, the following estimate holds for any~$\delta > 0$:
\begin{equation}
\label{eq:M5}
\begin{split}
M_5 & \lesssim h^{2k}  \semiNorm{\u}{\H^{k+1}(\Omega)}^2 + (\delta \nM)^{-1} h^{2k}  \Norm{\u}{\L^{\infty}(\Omega)}^2\semiNorm{\B}{\H^{k}(\Omega)}^2 + \delta \nM
\Norm{\curl \ehB}{\L^2(\Omega)}^2 \\
& \quad + \big(\Norm{\B}{\L^{\infty}(\Omega)}^2 + 
(\delta \nM)^{-1}
\Norm{\u}{\L^{\infty}(\Omega)}^2 \big) \Norm{\ehB}{\L^2(\Omega)}^2.
\end{split}
\end{equation}
\end{lemma}
\begin{proof}
Adding and subtracting suitable terms, the following splitting of~$M_5$ can be obtained:
\begin{alignat}{3}
\nonumber
M_5 & = c(\ehB; \B, \u) - c(\ehB; \Bh, \Picurl \u) \\
\nonumber
& = c(\ehB; \B, \eIu) + c(\ehB; \eIB, \Picurl \u) - c(\ehB; \ehB, \Picurl \u) \\
\label{eq:M5-split}
& =: M_5^{(1)} + M_5^{(2)} + M_5^{(3)}.
\end{alignat}

Using the H\"older inequality, polynomial inverse estimates, the approximation properties of~$\Picurl$, and the Young inequality, we get
\begin{alignat}{3}
\nonumber
M_5^{(1)} = c(\ehB; \B, \eIu) & = \big((\curl \ehB) \times \B, \eIu \big)_{\Omega} \\
\nonumber
& \le \Norm{\B}{\L^{\infty}(\Omega)} \Norm{\curl \ehB}{\L^2(\Omega)} \Norm{\eIu}{\L^2(\Omega)} \\
\nonumber
& \lesssim h^{-1} \Norm{\B}{\L^{\infty}(\Omega)} \Norm{\ehB}{\L^2(\Omega)} \Norm{\eIu}{\L^2(\Omega)} \\
\label{eq:M5-1}
& \lesssim  h^{2k} \semiNorm{\u}{\H^{k+1}(\Omega)}^2 + \Norm{\B}{\L^{\infty}(\Omega)}^2 \Norm{\ehB}{\L^2(\Omega)}^2 \, .
\end{alignat}

Proceeding similarly, and using the stability of~$\Picurl$ in the~$\L^{\infty}(\Omega)$ norm given in~\eqref{eq:stab-Picurl-infty} of Lemma~\ref{lemma:picurl_linfty_stab}, we have, for any~$\delta>0$,
\begin{alignat}{3} 
\nonumber
M_5^{(2)} = c(\ehB; \eIB, \Picurl \u) & = \big((\curl \ehB) \times \eIB, \Picurl \u \big)_{\Omega} \\
\nonumber
& \le \Norm{\Picurl \u}{\L^{\infty}(\Omega)} \Norm{\curl \ehB}{\L^2(\Omega)} \Norm{\eIB}{\L^2(\Omega)} \\
\nonumber
& \lesssim  \Norm{\u}{\L^{\infty}(\Omega)} \Norm{\curl \ehB}{\L^2(\Omega)} \Norm{\eIB}{\L^2(\Omega)} \\
\label{eq:M5-2}
&\lesssim  \delta \nM \Norm{\curl \ehB}{\L^2(\Omega)}^2 + (\delta \nM)^{-1} h^{2k} \Norm{\u}{\L^{\infty}(\Omega)}^2  \semiNorm{\B}{\H^{k}(\Omega)}^2\, .
\end{alignat}

As for~$M_5^{(3)}$, we use again Lemma~\ref{lemma:picurl_linfty_stab}, the H\"older inequality, and the Young inequality to obtain
\begin{alignat}{3}
\nonumber
M_5^{(3)} = - c(\ehB; \ehB, \Picurl \u) & = - \big((\curl \ehB) \times \ehB, \Picurl \u \big)_{\Omega} \\
\nonumber
& \le \Norm{\Picurl \u}{\L^{\infty}(\Omega)} \Norm{\curl \ehB}{\L^2(\Omega)} \Norm{\ehB}{\L^2(\Omega)} \\
\label{eq:M5-3}
& \lesssim \delta \nM  
\Norm{\curl \ehB}{\L^2(\Omega)}^2 + 
(\delta \nM)^{-1}\Norm{\u}{\L^{\infty}(\Omega)} ^2
\Norm{\ehB}{\L^2(\Omega)}^2 \, .
\end{alignat}
Estimate~\eqref{eq:M5} can then be obtained combining~\eqref{eq:M5-1}--\eqref{eq:M5-3} with~\eqref{eq:M5-split}.
\end{proof}
\begin{lemma}[Estimate of~$M_6$]\label{lemma:M6}
If~$\u \in 
L^2(0, T; \WW^{k+1,4}(\Omega))$, the following estimate holds for any~$\delta > 0$:
\begin{equation}\label{eq:lemma6}
M_6 \lesssim \delta^{-1} h^{2k} \Norm{\uh}{\L^2(\Omega)} \semiNorm{\u}{\WW^{k+1,4}(\Omega)}^2 + \delta \semiNorm{\ehu}{\uh}^2.
\end{equation}
\end{lemma}
\begin{proof}
Using the definition of the stabilization form~$\sh(\cdot; \cdot, \cdot)$ and the Young inequality, we get, for any~$\delta>0$, 
\begin{alignat}{3}
\nonumber
M_6 = \sh(\uh; \eIu, \ehu) & = \sum_{f \in \Fho} h_f^{-1} \gamma(\uh{}_{|_f}) \big(\jump{\eIu}, \jump{\ehu} \big)_{f} \\
\label{eq:M6-split}
& \le \frac{1}{2\delta} \sum_{f \in \Fho} h_f^{-1} \gamma(\uh{}_{|_f}) \Norm{\eIu}{\L^2(f)}^2 + \frac{\delta}{2} \semiNorm{\ehu}{\uh}^2.
\end{alignat}
It only remains to estimate the first term on the right-hand side of~\eqref{eq:M6-split}. 
Let us restrict to the case~$\gamma(\uh{}_{|_f}) = \Norm{\uh}{\L^{\infty}(f)}$ for all~$f \in \Fho$, the case~$\gamma(\uh{}_{|_f}) = C_S$ being analogous and simpler. 

    We denote by~$\omega_f$ the union of the two elements in~$\Th$ sharing an internal face~$f$, and by~$h_{\omega_f}$ the maximum of their diameters. Using the trace inequality for continuous functions, and the approximation properties of~$\Picurl$ in Lemma~\ref{lemma:approx-Picurl}, the H\"older inequality ($\semiNorm{\varphi}{L^2(D)} \le |D|^{1/4} \semiNorm{\varphi}{L^4(D)}$), the polynomial inverse estimate~$\Norm{\phi_h}{L^{\infty}(K)} \lesssim h_K^{-3/2} \Norm{\phi_h}{L^2(K)}$ in dimension~$3$, and the fact that~$|K| \simeq h_K^3$, we obtain
\begin{alignat*}{3}
\nonumber
\sum_{f \in \Fho} h_f^{-1} \Norm{\uh}{\L^{\infty}(f)} \Norm{\jump{\eIu}}{\L^2(f)}^2 & \lesssim \sum_{f \in \Fho} h_f^{-1} h_{\omega_f}^{2k + 1}
\Norm{\uh}{\L^{\infty}(\omega_f)}\semiNorm{\u}{\H^{k+1}(\omega_f)}^2 \\
\nonumber
& \lesssim \sum_{f \in \Fho} h_f^{2k} \Norm{\uh}{\L^{\infty}(\omega_f)}|\omega_f|^{1/2} \semiNorm{\u}{\WW^{k+1, 4}(\omega_f)}^2 \\
\nonumber
& \lesssim h^{2k} \sum_{f \in \Fho} \Norm{\uh}{\L^2(\omega_f)} \semiNorm{\u}{\WW^{k+1, 4}(\omega_f)}^2 \\
& \lesssim h^{2k} \Norm{\uh}{L^2(\Omega)} \semiNorm{\u}{\WW^{k+1, 4}(\Omega)}^2,
\end{alignat*}
which completes the proof. 
\end{proof}

\begin{remark}[Uniform bound on~$\Norm{\uh}{L^2(\Omega)}$]
\label{rem:uh-bound}
Recalling the stability bound in~\eqref{eq:stab-semidiscrete}, under suitable assumptions on the data, the quantity $\Norm{\uh(\cdot, t)}{\L^2(\Omega)}$ appearing in~\eqref{eq:lemma6} is uniformly bounded for a.e. $t \in (0, T)$.
\eremk
\end{remark}

\begin{lemma}[Estimate of~$M_7$]
\label{lemma:M7}
Let~$\widetilde{\f} = \f + \nabla p$. 
If~$\widetilde{\f} \in L^2(0, T; \H^k(\Omega))$, 
then
\begin{equation*}
M_7 \lesssim h^{2k} 
\semiNorm{\widetilde{\f}}{\H^k(\Omega)}^2 
+\Norm{\ehu}{\L^2(\Omega)}^2.
\end{equation*}
\end{lemma}
\begin{proof}
The estimate can be obtained from the fact that~$(\Ihcurl{k}(\nabla \phi), \vh) = 0$ for all~$\phi \in H^1(\Omega)$ and~$\vh \in \Xh$ (see Remark~\ref{rem:pressure-robustness}), the approximation properties of~$\Ihcurl{k}$ from~\cite[Cor.~19.9(i)]{Ern_Guermond-I:2020}, and the Young inequality as follows:
\begin{alignat*}{3}
M_7 = \big(\Ihcurl{k} (\f) - \f - \nabla p, \ehu \big)_{\Omega} & = \big(\Ihcurl{k}(\widetilde{\f}) - \widetilde{\f}, \ehu \big)_{\Omega}  \lesssim h^{2k} \semiNorm{\widetilde{\f}}{\H^{k}(\Omega)}^2 
+ \Norm{\ehu}{\L^2(\Omega)}^2.
\end{alignat*}
\end{proof}

\begin{remark}[Pressure-robust character of~$M_7$]\label{corol:presrob}
By equation \eqref{eq:MHD-system-curl-1} and recalling that~$p$ represents the modified pressure, we can write
$$
\widetilde{\f} = \dpt \u + \nS \curl(\curl \u) + (\curl \u) \times \u + \B \times \curl \B \, .
$$
Consequently, if all the terms on the right-hand side belong to 
$L^2(0,T; \H^{k}(\Omega))$, also 
the approximation of the term~$M_7$ is independent of the modified pressure~$p$. 
\eremk
\end{remark}

The previous results and a Gr\"onwall argument allow us to derive \emph{a priori} error estimates for the solution of the semidiscrete formulation~\eqref{eq:curl-curl-semidiscrete}.

\begin{theorem}[\emph{A priori} error estimates]\label{th:main:1}  
Let the solution~$(\u,\pressure,\B)$ 
to problem \eqref{eq:MHD-system} satisfy the regularity Assumption \ref{asm:regularity}. Let also~$(\uh,\ph, \Bh)$ be the solution to 
the semidiscrete formulation~\eqref{eq:curl-curl-semidiscrete}, assuming~$\f \in C^0([0, T]; {\cal S})$, where~${\cal S}$ has sufficient regularity for $\Ihcurl{k}(\f)$ to be well defined (see Remark~\ref{rem:f-interp}).
Furthermore, let the mesh Assumptions \ref{asm:mesh} and \ref{asm:mesh:2} hold, and the parameter~$\gammaone$ be sufficiently large.
If 
\[
\u \in L^\infty(0,T;\bW^{1,\infty}(\Omega)),\quad 
\B, \curl \B
\in L^\infty(0,T;\L^{\infty}(\Omega)),
\]
and the following additional $k$-dependent regularity conditions hold:
\begin{equation*}
\begin{aligned}
& \u, \B \in H^1(0,T;\H^k(\Omega)) \, , \quad 
& & \u \in L^2(0,T;\bW^{k+1,4}(\Omega)) \, , \quad & & \curl \B \in L^2(0,T;\H^{k}(\Omega)) \, ,\\
& \f \in L^2(0,T; \H^{k}(\Omega))\, , \quad & & p \in L^2(0, T; H^{k+1}(\Omega)),
\end{aligned}
\end{equation*}
with~$p=\pressure+|\u|^2/2$, 
then the following estimate holds for all~$t \in (0, T]$: 
\begin{alignat}{3}
\nonumber
        \Norm{\eu}{L^{\infty}(0, t; \L^2(\Omega))}^2 + \Norm{\eB}{L^{\infty}(0, t; \L^2(\Omega))}^2 
    & + \beta \!\! \int_0^{t} \!\!\Big( \nS \Norm{\eu(\cdot, s)}\dn^2 + \nM \Norm{\curl \eB(\cdot, s)}{\L^2(\Omega)}^2 + \semiNorm{\eu(\cdot, s)}{\uh}^2 \Big) \ds \\
\label{eq:main:bound}
    & \lesssim 
    (1 + \nS + \nM + \nM^{-1}) e^{{\cal R}_2(1 + \nM^{-1}) t} h^{2k} \, ,
\end{alignat}
where the hidden constant is independent of~$h$, $\nS$,  and $\nM$, but depends, in particular, on the norms of the continuous solution indicated in the assumptions above and the mesh regularity parameters. Moreover, the constant~${\cal R}_2$ depends on~$\Norm{\u}{L^{\infty}(0, T; \bW^{1, \infty}(\Omega))}$, $\Norm{\B}{L^{\infty}(0, T; \L^{\infty}(\Omega))}$, and~$\Norm{\curl \B}{L^{\infty}(0, T; \L^{\infty}(\Omega))}$.
\end{theorem}
\begin{proof} 
After combining Proposition \ref{prop:discrete-errors} with Lemmas \ref{lemma:M0}--\ref{lemma:M7} (see also Remark \ref{rem:uh-bound}), we observe that all terms multiplied by~$\delta$ can be absorbed into the left-hand side by taking $\delta>0$ sufficiently small, only depending on the 
shape-regularity parameter of the mesh and the safeguard constant~$C_S$ in~\eqref{eq:def-gamma-3}.
We thus obtain, for a.e. $t \in (0,T)$,
\begin{equation}\label{pixel:1}
\begin{aligned}
\ddt \Big(\Norm{\ehu}{\L^2(\Omega)}^2  + \Norm{\ehB}{\L^2(\Omega)}^2\Big) & + \beta \nS \Norm{\ehu}\dn^2 + \nM \Norm{\curl \ehB}{\L^2(\Omega)}^2 + \semiNorm{\ehu}{\uh}^2  \\
& \le 
\, h^{2k}  \,(1 +\nS + \nM + \nM^{-1}) \, {\cal R}_1 \,  + \, (1 + \nM^{-1}) {\cal R}_2 \, ( \Norm{\ehu}{\L^2(\Omega)}^2 + \Norm{\ehB}{\L^2(\Omega)}^2),
\end{aligned}
\end{equation}
where
\begin{itemize}
\item 
the positive real function~${\cal R}_1 = {\cal R}_1(t)$, $t \in [0,T]$, is independent of $h$, $\nM$, and~$\nS$ but depends, in particular, on the 
shape-regularity parameter, the quasi-uniformity constant, the domain $\Omega$, and the following norms:
\begin{equation*}
\begin{aligned}
& \Norm{\partial_t \u (\cdot, t)}{\H^k(\Omega)} , \quad
\Norm{\partial_t \B (\cdot, t)}{\H^k(\Omega)} , \quad
\semiNorm{\u (\cdot, t)}{\bW^{k+1, 4}(\Omega)} ,  
\\
& \Norm{\B (\cdot, t)}{\H^{k}(\Omega)} , \quad \Norm{\curl \B (\cdot, t)}{\H^{k}(\Omega)}, \quad
\Norm{\f (\cdot, t)}{\H^{k}(\Omega)} , \quad
\Norm{p (\cdot, t)}{H^{k+1}(\Omega)} \, ,
\end{aligned}
\end{equation*}
as well as on
$$\Norm{\u}{L^{\infty}(0, t; \WW^{1,\infty}(\Omega))}\quad \text{and}\quad\Norm{\curl \B}{L^{\infty}(0, t; \L^{\infty}(\Omega))};$$
\item 
the positive real function~${\cal R}_2$ is independent of $t$, $h$, $\nM$, and~$\nS$ but depends, in particular, on the mesh regularity parameters and the norms~$\Norm{\u}{L^{\infty}(0, T; \bW^{1, \infty}(\Omega))}$, $\Norm{\B}{L^{\infty}(0, T; \L^{\infty}(\Omega))}$, and $\Norm{\curl\B}{L^{\infty}(0, T; \L^{\infty}(\Omega))}$.
\end{itemize}
By classical arguments and the Gr\"onwall lemma 
bound \eqref{pixel:1} yields, for all~$t\in (0,T)$,
\begin{equation*}
\begin{aligned}
&\Norm{\ehu(\cdot, t)}{\L^2(\Omega)}^2   + \Norm{\ehB(\cdot, t)}{\L^2(\Omega)}^2 \\ 
    &\qquad +  \!\! \int_0^t \!\!e^{(1+\nu_M^{-1}){\cal R}_2(t-s)}\Big( \beta \nS \Norm{\ehu(\cdot, s)}\dn^2 + \nM \Norm{\curl \ehB(\cdot, s)}{\L^2(\Omega)}^2 + \semiNorm{\ehu(\cdot, s)}{\uh}^2 \Big) {\rm d}s \\
& \qquad \qquad \quad \le 
e^{(1+\nu_M^{-1}){\cal R}_2 t}\Big( \Norm{\ehu(\cdot, 0)}{\L^2(\Omega)}^2 + \Norm{\ehB(\cdot, 0)}{\L^2(\Omega)}^2\Big)\\
& \qquad \qquad\qquad+ h^{2k} (1 +\nS + \nM + \nM^{-1}) \, \int_0^t \!\! 
e^{(1+\nu_M^{-1}){\cal R}_2(t-s)}
{\cal R}_1(s) \ds .
\end{aligned}
\end{equation*}  
Note that, by the approximation estimates in Lemmas \ref{lemma:approx-Picurl} and \ref{lemma:properties-Jhcurl}, 
\begin{equation*}\label{pixel:3t}
\Norm{\ehu(\cdot, 0)}{\L^2(\Omega)}^2 + \Norm{\ehB(\cdot, 0)}{\L^2(\Omega)}^2 
\lesssim h^{2k} \, ( \Norm{\u_0}{\H^k(\Omega)}^2 + \Norm{\B_0}{\H^k(\Omega)}^2 ) \, .
\end{equation*} 
Then, estimating the exponential on the left--hand side from below by~$1$ and the exponential in the integral on the right--hand side from above by~$e^{(1+\nu_M^{-1}){\cal R}_2 t}$, we obtain
\begin{alignat}{3}
\label{pixel:2t}
\nonumber
&\Norm{\ehu(\cdot, t)}{\L^2(\Omega)}^2   + \Norm{\ehB(\cdot, t)}{\L^2(\Omega)}^2  
    +  \!\! \int_0^t \Big( \beta \nS \Norm{\ehu(\cdot, s)}\dn^2 + \nM \Norm{\curl \ehB(\cdot, s)}{\L^2(\Omega)}^2 + \semiNorm{\ehu(\cdot, s)}{\uh}^2 \Big) {\rm d}s \\
& \quad \lesssim
\Big( \Norm{\u_0}{\H^k(\Omega)}^2 + \Norm{\B_0}{\H^k(\Omega)}^2 +  (1 +\nS + \nM + \nM^{-1}) \, \int_0^T \!\! {\cal R}_1(s) \ds \Big) \, e^{(1 + \nM^{-1}) {\cal R}_2 t} h^{2k}\, .
\end{alignat}
Furthermore, by standard arguments, the approximation estimates in Lemmas~\ref{lemma:approx-Picurl} and~\ref{lemma:properties-Jhcurl} easily lead to
\begin{equation}\label{pixel:4t}
\begin{aligned}
&\Norm{\eIu(\cdot,t)}{\L^2(\Omega)}^2   + \Norm{\eIB(\cdot,t)}{\L^2(\Omega)}^2 + \!\! \int_0^t \Big( \beta \nS \Norm{\eIu(\cdot, s)}\dn^2 + \nM \Norm{\curl \eIB(\cdot, s)}{\L^2(\Omega)}^2 + \semiNorm{\eIu(\cdot, s)}{\uh}^2 \Big) {\rm d}s \\
&\quad \lesssim (1 +\nS + \nM) \, h^{2k} \! \int_0^t \!\! {\cal R}_1(s) {\rm d}s \, .
\end{aligned} 
\end{equation} 
The result now follows from the estimates in~\eqref{pixel:2t} and~\eqref{pixel:4t}, and the triangle inequality.
\end{proof}

\begin{remark}[Pressure-robust estimate]\label{rem:meth1:presrob}
The error estimate in Theorem~\ref{th:main:1} also depends on the $L^2(0,T;H^{k+1}(\Omega))$ norm of~$p$. 
In order to obtain an error estimate that reflects the pressure robustness of the scheme (i.e. independent of $p$), 
we just recall Remark~\ref{corol:presrob}.
Consequently, the constant $C$ appearing in \eqref{eq:main:bound} becomes independent of $p$, at the expense of requiring the additional regularity mentioned in that remark.
\eremk
\end{remark}

\begin{remark}[Suitability for nonconvex polyhedral domains]
Thanks to the formulation and discrete spaces adopted in method~\eqref{eq:curl-curl-semidiscrete}, we have derived error estimates under the spatial regularity assumptions for the magnetic field~$\B(\cdot,t) \in \H^k(\Omega) \cap \L^\infty (\Omega)$ and~$\curl\B(\cdot,t) \in \H^k(\Omega) \cap \L^\infty (\Omega)$ a.e. in~$(0, T)$, as opposed to the stronger requirement~$\B(\cdot,t) \in \H^{k+1}(\Omega) \cap {\bf W}^{1,\infty} (\Omega)$, typically needed in related approaches (see, e.g., \cite[\S5.2]{RobustMHD}). 
These assumptions are better suited to the regularity expected for magnetic fields in nonconvex polyhedral domains. 
Furthermore, with a simple modification of bounds~\eqref{eq:M4-3} and~\eqref{eq:M5-1}, based on a different 
application of the H\"older inequality, one can verify that the minimal regularity assumptions in space required for~$\B$ are~$\B (\cdot, t) \in  \H^s(\Omega)$, $\dpt \B(\cdot, t) \in \H^s(\Omega)$, 
and~$\curl\B(\cdot, t) \in \WW^{s,\infty}(\Omega)$, with $s \in (1/2,1]$, 
yielding an~${\mathcal O}(h^s)$ convergence rate. Importantly, the method does not enforce~$\H^1(\Omega)$-regularity on $\Bh$ in the limit of vanishing $h$, thereby allowing solutions with reduced regularity. In particular, this includes the 
classical magnetostatic singularities arising in nonconvex polyhedral domains.
\eremk
\end{remark}

\begin{remark}[Lack of~$\nM$-quasi-robustness and $\mathcal{O}(h^k)$ convergence]
\label{rem:lack}
The error estimate \eqref{pixel:1} shows that the method is quasi-robust with respect to the fluid Reynolds number. Indeed, since no factor of~$\nS^{-1}$ appears, the right-hand side remains bounded as~$\nS \to 0$. In contrast, the presence of $\nM^{-1}$ indicates that the method is not quasi-robust with respect to the magnetic Reynolds number. Such quasi-robustness can be achieved at the expense of additional stabilization terms and regularity assumptions, as we show in the following Section \ref{sec:nm-rob}. Furthermore, for method~\eqref{eq:curl-curl-semidiscrete},
we are unable to obtain pre-asymptotic error reduction rates of order~$\mathcal{O}(h^{k+\frac{1}{2}})$ for the velocity when~$\nS < h$. This limitation arises from the $h^{-1}$ factor in the form~$s_h(\w_h; \u_h, \v_h)$, which is required to control the term~$M_3^{(4)}$ in~\eqref{eq:aux-M3-4}. However, for all the other terms, by introducing additional stabilization and assuming higher regularity, 
it is possible to recover convergence rates of order~$\mathcal{O}(h^{k+\frac{1}{2}})$, as we demonstrate in Section \ref{sec:ideal} below.   
\eremk
\end{remark}

\section{A \texorpdfstring{$\nM$}{nu-M}-quasi-robust variant}\label{sec:nm-rob}

In this section, we introduce an additional stabilization term so that the constant in the final error estimate does not depend on $\nM^{-1}$, and therefore does not grow unboundedly when $\nM\to 0$. To obtain such a robust estimate, we require higher regularity on the magnetic field~$\B$ than that assumed in Theorem~\ref{th:main:1} for the method in~\eqref{eq:curl-curl-semidiscrete}.
We modify the semidiscrete problem as follows: for all~$t \in (0, T]$, find~$(\uh(\cdot, t), \ph(\cdot, t), \Bh(\cdot, t), \phi_h(\cdot, t))\in \Vcurl \times \bVgrad \times \Vcurl \times \bVgrad$, with $\uh$ and $\Bh$ differentiable in time, such that
\begin{subequations}
\label{eq:curl-curl-semidiscrete-nmrobust}
\begin{alignat}{3}
\nonumber
(\dpt \uh, \vh)_{\Omega} + \nS a(\uh, \vh) + c(\uh;\uh,\vh) - c(\Bh; \Bh, \vh) \\
\label{eq:curl-curl-semidiscrete-a-nmrobust}
 + \nS d_h(\uh, \vh) - b(\vh, \ph) + \pds s_h(\uh;\uh,\vh) & = (\Ihcurl{k} (\f), \vh)_{\Omega} & & \quad \forall \vh \in \Vcurl, \\
\label{eq:curl-curl-semidiscrete-b-nmrobust}
b(\uh, \qh) & = 0 & & \quad \forall \qh \in \bVgrad, \\
\nonumber
(\dpt \Bh, \Ch)_{\Omega} + \nM a(\Bh, \Ch) 
+ c(\Ch; \Bh, \uh)\\
+ \pdsb s_h(\u_h; \B_h, \C_h) +  b(\C_h, \phi_h) & = 0 & & \quad \forall \Ch \in \Vcurl, \label{eq:curl-curl-semidiscrete-c-nmrobust} \\
b(\B_h, \psi_h) &= 0 & & \quad \forall \psi_h \in \bVgrad,
\label{eq:curl-curl-semidiscrete-d-nmrobust}
\\ 
\uh(\cdot, 0) = \Picurl \u_0(\cdot) \quad \text{ and } \quad \Bh(\cdot, 0) = \Picurl \B_0(\cdot) &   & & \quad \text{ in~$\Omega$} ,
\label{eq:curl-curl-semidiscrete-e-nmrobust}
\end{alignat}
\end{subequations} 
where $\pds$ and~$\pdsb$ are real positive parameters, which are set equal to~$1$ in the forthcoming analysis to simplify the presentation.

\begin{remark}[Additional stabilization term]\label{rem:limitations}
The main difference with respect to the scheme in~\eqref{eq:curl-curl-semidiscrete} is the addition of the term~$s_h(\u_h; \B_h, \C_h)$ in~\eqref{eq:curl-curl-semidiscrete-c-nmrobust}, which is a jump stabilization 
for the magnetic field. 
Consequently, we have also introduced a Lagrange multiplier $\phi_h$ in~\eqref{eq:curl-curl-semidiscrete-nmrobust}, since $\Xh$ is no longer an invariant subspace for \eqref{eq:curl-curl-semidiscrete-c-nmrobust} due to the presence of $s_h(\u_h; \B_h, \C_h)$. Furthermore, we note that this jump stabilization enforces additional regularity on the discrete magnetic field~$\Bh$;
in particular, or every~$t$, it may drive $\Bh$ to converge, as~$h\to 0$, to a limit function in $\H^1(\Omega)$. As a consequence, the formulation implicitly restricts the class of admissible solutions to~$\H^1(\Omega)$-regular magnetic fields, a property that may fail, for instance, in nonconvex polyhedral domains.
Both the convexity assumption and the introduction of a Lagrange multiplier to handle the divergence-free constraint are standard features in the literature; see, e.g., \cite{RobustMHD, DiPietroDroniuPatierno26}. In Section~\ref{sec:ideal} below, we present an alternative method that can potentially circumvent these limitations.
\eremk
\end{remark}

\subsection{\emph{A priori} error estimates}
For the error analysis of~\eqref{eq:curl-curl-semidiscrete-nmrobust}, compared to that of~\eqref{eq:curl-curl-semidiscrete}, we replace~$\Jhcurl$ with~$\Picurl$ 
in the choice of the approximant for the magnetic field~$\B$. More precisely, we define
\begin{equation}\label{eq:newerrordefs}
\begin{aligned}
    \eu & \coloneqq \u - \uh, & \qquad  \eIu & \coloneqq \u - \Picurl\u,&  \qquad \ehu & \coloneqq \uh - \Picurl\u,\\
    \eB & \coloneqq \B - \Bh, & \qquad \eIB & \coloneqq \B - \Picurl \B, & \qquad \ehB & \coloneqq \Bh - \Picurl \B.
\end{aligned}
\end{equation}
The reason for such a modification is to guarantee that~$\ehB \in \Xh$, 
which is now required in the initial steps of the following proof.
By arguments similar to those employed in the proof of Proposition~\ref{prop:discrete-errors}, and making use of the discrete and continuous equations (also exploiting, as noted above, that $\ehB \in \Xh$), for a.e. $t \in (0,T)$, we obtain
\begin{equation*}
    \frac12 \ddt \Big(\Norm{\ehu}{\L^2(\Omega)}^2 + \Norm{\ehB}{\L^2(\Omega)}^2\Big) + \beta \nS\Norm{\ehu}\dn^2 + \nM \Norm{\curl \ehB}{\L^2(\Omega)}^2 + \semiNorm{\ehu}{\uh}^2  + \semiNorm{\ehB}{\u_h}^2 \le \sum_{i = 0}^{8} M_i \, ,
\end{equation*}
where
$$
M_8 := \sh(\uh; \eIB, \ehB), 
$$
and~$M_0, \ldots ,M_7$ are as defined in Proposition~\ref{prop:discrete-errors}, with the only difference being that, in $M_4$, $\Jhcurl\B$ is replaced by~$\Picurl\B$.

\paragraph{Estimates of individual terms.}
The estimates of~$M_0$, $M_2$, $M_3$, $M_6$, and $M_7$ are identical to those in the previous section (see Lemmas~\ref{lemma:M0}, \ref{lemma:M2}, \ref{lemma:M3}, \ref{lemma:M6}, and~\ref{lemma:M7}). We estimate the remaining terms.

\begin{lemma}[Estimate of~$M_1$]
Let~$\u$ and~$\B$ belong to~$L^2(0, T; \H^{k+1}(\Omega))$.
 Then, for any~$\delta > 0$, it holds
\begin{equation*}
\begin{split}
M_1 & = \nS a(\eIu, \ehu) + \nM a(\eIB, \ehB) \\
& \lesssim \frac{1}{2\delta} h^{2k} \big(\nS \semiNorm{\u}{\H^{k+1}(\Omega)}^2 + \nM \semiNorm{\B}{\H^{k+1}(\Omega)}^2 \big) + \frac{\delta}{2}\big(\nS\Norm{\ehu}{\dn}^2 + \nM \Norm{\curl \ehB}{\L^2(\Omega)}^2 \big).
\end{split}
\end{equation*}
\end{lemma}
\begin{proof}
The proof uses the same arguments as those for Lemma~\ref{lemma:M1}, but employs the approximation properties in Lemma~\ref{lemma:approx-Picurl} of~$\Picurl$ also for the magnetic field.
\end{proof}

\begin{lemma}[Estimate of~$M_4$]
If~$\B \in L^2(0, T; \H^{k+1}(\Omega)) \cap L^{\infty}(0, T; \WW^{1, \infty}(\Omega))$, 
the following estimate holds:
\begin{equation*}
M_4 \lesssim h^{2k} \semiNorm{\B}{\H^{k+1}(\Omega)}^2 + 
 \semiNorm{\B}{{\bf W}^{1,\infty}(\Omega)} \Norm{\ehB}{\L^2(\Omega)}^2  
 + \Big( \semiNorm{\B}{{\bf W}^{1,\infty}(\Omega)} + \Norm{\B}{\boldsymbol{L}^{\infty}(\Omega)}^2 \Big) \Norm{\ehu}{\L^2(\Omega)}^2.
\end{equation*}
\end{lemma}
\begin{proof}
We split the term $M_4$ as in \eqref{eq:split-M4}, 
recalling 
that $\Jhcurl\B$ is now replaced by~$\Picurl\B$ in all terms. The first term $M_4^{(1)}$ is bounded essentially as in~\eqref{eq:M4-1}, but now using the stability properties in Lemma~\ref{lemma:picurl_linfty_stab} of~$\Picurl$. We obtain 
\begin{equation}
\nonumber
M_4^{(1)}  = c(\Picurl \B; \ehB, \ehu) \lesssim 
\semiNorm{\B}{{\bf W}^{1,\infty}(\Omega)} \big(\Norm{\ehB}{\L^2(\Omega)}^2 + \Norm{\ehu}{\L^2(\Omega)}^2\big).
\end{equation}
By an inverse estimate, the mesh quasi-uniformity, and the approximation bound~\eqref{eq:Picurl-approx-L2} for~$\Picurl$, we obtain
\begin{alignat*}{3}
\nonumber
M_4^{(2)} = -c(\Picurl \B; \eIB, \ehu) & 
\lesssim h^{-1} \Norm{\B}{\L^{\infty}(\Omega)} \Norm{\eIB}{\L^2(\Omega)} \Norm{\ehu}{\L^2(\Omega)} \\
& \lesssim h^{k} \Norm{\B}{\L^{\infty}(\Omega)} \semiNorm{\B}{\H^{k+1}(\Omega)} \Norm{\ehu}{\L^2(\Omega)} \\
& \lesssim  h^{2k} \semiNorm{\B}{\H^{k+1}(\Omega)}^2 + \Norm{\B}{\L^{\infty}(\Omega)}^2 \Norm{\ehu}{\L^2(\Omega)}^2.
\nonumber
\end{alignat*}
The term $M_4^{(3)}$ is bounded as in \eqref{eq:M4-3}, but now 
uses the approximation bound~\eqref{eq:Picurl-approx-curl} for~$\Picurl$, yielding
$$
M_4^{(3)}  =  -c(\eIB; \B, \ehu) 
\lesssim h^{2k}  \semiNorm{\B}{\H^{k+1}(\Omega)}^2
+ \Norm{\B}{\L^{\infty}(\Omega)}^2 \Norm{\ehu}{\L^2(\Omega)}^2 . 
$$
The proof concludes combining the above bounds with~\eqref{eq:split-M4}.
\end{proof}

\begin{lemma}[Estimate of~$M_5$]
\label{lemma:M5-bis}
If $\u \in L^2(0, T; \H^{k+1}(\Omega)) \cap L^{\infty}(0, T; \WW^{1, \infty}(\Omega))$ and~$\B \in L^2(0, T; \H^{k+1}(\Omega)) \cap L^{\infty}(0, T; \L^{\infty}(\Omega))$, 
the following estimate holds for any~$\delta > 0$:
\begin{equation}
\label{eq:M5_nu}
\begin{aligned}
M_5 & \lesssim h^{2k} \big(\semiNorm{\u}{\H^{k+1}(\Omega)}^2 + \semiNorm{\B}{\H^{k+1}(\Omega)}^2 \big) 
+ \delta \semiNorm{\ehB}{\uh}^2\\
&\quad + \big( (1+\delta^{-1})\Norm{\u}{\L^{\infty}(\Omega)}^2 + \Norm{\B}{\L^{\infty}(\Omega)}^2 + \semiNorm{\u}{\WW^{1, \infty}(\Omega)} \big) \Norm{\ehB}{\L^2(\Omega)}^2.
\end{aligned}
\end{equation}
\end{lemma}
\begin{proof}
We split $M_5$ as in \eqref{eq:M5-split}. While the first term $M_5^{(1)}$ is handled identically as in \eqref{eq:M5-1}, the other two terms need to be modified to obtain $\nM$-quasi-robust bounds.
For the term $M_5^{(2)}$, we proceed 
as in~\eqref{eq:M5-2}, 
but we now make use of an inverse estimate to obtain
\begin{alignat*}{3} 
M_5^{(2)} = c(\ehB; \eIB, \Picurl \u) 
& \le \Norm{\Picurl \u}{\L^{\infty}(\Omega)} \Norm{\curl \ehB}{\L^2(\Omega)} \Norm{\eIB}{\L^2(\Omega)} \\
\nonumber
& \lesssim  \Norm{\u}{\L^{\infty}(\Omega)} h^{-1} \Norm{\ehB}{\L^2(\Omega)} \Norm{\eIB}{\L^2(\Omega)} \\
&\lesssim  h^{2k} \semiNorm{\B}{\H^{k+1}(\Omega)}^2
+ \Norm{\u}{\L^{\infty}(\Omega)}^2  \Norm{\ehB}{\L^2(\Omega)}^2 \, .
\end{alignat*}
We split the third term~$M_5^{(3)}$ as :
\begin{equation}
\nonumber
M_5^{(3)} =  - c(\ehB; \ehB, \Picurl \u)
= - c(\ehB; \ehB, \Picurl \u - \u) - c(\ehB; \ehB, \u) 
=: M_5^{(3a)} + M_5^{(3b)} \, .
\end{equation}
As for~$M_5^{(3a)}$, we use the approximation property in~\eqref{eq:Picurl-approx-Linfty} of $\Picurl$  and an inverse inequality to obtain 
\begin{alignat*}{3}
M_5^{(3a)} &= - c(\ehB; \ehB, \Picurl\u - \u)
          = -\bigl(( \curl \ehB) \times \ehB, \Picurl \u - \u)_{\Omega} 
          \lesssim \semiNorm{\u}{\WW^{1, \infty}(\Omega)} \Norm{\ehB}{\L^2(\Omega)}^2.
\end{alignat*}
Finally, 
due to the skew-symmetry property of $c(\cdot;\cdot,\cdot)$, we have 
$M_5^{(3b)} = c(\ehB; \u, \ehB)$. 
As a consequence, we can estimate $M_5^{(3b)}$ identically to~$M_3^{(4)}$, see
\eqref{eq:aux-M3-4}, since the two terms are the same up to the trivial substitution of $\ehB$ in lieu of $\ehu$. 
In this respect, the presence of the stabilization term~$\sh(\cdot; \cdot,\cdot)$ 
also for the magnetic field is critical.
We have 
\begin{equation}
    M_5^{(3b)} \lesssim 
    \Big(\frac{1+C_S}{2C_S}\Big)\delta \semiNorm{\ehB}{\uh}^2 
+  
\Big( \Big(\frac{1+C_S}{2C_S}\Big)\delta^{-1} \Norm{\u}{\L^{\infty}(\Omega)}^2 + \semiNorm{\u}{\WW^{1, \infty}(\Omega)} \Big)
\Norm{\ehB}{\L^2(\Omega)}^2 .
    \label{eq:M5-4-stab}
\end{equation}
The derivation of~\eqref{eq:M5-4-stab} is analogous to that of~$M_3^{(4)}$ in Lemma~\ref{lemma:M3}, with~$\ehB$ in place of~$\ehu$. We emphasize that the first term still yields the seminorm indexed by~$\uh$, since in the analogue of~\eqref{eq:M3-4b} we still multiply and divide by~${\gamma(\uh{}_{|_{f}})}$ (and not by~${\gamma(\Bh{}_{|_{f}})}$).
Estimate~\eqref{eq:M5_nu} can then be obtained combining the bounds above with~\eqref{eq:M5-split}.
\end{proof}

\begin{lemma}[Estimate of~$M_8$]
If $\B \in L^2(0, T; \WW^{k+1,4}(\Omega))$, 
the following estimate holds:
\begin{equation*}
M_8 \lesssim \delta^{-1} h^{2k} \Norm{\uh}{\L^2(\Omega)} \semiNorm{\B}{\WW^{k+1,4}(\Omega)}^2 + \delta \semiNorm{\ehB}{\uh}^2.
\end{equation*}
\end{lemma}

\begin{proof}
The estimate for the new term~$M_8$ is 
treated exactly as~$M_6$ in Lemma~\ref{lemma:M6}, 
up to the trivial substitutions of~$\ehu$ by~$\ehB$ and~$\eIu$ by~$\eIB$. 
\end{proof}

Once the above modified bounds are derived, we can apply the same arguments as in the proof of Theorem~\ref{th:main:1}, obtaining the following error estimates.   

\begin{theorem}[\emph{A priori} error estimates]\label{th:main:1:bis}  
Let the solution~$(\u,\pressure,\B)$ 
to problem \eqref{eq:MHD-system} satisfy the regularity Assumption \ref{asm:regularity}. Let also~$(\uh, \ph,\Bh)$ be the solution to 
the semidiscrete formulation~\eqref{eq:curl-curl-semidiscrete-nmrobust}, assuming~$\f \in C^0([0, T]; {\cal S})$, where~${\cal S}$ has sufficient regularity for $\Ihcurl{k}(\f)$ to be well defined (see Remark~\ref{rem:f-interp}).
Furthermore, let the mesh Assumptions \ref{asm:mesh} and \ref{asm:mesh:2} hold, and the parameter~$\gammaone$ be sufficiently large.
If 
\[
\u, \B \in L^\infty(0,T;\bW^{1,\infty}(\Omega)),
\]
and the following additional $k$-dependent regularity conditions hold:
\begin{equation*}
\begin{aligned}
& \u, \B \in H^1(0,T;\H^k(\Omega)) \, , \quad 
& & \u, \B \in L^2(0,T;\bW^{k+1,4}(\Omega)) , \\
& \f \in L^2(0,T; \H^{k}(\Omega))\, , \quad & & p \in L^2(0, T; H^{k+1}(\Omega)),
\end{aligned}
\end{equation*}
with~$p=\pressure+|\u|^2/2$, then the following estimate holds for a.e. $t \in (0, T)$:
\begin{equation*}
\begin{aligned}
    & \Norm{\eu}{L^{\infty}(0, t; \L^2(\Omega))}^2 + \Norm{\eB}{L^{\infty}(0, t; \L^2(\Omega))}^2 \\ & \quad + \beta \!\! \int_0^{t} \!\!\Big( \nS \Norm{\eu(\cdot, s)}\dn^2 + \nM \Norm{\curl \eB(\cdot, s)}{\L^2(\Omega)}^2 + \semiNorm{\eu(\cdot, s)}{\uh}^2 + \semiNorm{\eB(\cdot, s)}{\uh}^2 \Big) \ds \\
    & \quad\qquad \lesssim 
    (1 + \nS + \nM) \exp({\cal R}_2 t) h^{2k} \, ,
\end{aligned}
\end{equation*}
where the hidden constant is independent of~$h$, $\nS$,  and $\nM$, but depends, in particular, on the norms of the continuous solution indicated in the assumptions above and the mesh regularity parameters. Moreover, the constant~${\cal R}_2$ depends on~$\Norm{\u}{L^{\infty}(0, T; \bW^{1, \infty}(\Omega))}$ and
~$\Norm{\B}{L^{\infty}(0, T; \bW^{1, \infty}(\Omega))}$.
\end{theorem}

Compared to Theorem~\ref{th:main:1}, Theorem~\ref{th:main:1:bis} eliminates the term~$\nM^{-1}$ on the right-hand side of the estimate, at the cost of the stronger regularity assumptions on~$\B$.  
Similar to the method analyzed in Section \ref{sec:4}, the theoretical estimates in Theorem~\ref{th:main:1:bis} reflect the pressure robustness of the scheme in the sense of Remark~\ref{rem:meth1:presrob}.

\section{A variant toward~\texorpdfstring{$\mathcal{O}(h^{k+1/2})$}{O(k+1/2)} convergence}\label{sec:ideal}
\subsection{Motivation}
The method introduced in Section~\ref{sec:nm-rob} achieves $\nS$- and $\nM$-quasi-robustness, but it comes with several limitations: 
\emph{i)} it yields only~$\mathcal{O}(h^k)$ pre-asymptotic convergence rates even in low-diffusion regimes, \emph{ii)}~requires higher regularity of $\B$ to guarantee convergence, 
and \emph{iii)} introduces a Lagrange multiplier to enforce the divergence-free constraint on~$\Bh$. 
The last two shortcomings arise in other quasi-robust MHD discretizations, as emphasized in Remark~\ref{rem:limitations}. 
Ideally, one would like a quasi-robust method that also applies to nonconvex polyhedral domains, achieves a pre-asymptotic convergence rate of~$\mathcal{O}(h^{k+\frac{1}{2}})$ for sufficiently smooth solutions, and avoids the use of a Lagrange multiplier for the magnetic field.

To the best of the authors' knowledge, no method has been rigorously shown to satisfy all of these properties simultaneously, and even numerical evidence in this direction remains limited. In this section, we therefore introduce a new variant that appears to be a strong candidate for achieving these goals.

The proposed formulation has several favorable features. First, it does not involve jump stabilization of the discrete magnetic field $\B_h$, which may otherwise hinder convergence if~$\B$ lacks sufficient regularity, and it avoids introducing Lagrange multipliers for the magnetic field. 
Second, a partial analysis suggests 
that the method is $\nS$- and~$\nM$-quasi-robust  
and achieves the improved pre-asymptotic rate of~$\mathcal{O}(h^{k+\frac{1}{2}})$. The argument is complete 
except for 
the estimation of two contributions, $M_3^{(4)}$ and $M_5^{(3b)}$, which share the same structure. 
Nevertheless, 
we include the method because of its appealing structure and its very favorable numerical performance, and we present the partial analysis because it provides valuable insights: \emph{i)} it shows that all other error terms can be controlled in a robust and ``optimal'' way, \emph{ii)} it highlights the specific role of the proposed stabilization terms, and \emph{iii)} it aligns with the numerical results reported below, supporting the effectiveness of the approach. Overall, the 
experiments 
provide strong evidence that this variant is a promising approach for constructing quasi-robust MHD discretizations.

\subsection{The method}
For sufficiently smooth vector fields $\w$, $\z$, $\u$, and $\v$, we define the forms
\begin{align}
    \nonumber
    \widetilde{s}_h(\w, \z; \u, \v) & \coloneqq \sum_{f\in\Fho}\widetilde{\gamma}(\w_{|_f}, \z_{|_f})(\jump{\u}, \jump{\v})_f + \sum_{f\in\Fhb}\widetilde{\gamma}(\w_{|_f}, \z_{|_f})(\u\cdot \nOmega, \v\cdot\nOmega)_f, \\
    \nonumber
    \sigma_h(\w, \z; \u, \v) &\coloneqq \sum_{f\in\Fho}h_f^2\widetilde{\gamma}(\w_{|_f}, \z_{|_f})(\jump{\nabla \u},\jump{\nabla \v})_f,\\ 
    \nonumber
    \tau_h(\w, \z; \u, \v) & \coloneqq \sum_{f\in\Fho}h_f^2\widetilde{\gamma}(\w_{|_f}, \z_{|_f})(\jump{\curl\u},\jump{\curl \v})_f,
\end{align}
with~$\widetilde{\gamma}(\w, \z) \coloneqq  \max \{C_S, \Norm{\w}{\L^{\infty}(f)},\Norm{\z}{\L^{\infty}(f)}\}$, together with the associated seminorms
\begin{equation}\label{eq:new:semi}
    \semiNorm{\u}{\widetilde{s}, \w,\z}^2\coloneqq \widetilde{s}_h(\w, \z; \u, \u), \qquad\semiNorm{ \u }{\sigma, \w, \z}^2 \coloneqq \sigma_h(\w, \z; \u, \u), \qquad \semiNorm{\u}{\tau,\w,\z}^2 \coloneqq \tau_h(\w, \z; \u, \u).
\end{equation}
By a slight abuse of notation, we let $\semiNorm{\u}{\cdot, \w} \coloneqq \semiNorm{\u}{\cdot, \w, \w}$. It follows directly that, for all $\w$ and $\z$, it holds~$\semiNorm{\u}{\cdot, \w} \leq \semiNorm{\u}{\cdot, \w, \z}$. We modify the method in~\eqref{eq:curl-curl-semidiscrete}
as follows:
for all~$t \in (0, T]$, find~$(\uh(\cdot, t),\ \ph(\cdot, t), \ \Bh(\cdot, t)) \in \Vcurl \times \bVgrad \times \Vcurl$, with~$\uh$ and~$\Bh$ differentiable in time, such that
\begin{subequations}
\label{eq:kernel-semidiscrete-var}
\begin{alignat}{4}
\label{eq:kernel-semidiscrete-var-1}
(\dpt \uh, \vh)_{\Omega} + \nS a(\uh, \vh) + c(\uh;\uh,\vh) - c(\Bh; \Bh, \vh) & + \nS d_h(\uh, \vh) & & \nonumber \\
 - b(\vh, \ph) + \pts \widetilde{s}_h(\uh, \Bh; \uh, \vh) + \ptsi \sigma_h(\uh, \Bh; \uh, \v_h) & = (\Ihcurl{k} (\f), \vh)_{\Omega} & & \quad \forall \vh \in \Vcurl, \\
b(\uh, \qh) & = 0 & & \quad \forall \qh \in \bVgrad, \\
\label{eq:kernel-semidiscrete-var-2}
(\dpt \Bh, \Ch)_{\Omega} + \nM a(\Bh, \Ch) + c(\Ch; \Bh, \uh)
& & & \nonumber \\
+ \ptt \tau_h(\uh, \Bh; \Bh, \Ch) & = 0 & & \quad \forall \Ch \in \Vcurl, \\
\label{eq:kernel-semidiscrete-var-3}
\uh(\cdot, 0) = \Pi_h^{\curl} \u_0(\cdot) \quad \text{ and } \quad \Bh(\cdot, 0) = \Pi_h^{\curl} \B_0(\cdot) & & & \quad \text{in~$\Omega$},&  
\end{alignat}
\end{subequations}
where~$\pts$, $\ptsi$, and~$\ptt$ are positive real parameters, which are set equal to~$1$ in the forthcoming analysis to simplify the presentation.

\begin{remark}[Discrete divergence-free property]
The same argument in Remark~\ref{rem:divergence-free-method-I} implies that
the discrete divergence-free property for~$\Bh$ is also preserved by
method~\eqref{eq:kernel-semidiscrete-var}, 
as the additional stabilization term~$\tau_h(\uh, \Bh; \B_h, \Ch)$ vanishes for all~$\w\in \L^{\infty}(\Omega)$ and~$\Ch = \nabla \varphi_h$ with~$\varphi_h\in \bVgrad$. As a consequence, no Lagrange multiplier is required for the magnetic variable in this formulation.
\eremk
\end{remark}

\subsection{\emph{A priori} error estimates}
To establish convergence rates of order~${\mathcal O}(h^{ k + \frac{1}{2}})$ in low-diffusion regimes, we assume the following property (which is used only to deal with terms $M_3^{(4)}$ and~$M_5^{(3)}$):
 given $\u \in \H^1_0(\Omega)\cap \WW^{1,\infty}(\Omega)$, it holds
\begin{equation}
    \label{eq:assumption}
    c(\w_h; \w_h, \u) \lesssim \semiNorm{\u}{\WW^{1,\infty}(\Omega)} \Norm{\w_h}{\L^2(\Omega)}^2 \qquad \forall \w_h \in \Xh \, ,
\end{equation}
possibly with the addition on the right-hand-side of terms involving the discrete seminorms in \eqref{eq:new:semi} evaluated on $\w_h$.
Although we have not been able to prove~\eqref{eq:assumption}, we consider it reasonable, at least on convex domains. This property is motivated by the fact that an analogous estimate can be established at the continuous level, as shown in the following lemma.
\begin{lemma}
    Let $\Omega$ be a convex domain, or a domain with $\mathcal{C}^{1,1}$ boundary. Then, for all $\u\in \H^1_0(\Omega)\cap \WW^{1,\infty}(\Omega)$ and $\w\in \Z$, it holds
    \begin{equation*}
        c(\w; \w, \u) \lesssim \semiNorm{\u}{\WW^{1,\infty}(\Omega)} \Norm{\w}{\L^2(\Omega)}^2.
    \end{equation*}
\end{lemma}
\begin{proof}

The assumptions on the domain imply that $\w\in \H^1(\Omega)$, see \cite[Thm. 2.17]{Amrouche}; therefore, all the following manipulations are well defined.
We start by recalling the identity 
\begin{equation}\label{eq:id:X}
\curl (\w \times \u) =
\w\,(\nabla \cdot\u)- \u\,(\nabla \cdot \w)
+ (\u \cdot \nabla) \w - (\w \cdot \nabla) \u \, .
\end{equation}
Elementary algebraic steps, integration by parts (also recalling that~$\u$ has zero trace on~$\partial\Omega$), and the identity in~\eqref{eq:id:X} give
$$
\begin{aligned}
c(\w; \w, \u) &= (\curl \w\times \w, \u)_{\Omega}
= (\curl \w , \w\times \u)_{\Omega} = (\w , \curl(\w\times \u))_{\Omega} \\
& = 
(\w,  \w\,(\nabla \cdot\u))_{\Omega} 
- (\w, \u\,(\nabla \cdot \w))_{\Omega} 
+ (\w,(\u \cdot \nabla) \w)_{\Omega} 
- (\w,(\w \cdot \nabla) \u )_{\Omega} \, .
\end{aligned}
$$
Since~$\nabla\cdot\w=0$ by assumption,
the second
term on the right-hand side vanishes. For the third term, an elementary identity and a standard integration by parts (again recalling that~$\u$ has zero trace on~$\partial\Omega$) yield
$$
(\w,(\u\cdot\nabla)\w)_\Omega= \frac{1}{2}(\u,\nabla(|\w|^2))_\Omega
 = -\frac{1}{2} 
(\w, \w\,(\nabla \cdot \u))_{\Omega} \, .
$$
As a consequence, we obtain
$$
c(\w; \w, \u) = 
\frac{1}{2} (\w, \w\,(\nabla \cdot \u))_{\Omega} 
- (\w,(\w \cdot \nabla) \u )_{\Omega} 
\lesssim \semiNorm{\u}{\WW^{1,\infty}(\Omega)} \Norm{\w}{\L^2(\Omega)}^2 \, ,
$$
and the proof is complete.
\end{proof}

As in Section~\ref{sec:nm-rob}, we use~$\Picurl$ to define the approximants of both~$\u$ and~$\B$, and adopt the same error notation as in \eqref{eq:newerrordefs}.
For the formulation in~\eqref{eq:kernel-semidiscrete-var}, the following analogue of the bound in Proposition~\ref{prop:discrete-errors} holds, for a.e.~$t \in (0, T)$:
\begin{alignat*}{3}
    \frac12 \ddt \Big(\Norm{\ehu}{\L^2(\Omega)}^2 & + \Norm{\ehB}{\L^2(\Omega)}^2\Big) + \nS \beta \Norm{\ehu}\dn^2 + \nM \Norm{\curl \ehB}{\L^2(\Omega)}^2  \\ &   + \semiNorm{\ehu}{\widetilde{s}, \uh, \Bh}^2  + \semiNorm{\ehu}{\sigma, \uh, \Bh}^2 + \semiNorm{\ehB}{\tau, \uh, \Bh}^2  \le \sum_{i = 0}^{9}M_i,
\end{alignat*}
where $M_i$ for $i = 0, \dots, 5$ and $M_7$ are defined as in Proposition~\ref{prop:discrete-errors}, the only difference being that $\Jhcurl\B$ is replaced by~$\Picurl\B$ in $M_4$, and
\begin{align}
    \nonumber
    M_6 & \coloneqq \widetilde{s}_h(\uh, \B_h; \eIu, \ehu), \\
    \nonumber
    M_8 & \coloneqq \sigma_h(\uh, \Bh; \eIu, \ehu), \\
    \nonumber
    M_{9} & \coloneqq \tau_h(\uh, \Bh; \eIB, \ehB) .
\end{align}

We estimate the terms $M_0$, $M_1$, $M_2$ as in Lemmas~\ref{lemma:M0}, \ref{lemma:M1}, \ref{lemma:M2}, respectively, taking into account the additional regularity assumed for~$\B$. For the terms 
$M_3$, $M_4$, $M_5$, and $M_7$, under additional regularity assumptions, we improve on the bounds obtained in Lemmas 
\ref{lemma:M3}, \ref{lemma:M4}, \ref{lemma:M5}, and~\ref{lemma:M7}. In particular, the most involved estimates are those of terms~$M_3^{(1)}$, $M_4^{(4)}$, and~$M_5^{(1)}$. Finally, the terms $M_6$, $M_8$, and~$M_{9}$ are estimated with standard arguments.
We recall that the assumed property~\eqref{eq:assumption} is used below only in the estimate of the terms~$M_3^{(4)}$ and~$M_5^{(3)}$.  

\begin{remark}[Treatment of the term~$M_5^{(3b)}$]
\label{rem:treatment-M5-3}
We highlight that the term~$M_5^{(3b)}$ can no longer be treated robustly with respect to~$\nM$ by following the argument used in Lemma~\ref{lemma:M5-bis}, since the modified stabilization term~$\widetilde{s}_h(\cdot, \cdot; \cdot, \cdot)$ has a different scaling in~$h$ than~$s_h(\cdot; \cdot, \cdot)$. 
Alternatively, we could estimate~$M_5^{(3b)}$ without the need of the assumed property~\eqref{eq:assumption}, at the cost of a factor~$\nM^{-1}$ in the estimates. This, however, would prevent proving the pre-asymptotic error decay~$\mathcal{O}(h^{k+1/2})$ sought in this section.
\eremk
\end{remark}

\begin{lemma}[Estimates of~$M_0$, $M_1$, $M_2$, and $M_7$]\label{lemma:M127-alt}
Let~$\widetilde{\f} = \f + \nabla p$. 
If~$\u,\dpt \u \in L^2(0, T; \H^{k+1}(\Omega))$, $\B,\dpt \B \in 
L^2(0, T; \H^{k+1}(\Omega))$, and~$\widetilde{\f} \in  L^2(0, T; \H^{k+1}(\Omega))$, for any~$\delta > 0$, there hold
\begin{align*}
M_0 & \lesssim h^{2k+2} \big(\semiNorm{\dpt \u}{\H^{k+1}(\Omega)}^2 + \semiNorm{\dpt \B}{\H^{k+1}(\Omega)}^2 \big) + \Norm{\ehu}{\L^2(\Omega)}^2 + \Norm{\ehB}{\L^2(\Omega)}^2,\\
M_1 
&\lesssim \frac{1}{2\delta} h^{
2k} \big(\nS \semiNorm{\u}{\H^{k+1}(\Omega)}^2 + \nM \semiNorm{\B}{\H^{k+1}(\Omega)}^2 \big) + \frac{\delta}{2}\big(\nS\Norm{\ehu}{\dn}^2 + \nM \Norm{\curl \ehB}{\L^2(\Omega)}^2 \big),\\
M_2 
&\lesssim \frac{\nS}{\delta}  h^{
2k} \semiNorm{\u}{\H^{k+1}(\Omega)}^2 + \delta \nS \Norm{\ehu}{\dn}^2,\\
M_7 &\lesssim h^{2k+2} 
\semiNorm{\widetilde{\f}}{\H^{k+1}(\Omega)}^2 +\Norm{\ehu}{\L^2(\Omega)}^2.
\end{align*}
\end{lemma}
\begin{proof}
These estimates can be proven using the same arguments as in~Lemmas~\ref{lemma:M0}, \ref{lemma:M1}, \ref{lemma:M2}, and \ref{lemma:M7}.
\end{proof}

\begin{lemma}[Alternative estimate of~$M_3$]
\label{lemma:M3-bis}
If~$u \in L^2(0, T; \H^{k+1}(\Omega)) \cap L^{\infty}(0, T; \WW^{1, \infty}(\Omega))$,  
 then the following estimate holds:
\begin{alignat*}{3}
\nonumber
 M_3 &\lesssim \bigl(\delta^{-1}\Norm{\u}{\L^{\infty}(\Omega)}^2h^{2k+1} + \semiNorm{\u}{\WW^{1, \infty}(\Omega)}h^{2k+2} \bigr)\semiNorm{\u}{\H^{k+1}(\Omega)}^2 \\ 
 & \quad + \delta\bigl( \semiNorm{\ehu}{\sigma, \uh}^2 +\semiNorm{\ehu}{\widetilde{s}, \uh}^2 \bigr)  + \semiNorm{\u}{\WW^{1, \infty}(\Omega)} \Norm{\ehu}{\L^2(\Omega)}^2.
\end{alignat*}
\end{lemma}
\begin{proof}
We split~$M_3$ as in~\eqref{eq:split-M3}. We start with $M_3^{(1)}$. Using an integration by parts and the identity in~\eqref{eq:id:X}, which can be written as
\begin{equation*}
    \curl( \v\times \w) = (\mathfrak{d}\v)\w -(\mathfrak{d}\w)\v,
\end{equation*}
with $\mathfrak{d}\v\coloneqq \nabla \v - (\nabla\cdot \v) \Id $,
we obtain
\begin{align*}
M_3^{(1)} = c(\eIu; \u, \ehu) &= (\curl \eIu \times\Ihgrad{1} \u, \ehu)_{\Omega} + (\curl \eIu \times (\u - \Ihgrad{1}\u), \ehu)_{\Omega} \\
& = (\curl \eIu, \Ihgrad{1}\u \times \ehu)_{\Omega}   + (\curl \eIu \times (\u - \Ihgrad{1}\u), \ehu)_{\Omega}\\
& = \sum_{K\in\mathcal{T}_h} (\eIu, \curl(\Ihgrad{1}\u \times \ehu))_K + \sum_{K \in \Th}
( \boldsymbol{n}_{K}\times \eIu, \Ihgrad{1}\u \times \ehu)_{\partial K} \\ 
& \quad   + (\curl \eIu \times (\u - \Ihgrad{1}\u), \ehu)_{\Omega}
\\
& =\sum_{K\in\mathcal{T}_h} (\eIu, (\mathfrak{d} \Ihgrad{1}\u)\ehu )_K -\sum_{K\in\mathcal{T}_h}  (\eIu, (\mathfrak{d}\ehu )\Ihgrad{1}\u)_K  \\
&\quad  + \sum_{K \in \Th}
( \boldsymbol{n}_{K}\times \eIu, \Ihgrad{1}\u \times \ehu)_{\partial K}  + (\curl \eIu \times (\u - \Ihgrad{1}\u), \ehu)_{\Omega}
\\
& \eqqcolon M_3^{(1a)} + M_3^{(1b)} + M_3^{(1c)} + M_3^{(1d)}.
\end{align*}
We bound $M_3^{(1a)}$ using the stability of $\Ihgrad{1}$ in~$\WW^{1, \infty}(\Omega)$ and the H\"older inequality as follows:
\begin{align*}
    M_3^{(1a)}  &\lesssim\lVert \eIu\rVert_{\L^2(\Omega)} \semiNorm{\u}{\WW^{1,\infty}(\Omega)} \lVert \ehu\rVert_{\L^2(\Omega)} \\
    & \lesssim \semiNorm{\u}{\WW^{1, \infty}(\Omega)}h^{2k+2}\semiNorm{\u}{\H^{k+1}(\Omega)}^2 +\semiNorm{\u}{\WW^{1,\infty}}\lVert \ehu \rVert^2_{\L^2(\Omega)} .
\end{align*}
To bound $M_3^{(1b)}$, we note that $(\mathfrak{d} \ehu)\Ihgrad{1}\u$ is a (possibly discontinuous) 
piecewise polynomial of degree~$k$ in~$\Omega$ (the differential operators in the definition of~$\mathfrak{d}$ are taken elementwise). Then, using the $L^2(\Omega)$-orthogonality of $\eIu$ to $\Vcurl$, 
bound~\eqref{eq:Jav_prop}, the stability of~$\Ihgrad{1}{}$ in the~$\L^{\infty}(\Omega)$ norm, and the approximation properties in Lemma~\ref{lemma:approx-Picurl} of~$\Picurl$, we obtain
\begin{align*}
    -M_3^{(1b)}&=\sum_{K\in\mathcal{T}_h}(\eIu, (\mathfrak{d}\ehu)\Ihgrad{1}\u )_K = (\eIu,(\mathfrak{d} \ehu)\Ihgrad{1}\u-\Jav((\mathfrak{d} \ehu)\Ihgrad{1}\u))_{\Omega}\\
    & \lesssim \lVert \eIu\rVert_{\L^2(\Omega)} \Big(\sum_{f\in\Fho}h_f\lVert \jump{(\mathfrak{d} \ehu)\Ihgrad{1}\u)}\times \nf\rVert_{\L^2(f)}^2\Big)^{1/2}\\ 
    & \lesssim \frac{1}{2\delta}h^{-1}\Norm{\u}{\L^\infty(\Omega)}^2\lVert \eIu\rVert_{\L^2(\Omega)}^2 + \frac{\delta}{2C_S} \semiNorm{\ehu}{\sigma, \uh}^2\\
    & \lesssim \frac{1}{2\delta}h^{2k+1}\Norm{\u}{\L^\infty(\Omega)}^2\semiNorm{ \u}{\H^{k+1}(\Omega)}^2 + \frac{\delta}{2C_S}\semiNorm{ \ehu}{\sigma, \uh}^2.
\end{align*}
As for~$M_3^{(1c)}$, we use similar arguments combined with the 
trace inequality to get
\begin{align*}
    M_3^{(1c)} &= \sum_{K\in\Th} (\boldsymbol{n}_{K}\times \eIu, \Ihgrad{1}\u \times \ehu)_{\partial K}\\ & \leq \sum_{f \in \Fho}\Norm{\Ihgrad{1}\u}{\L^\infty(f)} \Norm{\eIu\times\boldsymbol{n}_{F}}{\L^2(f)} \Norm{\jump{\ehu}}{\L^2(f)} \\
    &
    \lesssim \frac{1}{2\delta}\sum_{f\in\Fho}\Norm{\u}{\L^{\infty}(f)}^2\lVert \eIu\times\boldsymbol{n}_{F} \rVert^2_{\L^{2}(f)} + \frac{\delta}{2}\sum_{f\in \Fho}\lVert \jump{\ehu}\rVert_{\L^2(f)}^2\\
    & \lesssim \frac{1}{2\delta} \Norm{\u}{\L^\infty(\Omega)}^2 \sum_{K\in \Th} \left( \hK^{-1}\lVert \eIu\rVert_{\L^2(K)}^2 + \hK\semiNorm{\eIu}{\H^1(K)}^2
    \right)+ \frac{ \delta}{2}\sum_{f\in\Fho}  \lVert \jump{\ehu} \rVert_{\L^2(f)}^2\\
    & \lesssim  \frac{1}{2\delta} \Norm{\u}{\L^\infty(\Omega)}^2 h^{2k + 1}\semiNorm{\u}{\H^{k+1}(\Omega)}^2 + \frac{ \delta}{2}\sum_{f\in\Fho} \lVert \jump{\ehu} \rVert_{\L^2(f)}^2\\
    & \lesssim  \frac{1}{2\delta}\Norm{\u}{\L^\infty(\Omega)}^2 h^{2k + 1}\semiNorm{\u}{\H^{k+1}(\Omega)}^2 + \frac{\delta}{2C_S}\semiNorm{\ehu}{\widetilde{s},\uh}^2.
\end{align*}
Finally, we bound $M_3^{(1d)}$ using the H\"older inequality, a polynomial inverse estimate, and the 
approximation properties of~$\Ihgrad{1}$ in~$\L^{\infty}(\Omega)$ as follows:
\begin{align*}
    M_3^{(1d)} &=  (\curl \eIu \times (\u - \Ihgrad{1}\u), \ehu)_{\Omega} \\
    & \lesssim \Norm{\curl \eIu}{\L^2(\Omega)} h \semiNorm{ \u}{\boldsymbol{W}^{1,\infty}(\Omega)}\Norm{\ehu}{\L^2(\Omega)}\\
    & \lesssim \semiNorm{ \u}{\boldsymbol{W}^{1,\infty}(\Omega)} h^{2k + 2} \semiNorm{\u}{\H^{k+1}(\Omega)}^2 + \semiNorm{ \u}{\boldsymbol{W}^{1,\infty}(\Omega)} \Norm{ \ehu}{\L^2(\Omega)}^2. 
\end{align*}
Collecting the bounds for $M_3^{(1a)}$, $M_3^{(1b)}$, $M_3^{(1c)}$, and $M_3^{(1d)}$ yields 
the following estimate of~$M_3^{(1)}$:
\begin{equation}
\label{eq:M3-1-alt}
    \begin{split}
        M_3^{(1)}& \lesssim \bigl(\delta^{-1}\Norm{\u}{\L^{\infty}(\Omega)}^2h^{2k+1} + \semiNorm{\u}{\WW^{1, \infty}(\Omega)}h^{2k+2} \bigr)\semiNorm{\u}{\H^{k+1}(\Omega)}^2 \\ 
        & \quad + \frac{\delta}{2C_S} \bigl(\semiNorm{\ehu}{\sigma, \uh}^2 +\semiNorm{\ehu}{\widetilde{s}, \uh}^2 \bigr)  + \semiNorm{\u}{\WW^{1, \infty}(\Omega)} \Norm{\ehu}{\L^2(\Omega)}^2.
    \end{split}
\end{equation}
The terms $M_3^{(2)}$ and $M_3^{(3)}$ can be estimated as in the proof of Lemma~\ref{lemma:M3} as
\begin{align}
\label{eq:M3-2-alt}
M_3^{(2)} & \lesssim \semiNorm{\u}{\WW^{1, \infty}(\Omega)} \big(h^{2k+2} \semiNorm{\u}{\H^{k+1}(\Omega)}^2 + \Norm{\ehu}{\L^2(\Omega)}^2 \big),\\
\label{eq:M3-3-alt}
M_3^{(3)} & \lesssim \semiNorm{\u}{\WW^{1, \infty}(\Omega)} \Norm{\ehu}{\L^2(\Omega)}^2,
\end{align}
while, for $M_3^{(4)}$, we use the 
assumed property~\eqref{eq:assumption} to get
\begin{equation}
\label{eq:M3-4-alt}
M_3^{(4)} \lesssim\semiNorm{\u}{\WW^{1, \infty}(\Omega)} \Norm{\ehu}{\L^2(\Omega)}^2.
\end{equation}
The result follows by combining the splitting in~\eqref{eq:split-M3} with estimates~\eqref{eq:M3-1-alt}, \eqref{eq:M3-2-alt}, \eqref{eq:M3-3-alt}, and \eqref{eq:M3-4-alt}.
\end{proof}
\begin{lemma}[Estimate of~$M_4$]
If~$\B \in L^2(0, T; \H^{k+1}(\Omega)) \cap L^{\infty}(0, T; \WW^{1, \infty}(\Omega))$, the following estimate holds:
\begin{equation*}
\begin{split}
M_4 &\lesssim \big(\delta^{-1} \Norm{\B}{\L^{\infty}(\Omega)}^2 h^{2k+1} + \semiNorm{\B}{\WW^{1, \infty}(\Omega)}h^{2k+2} \big)\semiNorm{\B}{\H^{k+1}(\Omega)}^2 + \delta \big(\semiNorm{\ehu}{\sigma, \Bh}^2 + \semiNorm{\ehu}{\widetilde{s}, \Bh}^2 \big) \\
& \quad + 
\semiNorm{\B}{\WW^{1, \infty}(\Omega)}  \Norm{\ehu}{\L^2(\Omega)}^2 + 
\semiNorm{\B}{\WW^{1, \infty}(\Omega)}  \Norm{\ehB}{\L^2(\Omega)}^2. 
\end{split}
\end{equation*}
\end{lemma}
\begin{proof}
We split~$M_4$ as follows:
\begin{alignat}{3}
\nonumber
M_4 & = c(\Picurl\B; \Bh, \ehu) - c(\B; \B, \ehu) \\
\nonumber
& = c(\Picurl \B; \ehB, \ehu) - c(\Picurl \B; \eIB, \ehu) - c(\eIB; \B, \ehu) \\
\nonumber
& =c(\Picurl \B; \ehB, \ehu) - c(\Picurl \B; \eIB, \ehu) - c(\eIB; \B - \Ihgrad{1}\B, \ehu) - c(\eIB; \Ihgrad{1}\B, \ehu)\\
\label{eq:split-M4_stab}
& =: M_4^{(1)} + M_4^{(2)} + M_4^{(3)} + M_4^{(4)}.
\end{alignat}
Using the stability properties of~$\Picurl$ from Lemma~\ref{lemma:picurl_linfty_stab}, together with the H\"older and the Young inequalities, we obtain
\begin{alignat}{3}
\nonumber
M_4^{(1)}  = c(\Picurl \B; \ehB, \ehu) & = \big(\curl (\Picurl \B) \times \ehB, \ehu \big)_{\Omega} \\
\nonumber
& \le \Norm{\curl (\Picurl\B)}{\L^{\infty}(\Omega)} \Norm{\ehB}{\L^2(\Omega)} \Norm{\ehu}{\L^2(\Omega)} \\
\label{eq:M4-1-bis}
& \lesssim \semiNorm{\B}{\WW^{1,\infty}(\Omega)} \Big(\Norm{\ehB}{\L^2(\Omega)}^2 + \Norm{\ehu}{\L^2(\Omega)}^2\Big).
\end{alignat}
Proceeding similarly and using the approximation properties in Lemma~\ref{lemma:approx-Picurl} of~$\Picurl$, we get 
\begin{alignat}{3}
\nonumber
M_4^{(2)} = -c(\Picurl \B; \eIB, \ehu) & = -\big(\curl (\Picurl\B) \times \eIB, \ehu \big)_{\Omega} \\
\nonumber
& \lesssim \semiNorm{\B}{\WW^{1,\infty}(\Omega)} \Norm{\eIB}{\L^2(\Omega)} \Norm{\ehu}{\L^2(\Omega)} \\
\label{eq:M4-2-bis}
& \lesssim  h^{2k+2} \semiNorm{\B}{\WW^{1,\infty}(\Omega)}  \semiNorm{\B}{\H^{k+1}(\Omega)}^2 + 
\semiNorm{\B}{\WW^{1,\infty}(\Omega)}
\Norm{\ehu}{\L^2(\Omega)}^2 .
\end{alignat}
As for $M_4^{(3)}$, we use the H\"older inequality, the approximation properties of $\Ihgrad{1}$ in the $\L^{\infty}(\Omega)$ norm, a polynomial inverse estimate, and the approximation properties of~$\Picurl$ to obtain
\begin{alignat}{3}
\nonumber
M_4^{(3)}  =  -c(\eIB; \B-\Ihgrad{1}{\B}, \ehu) & = -\big((\curl \eIB)  \times (\B - \Ihgrad{1}\B), \ehu\big)_{\Omega} \\
\nonumber
& \lesssim h\semiNorm{\B}{\WW^{1,\infty}(\Omega)} \Norm{\curl \eIB}{\L^2(\Omega)} \Norm{\ehu}{\L^2(\Omega)} \\
\label{eq:M4-3-bis}
& \lesssim h^{2k+2} \semiNorm{\B}{\WW^{1,\infty}(\Omega)}
\semiNorm{\B}{\H^{k+1}(\Omega)}^2
+ \semiNorm{\B}{\WW^{1,\infty}(\Omega)} \Norm{\ehu}{\L^2(\Omega)}^2 . 
\end{alignat}
Finally, we estimate the term~$M_4^{(4)}$ by following the same approach used for~$M_3^{(1a)}+M_3^{(1b)}+M_3^{(1c)}$ in Lemma~\ref{lemma:M3-bis}. The additional boundary term, which arises here since~$\B$ (and hence~$\Ihgrad{1}\B$) does not vanish on~$\partial \Omega$, does not change the estimate. Specifically,  
\begin{equation}\label{eq:M4-4}
    \begin{split}
        M_4^{(4)} & = - c(\eIB; \Ihgrad{1}{\B}, \ehu) = -\bigl(\curl \eIB, (\Ihgrad{1}\B )\times\ehu\bigr)_{\Omega}\\
        & \lesssim \bigl(\delta^{-1}\Norm{\B}{\L^{\infty}(\Omega)}^2h^{2k+1} + \semiNorm{\B}{\WW^{1, \infty}(\Omega)}h^{2k+2} \bigr)\semiNorm{\B}{\H^{k+1}(\Omega)}^2 \\ 
        & \quad + \frac{\delta}{2C_S} \bigl(\semiNorm{\ehu}{\sigma, \Bh}^2 +\semiNorm{\ehu}{\widetilde{s}, \Bh}^2 \bigr)  + \semiNorm{\B}{\WW^{1, \infty}(\Omega)} \Norm{\ehu}{\L^2(\Omega)}^2.
    \end{split}
\end{equation}
 The proof concludes combining estimates~\eqref{eq:M4-1-bis}--\eqref{eq:M4-4} with the splitting~\eqref{eq:split-M4_stab}.
\end{proof}

\begin{lemma}[Alternative estimate of~$M_5$]
If both $\u$ and $\B$ are in $L^2(0, T; \H^{k+1}(\Omega)) \cap L^{\infty}(0, T; \WW^{1, \infty}(\Omega))$,
the following estimate holds for any~$\delta > 0$:
\begin{equation*}
\begin{split}
M_5 &\lesssim \bigl( \delta^{-1} \Norm{\B}{\L^\infty(\Omega)}^2h^{2k+ 1} + \semiNorm{\B}{\WW^{1, \infty}(\Omega)} h^{2k+2}\bigr)\semiNorm{\u}{\H^{k+1}(\Omega)}^2 + \delta\semiNorm{\ehB}{\tau, \uh, \Bh}^ 2 
\\ 
& \quad + \big(\delta^{-1}\Norm{\u}{\L^{\infty}(\Omega)}^2h^{2k+1} +\semiNorm{\u}{\WW^{1,\infty}(\Omega)} h^{2k+1}\big)\semiNorm{\B}{\H^{k+1}(\Omega)}^2 \\
& \quad +  (\semiNorm{\u}{\WW^{1,\infty}(\Omega)} + \semiNorm{\B}{\WW^{1,\infty}(\Omega)})\Norm{\ehB}{\L^2(\Omega)}^2.
\end{split}
\end{equation*}
\end{lemma}
\begin{proof}
We split $M_5$ as in \eqref{eq:M5-split}. Then, standard manipulations yield
\begin{alignat}{3}
\nonumber
M_5^{(1)} = c(\ehB; \B, \eIu) & = \big((\curl \ehB) \times \B, \eIu \big)_{\Omega} \\
\nonumber
& = \big((\curl \ehB) \times (\B - \Ihgrad{1}\B), \eIu \big)_{\Omega} + \big((\curl \ehB) \times \Ihgrad{1}\B, \eIu \big)_{\Omega} \\
\nonumber
& \eqqcolon M_5^{(1a)} + M_5^{(1b)}.
\end{alignat}
For the term~$M_5^{(1a)}$, the H\"older inequality, a polynomial inverse estimate, and the approximation properties of $\Ihgrad{1}$ lead to
\begin{alignat}{3}
\nonumber
    M_5^{(1a)} & \lesssim h^{-1}\lVert \ehB\rVert_{\L^2(\Omega)} h \semiNorm{\B}{\WW^{1,\infty}(\Omega)} \lVert \eIu \rVert_{\L^2(\Omega)} \\
    \label{eq:M5-1a}
    & \lesssim \semiNorm{ \B}{\WW^{1,\infty}(\Omega)} \left( \lVert \ehB \rVert_{\L^2(\Omega)}^2  + h^{2k+2}\semiNorm{\u}{\H^{k+1}(\Omega)}^2\right).
\end{alignat}
For~$M_5^{(1b)}$, we note that $(\curl \ehB) \times \Ihgrad{1}\B$ is a (possibly discontinuous)
piecewise polynomial of degree~$k$. Therefore, the~$L^2(\Omega)$-orthogonality of $\eIu$ to~$\Vcurl$, together with the H\"older inequality, bound~\eqref{eq:Jav_prop}, and the approximation properties of~$\Picurl$, gives
\begin{align*}
    M_5^{(1b)} & = \big( (\curl \ehB) \times \Ihgrad{1}\B - \Jhav((\curl \ehB) \times \Ihgrad{1}\B), \eIu \big) \\ 
    &\lesssim  \sqrt{\sum_{f\in\Fho} h\lVert \Ihgrad{1}\B\rVert_{\L^\infty(f)}^2 \lVert \jump{\curl \ehB} \rVert_{\L^2(f)}^2 } \lVert \eIu\rVert_{\L^2(\Omega)}\\
    & \lesssim \frac{\delta}{2}\sum_{f\in\Fho} h^2\lVert \jump{\curl \ehB} \rVert_{\L^2(f)}^2 + \frac{1}{2\delta}\Norm{\B}{\L^{\infty}(\Omega)}^2h^{-1}\lVert \eIu\rVert^2_{\L^2(\Omega)}\\
    & \lesssim \frac{\delta}{2 C_S}\semiNorm{ \ehB}{\tau, \Bh}^2 +  \frac{1}{2\delta}\Norm{\B}{\L^{\infty}(\Omega)}^2h^{2k+1}
    |\u|_{\H^{k+1}(\Omega)}^2.
\end{align*}
Combining these two estimates, we get 
\begin{equation}
\label{eq:M5-1-alt}
    \begin{split}
        M_5^{(1)} &\lesssim \bigl( \delta^{-1} \Norm{\B}{\L^\infty(\Omega)}^2h^{2k+ 1} + \semiNorm{\B}{\WW^{1, \infty}(\Omega)} h^{2k+2}\bigr)\semiNorm{\u}{\H^{k+1}(\Omega)}^2 \\
        & \quad +\delta \semiNorm{\ehB}{\tau, \B_h}^2  +  
        \semiNorm{\B}{\WW^{1, \infty}(\Omega)} \Norm{\ehB}{\L^2(\Omega)}^2.
    \end{split}
\end{equation}
We now focus on the term~$M_5^{(2)}$, which can be split as
\begin{equation*}
\begin{split}
M_5^{(2)} &= c(\ehB;\eIB, \Picurl \u - \u) + c(\ehB;\eIB, \u - \Ihgrad{1}\u) + c(\ehB; \eIB, \Ihgrad{1}\u)\\
& \eqqcolon M_5^{(2a)} + M_5^{(2b)} + M_5^{(2c)}.
\end{split}
\end{equation*}
The terms $M_5^{(2a)}$ and $M_5^{(2b)}$ can be estimated analogously to~$M_5^{(1a)}$ in~\eqref{eq:M5-1a}, thus obtaining
\begin{equation*}
    M_5^{(2a)} + M_5^{(2b)} \lesssim \semiNorm{\u}{\WW^{1,\infty}(\Omega)} h^{2k+2}\semiNorm{ \B}{\H^{k+1}(\Omega)}^2 + \semiNorm{\u}{\WW^{1,\infty}(\Omega)} \Norm{\ehB}{\L^2(\Omega)}^2.
\end{equation*}
Finally, for~$M_5^{(2c)}$ we adopt the same technique as in $M_5^{(1b)}$:
\begin{equation*}
    M_5^{(2c)} \lesssim \frac{\delta}{2C_S}\semiNorm{\ehB}{\tau, \uh}^2 + \frac{1}{2 \delta} \Norm{\u}{\L^{\infty}(\Omega)}^2 h^{2k+1} \semiNorm{\B}{\H^{k+1}(\Omega)}^2.
\end{equation*}
We then conclude that 
\begin{equation}
\label{eq:M5-2-alt}
\begin{split}
M_5^{(2)} &\lesssim \frac{\delta}{2C_S}\semiNorm{\ehB}{\tau, \uh}^2 + \frac{1}{2 \delta}\Norm{\u}{\L^{\infty}(\Omega)}^2h^{2k+1}\semiNorm{\B}{\H^{k+1}(\Omega)}^2 \\
& \quad + h^{2k+2} \semiNorm{\u}{\WW^{1,\infty}(\Omega)} \semiNorm{ \B}{\H^{k+1}(\Omega)}^2 + \semiNorm{\u}{\WW^{1,\infty}(\Omega)} \Norm{\ehB}{\L^2(\Omega)}^2.
\end{split}
\end{equation}

We split $M_5^{(3)}$ as 
\begin{equation*}
M_5^{(3)} = -c(\ehB; \ehB, \Picurl \u - \u) - c(\ehB; \ehB, \u)  \eqqcolon M_5^{(3a)} + M_5^{(3b)}.
\end{equation*}
The term $M_5^{(3a)}$ can be bounded again as $M_5^{(1a)}$ in~\eqref{eq:M5-1a}, so that
\begin{equation*}
    M_5^{(3a)} \lesssim \semiNorm{\u}{\WW^{1,\infty}(\Omega)}\Norm{\ehB}{\L^2(\Omega)}^2.
\end{equation*}
The same bound also~holds for $M_5^{(3b)}$ thanks to the assumed property 
\eqref{eq:assumption} (we also recall Remark~\ref{rem:treatment-M5-3})). 
Consequently, we conclude 
that 
\begin{equation}
\label{eq:M5-3-alt}
    M_5^{(3)}\lesssim \semiNorm{\u}{\WW^{1,\infty}(\Omega)}\Norm{\ehB}{\L^2(\Omega)}^2.
\end{equation}
Combining the splitting~\eqref{eq:M5-split} with estimates~\eqref{eq:M5-1-alt}, \eqref{eq:M5-2-alt}, and \eqref{eq:M5-3-alt}, we obtain the result.
\end{proof}

\begin{lemma}[Estimates of~$M_6$, $M_8$, and $M_{9}$]
If~$\u, \B \in L^2(0, T; \WW^{k+1,4}(\Omega))$, the following estimates hold for all~$\delta > 0$:
\begin{align*}
    M_6  &\lesssim \delta^{-1} h^{2k+1} \Norm{\uh}{\L^2(\Omega)} \semiNorm{\u}{\WW^{k+1,4}(\Omega)}^2 + \delta \semiNorm{\ehu}{\widetilde{s}, \uh}^2,\\
    M_8 &\lesssim \delta^{-1} h^{2k+1} \Norm{\uh}{\L^2(\Omega)} \semiNorm{\u}{\WW^{k+1,4}(\Omega)}^2 + \delta \semiNorm{\ehu}{\sigma, \uh}^2,\\
    M_9  &\lesssim \delta^{-1} h^{2k+1} \Norm{\Bh}{\L^2(\Omega)} \semiNorm{\B}{\WW^{k+1,4}(\Omega)}^2 + \delta \semiNorm{\ehB}{\tau, \uh}^2.
\end{align*}
\end{lemma}
\begin{proof}
Since the new term $M_6$ is identical to that in Proposition \ref{prop:discrete-errors} up to a scaling factor $h$, applying the same argument as in  Lemma~\ref{lemma:M6}, we immediately obtain the first bound above. The remaining two terms are handled similarly. 
The underlying idea is that, 
compared to the term~$M_6$, the presence of a higher~$h$-scaling factor in~$M_8$ and~$M_9$ exactly compensates the first-order derivatives appearing in such terms.
\end{proof}

Once the above modified estimates are derived (recalling the assumed property~\eqref{eq:assumption}), we can apply the same arguments as in the proof of Theorem~\ref{th:main:1}, yielding the following result.  

\begin{theorem}[\emph{A priori} error estimates]\label{th:main:1:bis2}  
Let the solution~$(\u,\pressure,\B)$ 
to problem \eqref{eq:MHD-system} satisfy the regularity Assumption \ref{asm:regularity}. Let also~$(\uh, \ph,\Bh)$ be the solution to 
the semidiscrete formulation~\eqref{eq:curl-curl-semidiscrete-nmrobust}, assuming~$\f \in C^0([0, T]; {\cal S})$, where~${\cal S}$ has sufficient regularity for $\Ihcurl{k}(\f)$ to be well defined (see Remark~\ref{rem:f-interp}).
Furthermore, let the mesh Assumptions \ref{asm:mesh} and \ref{asm:mesh:2} hold, and the parameter~$\gammaone$ be sufficiently large.
If 
\[
\u, \B \in L^\infty(0,T;\bW^{1,\infty}(\Omega)),
\]
and the following additional $k$-dependent regularity conditions hold:
\begin{equation*}
\begin{aligned}
& \u, \B \in H^1(0,T;\H^{k+1}(\Omega)) \, , \quad 
& & \u, \B \in L^2(0,T;\bW^{k+1,4}(\Omega)) , \\
& \f \in L^2(0,T; \H^{k+1}(\Omega))\, , \quad & & p \in L^2(0, T; H^{k+2}(\Omega)),
\end{aligned}
\end{equation*}
with~$p=\pressure+|\u|^2/2$, then the following estimate holds for a.e. $t \in (0, T)$:
\begin{equation*}
\begin{aligned}
    & \Norm{\eu}{L^{\infty}(0, t; \L^2(\Omega))}^2 + \Norm{\eB}{L^{\infty}(0, t; \L^2(\Omega))}^2 + \beta \!\! \int_0^{t} \!\!\Big( \nS \Norm{\eu(\cdot, s)}\dn^2 + \nM \Norm{\curl \eB(\cdot, s)}{\L^2(\Omega)}^2 \Big) \\ &
    + \!\! \int_0^{t} \!\!\Big( \semiNorm{\eu}{\widetilde{s}, \uh, \Bh}^2  + \semiNorm{\eu}{\sigma, \uh, \Bh}^2 + \semiNorm{\eB}{\tau, \uh, \Bh}^2 \Big) 
    \lesssim (h + \nS + \nM) \exp({\cal R}_2 t) h^{2k} \, ,
\end{aligned}
\end{equation*}
where the hidden constant is independent of~$h$, $\nS$,  and $\nM$, but depends, in particular, on the norms of the continuous solution indicated in the assumptions above and the mesh regularity parameters. Moreover, the constant~${\cal R}_2$ depends on~$\Norm{\u}{L^{\infty}(0, T; \bW^{1, \infty}(\Omega))}$ and
~$\Norm{\B}{L^{\infty}(0, T; \bW^{1, \infty}(\Omega))}$.
\end{theorem}
Also for the present method, the theoretical estimates in Theorem~\ref{th:main:1:bis2} reflect the pressure robustness of the scheme, in the sense of Remark~\ref{rem:meth1:presrob}.

\section{Numerical experiments}\label{sec:num}
In this section, we validate the theoretical results for the three proposed methods: 
\[
\begin{split}
\texttt{Method 1:} & \ \text{presented in Section~\ref{sec:method1}, formulation~\eqref{eq:curl-curl-semidiscrete}}; \\
\texttt{Method 2:}& \ \text{presented in Section~\ref{sec:nm-rob}, formulation~\eqref{eq:curl-curl-semidiscrete-nmrobust}};\\
\texttt{Method 3:}& \ \text{presented in Section~\ref{sec:ideal}, formulation~\eqref{eq:kernel-semidiscrete-var}}.
\end{split}
\]
We restrict ourselves to two-dimensional problems in space. 
First, in Sections~\ref{sec:smooth-sol} and~\ref{sec:nonsmooth-sol}, we carry out some convergence tests, where we measure the error in the natural time-discrete version of the following ``total" energy norm:
\begin{equation}\label{eq:norms:err}
\begin{split}
    \Norm{(\u, \B)}{\mathrm{tot}}^2&  \coloneqq \Norm{\u}{L^\infty(0,T;\L^2(\Omega))}^2 + \Norm{\B}{L^\infty(0,T;\L^2(\Omega))}^2 \\
    &+ \int_0^T \left(\nS \Norm{\u(\cdot, t)}{\#}^2 + \nM\Norm{ \curl \B(\cdot, t)}{\L^2(\Omega)}^2 + \semiNorm{(\u, \B)}{\mathrm{stab}, \uh, \Bh}^2\right)\,\mathrm{d}t,
    \end{split}
\end{equation}
with $\semiNorm{(\u, \B)}{\mathrm{stab}, \u_h, \B_h}$ being the seminorm induced by the corresponding stabilization terms, namely,
\begin{equation*}
\semiNorm{(\u, \B)}{\mathrm{stab}, \u_h, \B_h}^2  \coloneqq \begin{cases} 
\semiNorm{\u}{\uh}^2, \qquad &\text{ for {\texttt{Method~1}},
} \\ \semiNorm{\u}{\uh}^2 + \semiNorm{\B}{\uh}^2, \qquad &\text{ for {\texttt{Method~2}},
}  \\ 
\semiNorm{\u}{\widetilde{s}, \uh, \Bh}^2 + \semiNorm{\u}{\sigma, \uh, \Bh}^2 + \semiNorm{\B}{\tau, \uh, \Bh}^2, \qquad &\text{ for {\texttt{Method~3}}.
} 
\end{cases}
\end{equation*}
Then, in Sections~\ref{sec:magenetic-field-loop} and~\ref{sec:Orszag-Tang-vortex}, we perform some challenging tests in which $\nS$- and $\nM$-quasi-robustness is crucial to avoid unphysical oscillations. In such a numerical comparison, we also include the \emph{unstabilized} version obtained by removing the stabilization term~$\sh(\uh; \uh, \vh)$ in \texttt{Method~1}.

For time discretization, we have used the implicit midpoint rule, which is known to  exactly conserve quadratic invariants such as energy and cross helicity in the unstabilized inviscid limit. The method parameters are chosen as $C_S = \pus = \pdsb =  0.1$, $\alpha = 10$, and $\ptsi = \ptt = 0.025$ for all numerical experiments.

The code is implemented in the open-source finite element library NGSolve \cite{ngSolve} (\url{https://ngsolve.org/}), and is freely available at \url{https://github.com/EnricoZampa/RobustMHD}.

\subsection{Convergence for a smooth solution} 
\label{sec:smooth-sol}
We 
first consider a manufactured problem on the unit square~$\Omega = (0,1)^2$ with final time~$T=1$, 
augmented by an additional source term~$\boldsymbol{g}$ on the right-hand side of~\eqref{eq:MHD-system-curl-2}. 
The exact solution given by
\begin{equation}
\label{eq:smooth-sol}
    \begin{split}
        \u(x, y, t) & = -2\pi e^{-\frac{1}{2}t}(\sin(\pi x)^2 \sin(\pi y) \cos(\pi y), -\sin(\pi x)\cos(\pi x)\sin(\pi y)^2 ), \\\
        p(x, y, t) &= -e^{-\frac{1}{2}t}\sin(2\pi x)\cos(2 \pi y), \\
        \B(x, y, t) & = -\pi e^{-\frac{1}{2}t} ( \sin(\pi x) \cos(\pi y), -\cos(\pi x) \sin( \pi y)).
        \end{split}
\end{equation}

Since the implicit midpoint rule is second-order accurate, for a space discretization with polynomial degree~$k$, the time-step is chosen as $\Delta t = \frac{1}{10}h ^{\frac{k +1}{2}}$ to balance the spatial and temporal discretization errors.

\paragraph{Method 1.}
We study the convergence errors for~$k \in \{1,2\}$, $\nS \in \{1, 10^{-4}, 10^{-8}\}$, and $\nM = 1$ (recalling that this method is only $\nS$-quasi-robust). The results, reported in Figure \ref{fig:method1}, confirm the convergence rates~$\mathcal{O}(h^k)$ predicted by Theorem~\ref{th:main:1}, for all considered values of~$\nS$.
\begin{figure}[!ht]
\centering

\begin{subfigure}{0.49\textwidth}
    \centering
    \includegraphics[width=\linewidth, trim=0 0 0 0, clip]{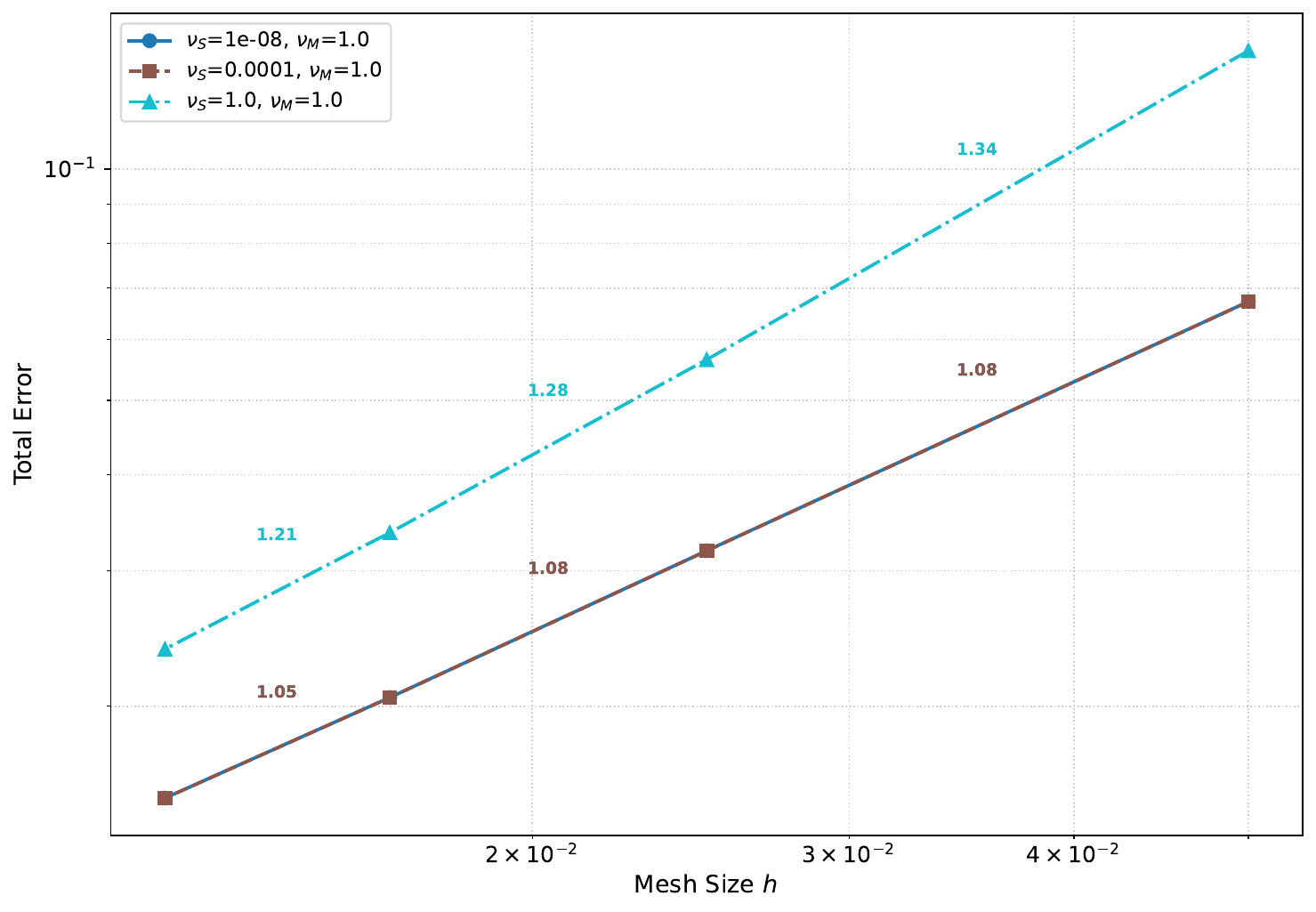}
    \caption{$k = 1$}
\end{subfigure}
\hfill
\begin{subfigure}{0.49\textwidth}
    \centering
    \includegraphics[width=\linewidth, trim=0 0 0 0, clip]{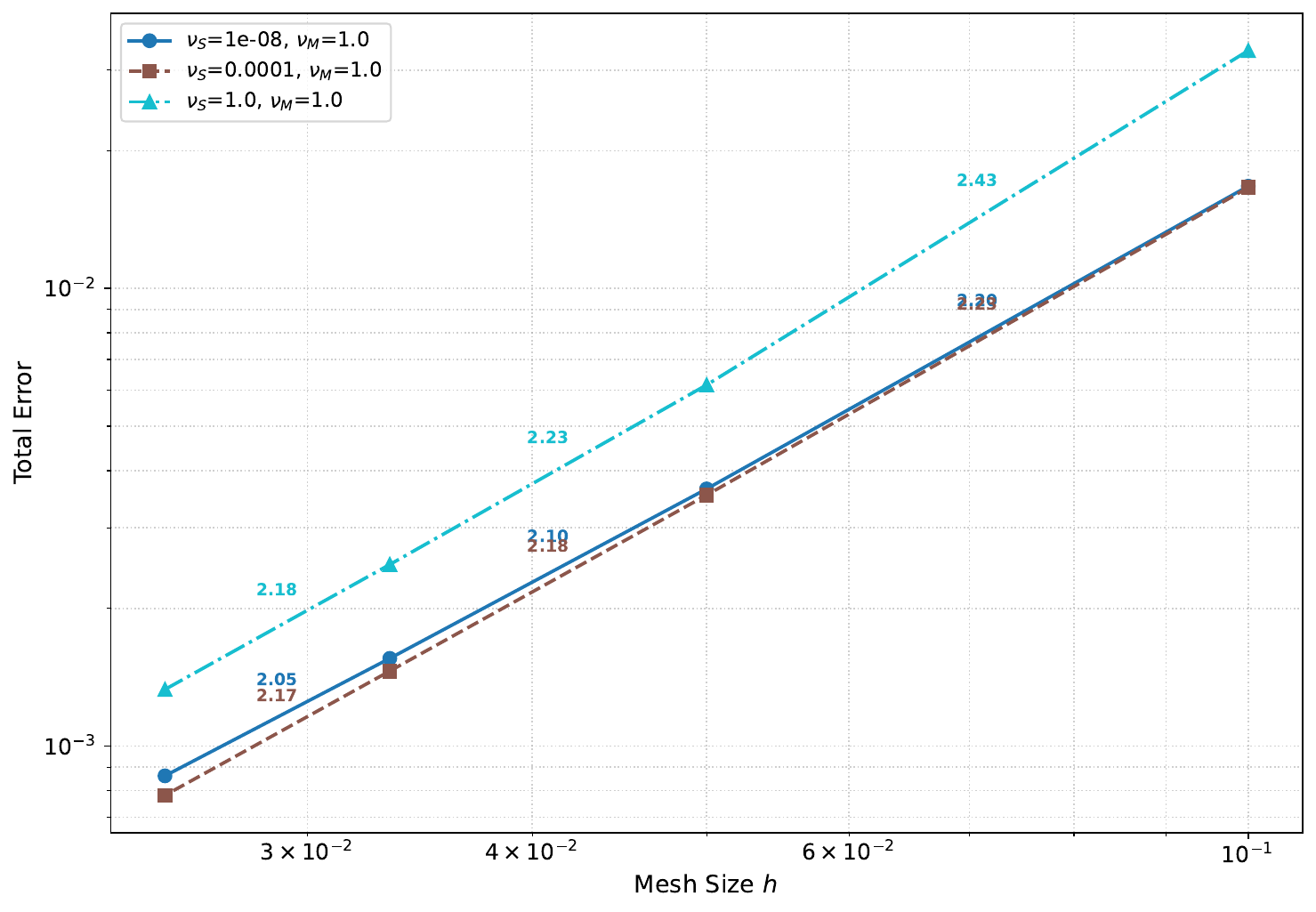}
    \caption{$k = 2$}
\end{subfigure}

\caption{Convergence errors for \texttt{Method~1} 
considering the smooth solution in~\eqref{eq:smooth-sol}.}
\label{fig:method1}
\end{figure}

\paragraph{Methods 2 and 3.}
We study the convergence errors for~$k \in \{1,2\}$ and~$\nS = \nM \in \{1, 10^{-4}, 10^{-8}\}$. The results for \texttt{Method~2} and~\texttt{Method~3}, reported in Figures~\ref{fig:method2} and~\ref{fig:method3}, respectively, confirm the pre-asymptotic convergence rates~$\mathcal{O}(h^k)$ and $\mathcal{O}(h^{k+ \frac{1}{2}})$ predicted by Theorems~\ref{th:main:1:bis} and~\ref{th:main:1:bis2} for the high fluid and magnetic Reynolds number regime. 
\begin{figure}[!ht]
\centering
\begin{subfigure}{0.49\textwidth}
    \centering
    \includegraphics[width=\linewidth, trim=0 0 0 0, clip]{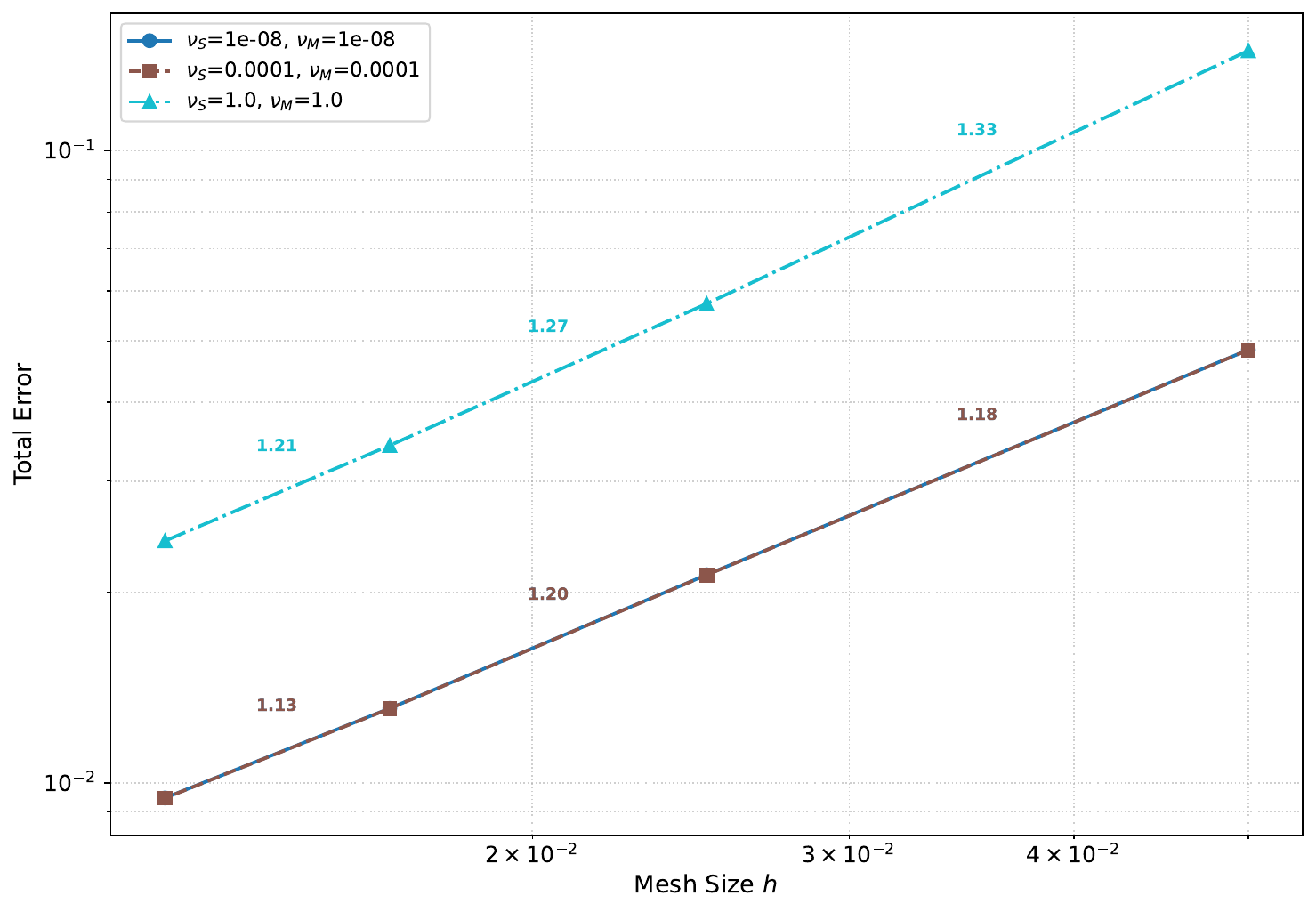}
    \caption{$k = 1$}
\end{subfigure}
\hfill
\begin{subfigure}{0.49\textwidth}
    \centering   \includegraphics[width=\linewidth, trim=0 0 0 0, clip]{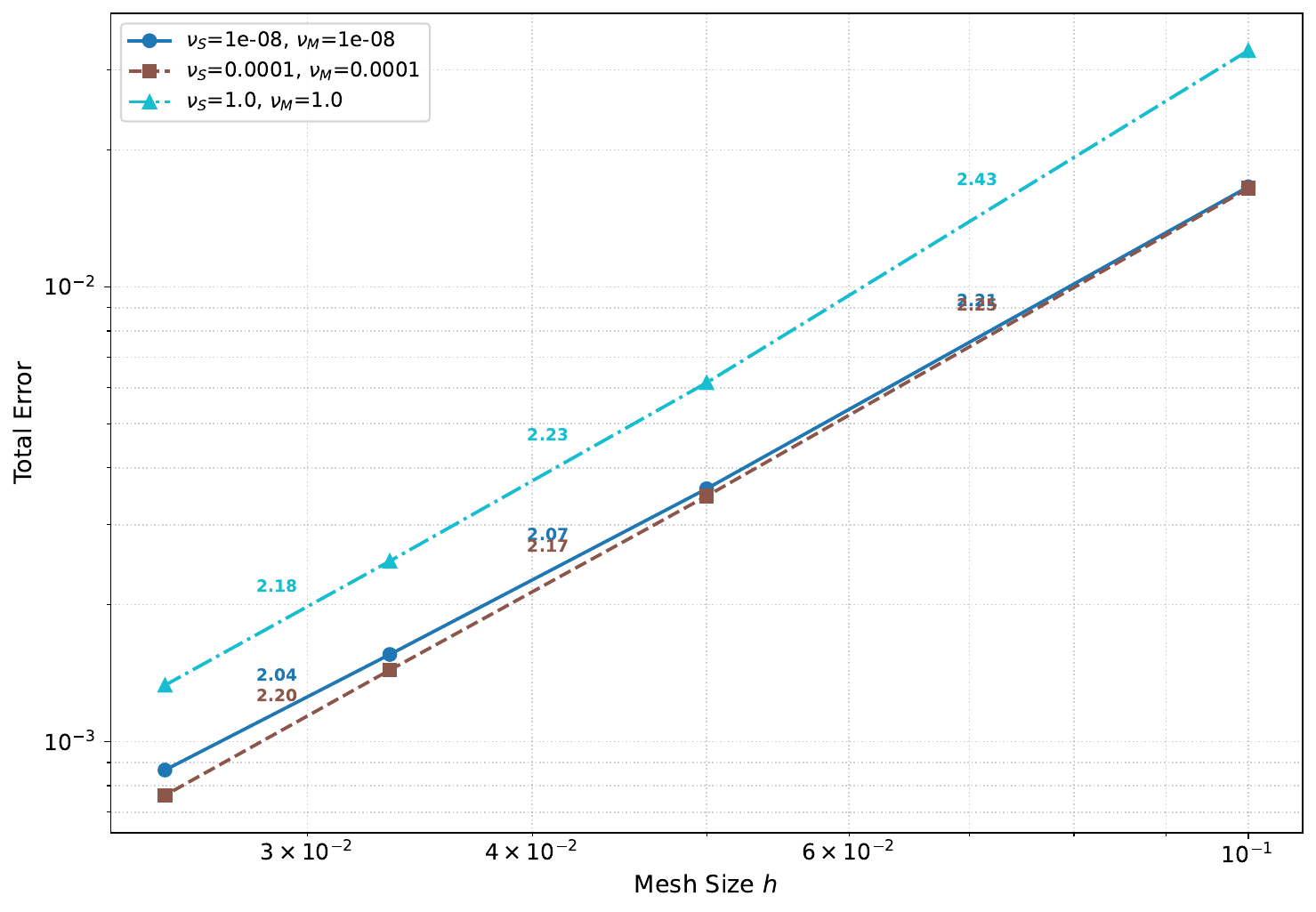}
    \caption{$k = 2$}
\end{subfigure}
\caption{Convergence errors for \texttt{Method~2} 
considering the smooth solution in~\eqref{eq:smooth-sol}.}
\label{fig:method2}
\end{figure}

\begin{figure}[!ht]
\centering
\begin{subfigure}{0.49\textwidth}
    \centering
    \includegraphics[width=\linewidth, trim=0 0 0 0, clip]{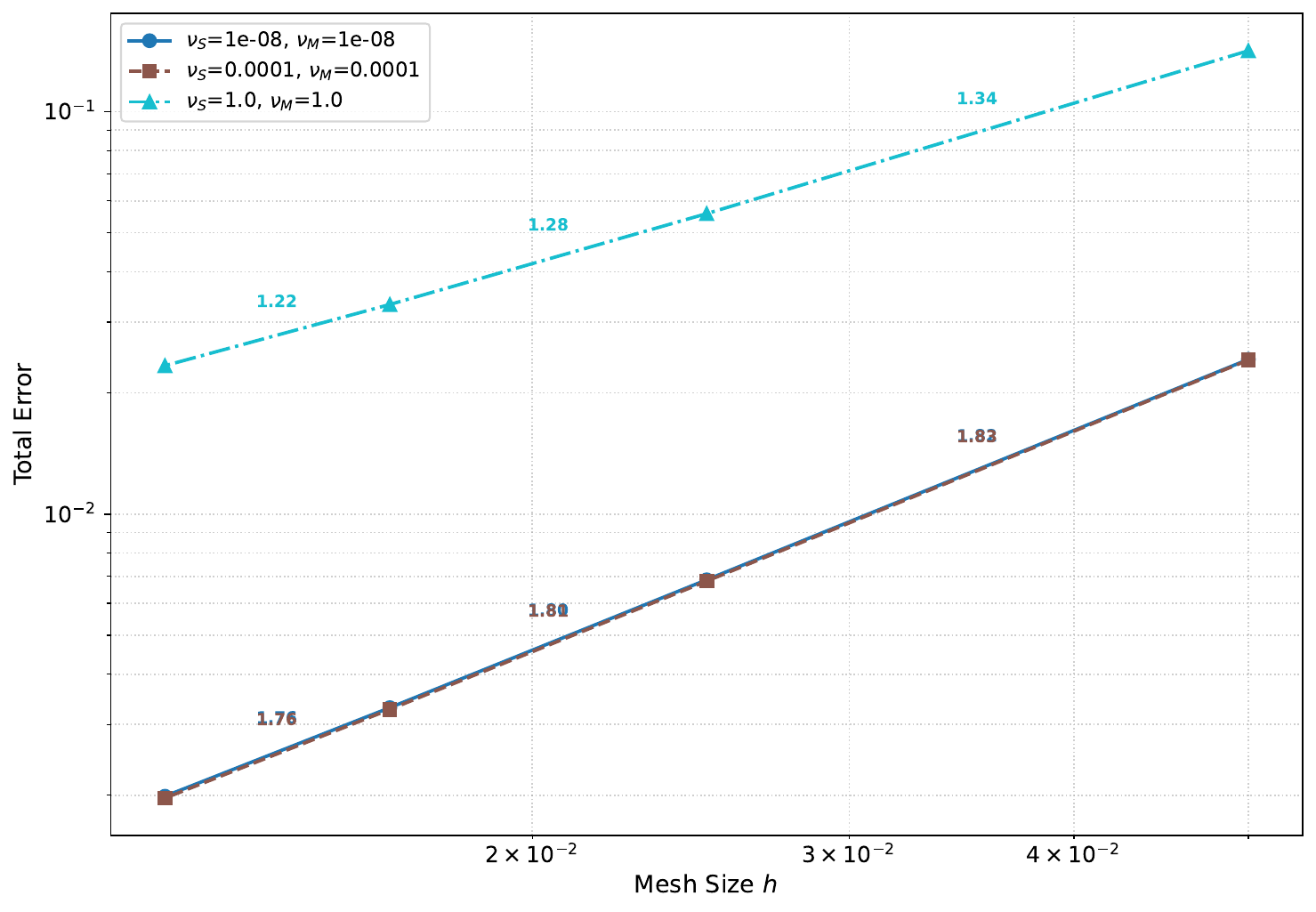}
    \caption{$k = 1$}
\end{subfigure}
\hfill
\begin{subfigure}{0.49\textwidth}
    \centering
    \includegraphics[width=\linewidth, trim=0 0 0 0, clip]{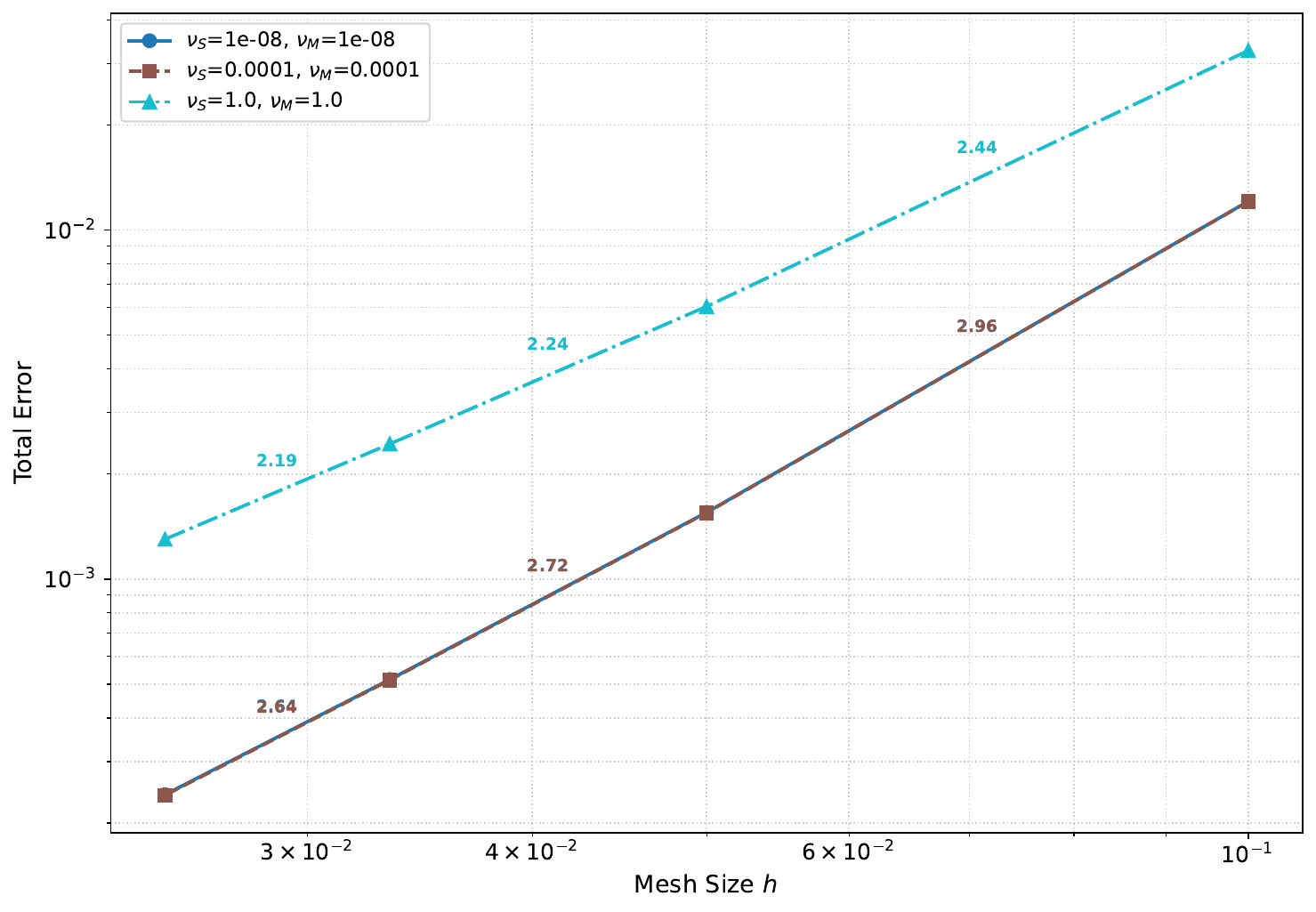}
    \caption{$k = 2$}
\end{subfigure}
\caption{Convergence errors for \texttt{Method~3} 
considering the smooth solution in~\eqref{eq:smooth-sol}.}
\label{fig:method3}
\end{figure}

\subsection{Convergence for a nonsmooth solution on a nonconvex polygonal domain}
\label{sec:nonsmooth-sol}
To investigate the convergence to nonsmooth solutions, we consider the following 
benchmark problem on the L-shaped domain~$(-1,1)^2\setminus [-1,0]^2$ with the following stationary solution:
\begin{equation*}
\begin{split}
\u(x,y) &= (\sin^2(\pi x) \sin(\pi y) \cos(\pi y),\ -\sin(\pi x) \cos(\pi x)\sin^2(\pi y) ), \\
    p(x,y) &= 0, \\
    \mathbf{B}(x,y) &= \nabla \bigg[r(x,y)^{\frac{2}{3}}\sin\Big(\frac{2}{3}\theta(x, y)\Big) \bigg],
    \end{split}
\end{equation*}
where $r(x, y) \coloneqq \sqrt{x^2 + y^2}$ and~$\theta(x,y) = \mathrm{atan2}(y,x)$ are, respectively, the modulus and the principal value of the phase of 
the standard polar coordinates \cite[\S1.9.7]{NIST-Handbook:2010}.
Note that $\B\not\in \H^1(\Omega)$. The results for~$k = 1$ and~$\nS = \nM \in \{1, 10^{-4}, 10^{-8}\}$ are reported in Figure~\ref{fig:Lshape}. The 
$\nM$-robust \texttt{Method~2} 
fails to converge for $\nM \in \{10^{-4}, 10^{-8}\}$, as predicted in Remark \ref{rem:limitations}, whereas \texttt{Method~1} and \texttt{Method~3} converge. 
We finally observe that the particularly good performance of all methods for $\nM = 1$ is most likely a pre-asymptotic effect, related to the fact that the term $\nM \Norm{\curl(\B - \B_h)}{\L^2(\Omega)}^2$, c.f. \eqref{eq:norms:err}, still dominates the error. Indeed, in this test case, 
$\curl \B$ vanishes, so this error component is not affected by the low Sobolev regularity of the magnetic field, which is responsible for the reduced convergence order (or even the complete lack of convergence, in the case of \texttt{Method~2}) of the schemes.

\begin{figure}[!ht]
\centering
\begin{subfigure}{0.49\textwidth}
    \centering
    \includegraphics[width=\linewidth, trim=0 0 0 0, clip]{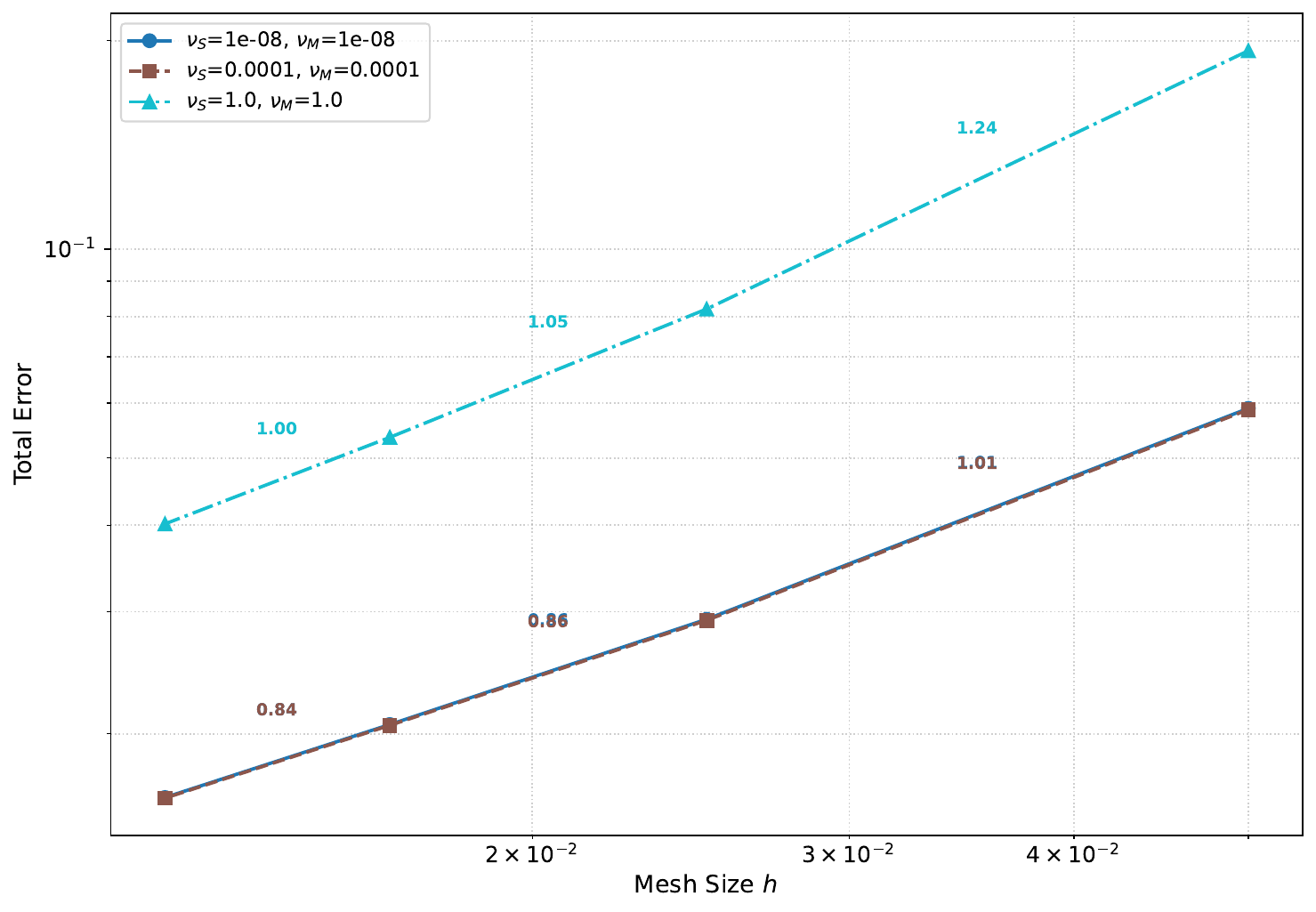}
    \caption{\texttt{Method 1}}
\end{subfigure}
\hfill
\begin{subfigure}{0.49\textwidth}
    \centering
    \includegraphics[width=\linewidth, trim=0 0 0 0, clip]{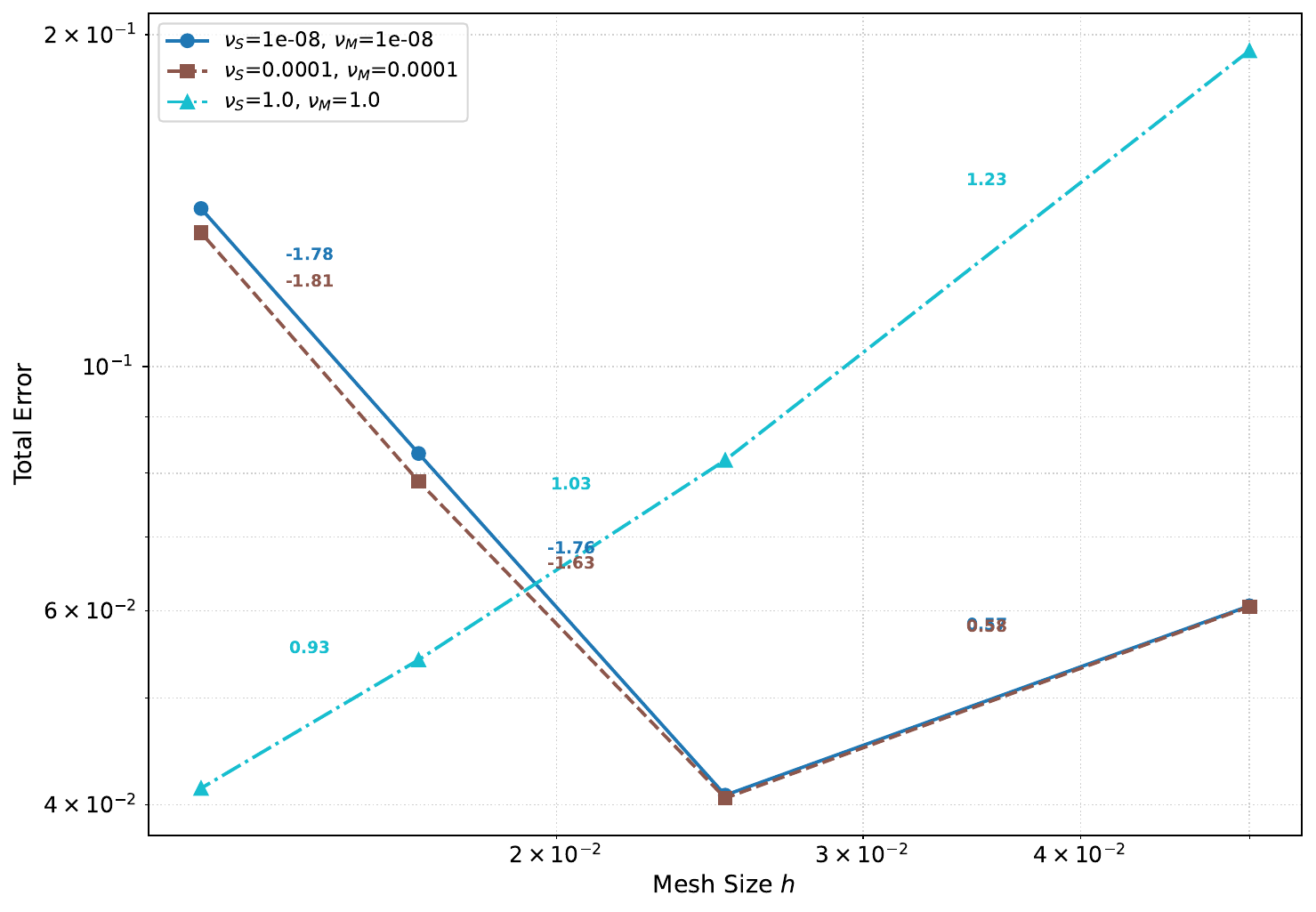}
    \caption{\texttt{Method 2}}
\end{subfigure}
\begin{subfigure}{0.49\textwidth}
    \centering
    \includegraphics[width=\linewidth, trim=0 0 0 0, clip]{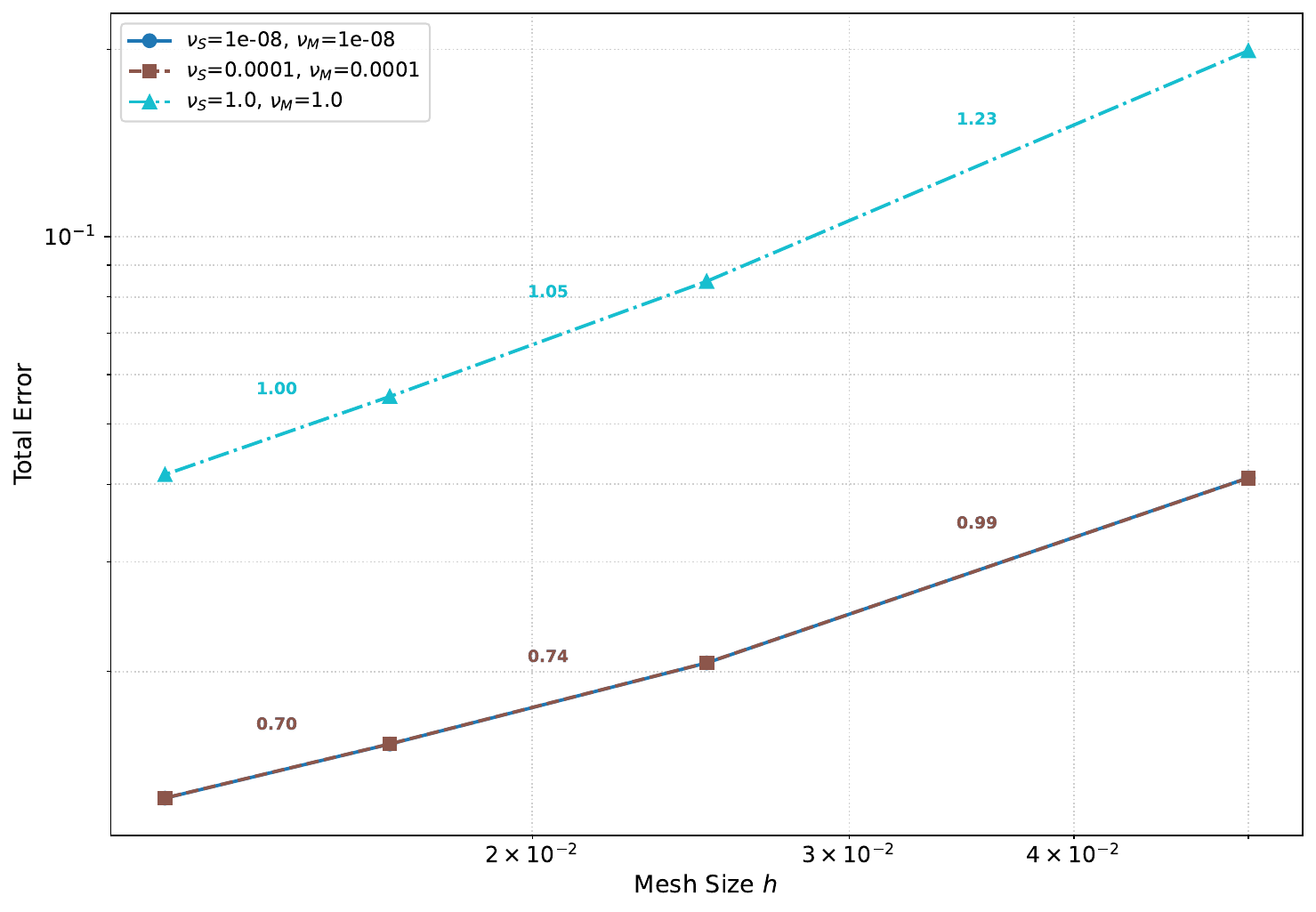}
    \caption{\texttt{Method 3}}
\end{subfigure}
\caption{Convergence results for the problem with singular solution on the 
L-shaped domain.}
\label{fig:Lshape}
\end{figure}
\subsection{Magnetic field loop advection}
\label{sec:magenetic-field-loop}
Originally introduced by Gardiner and Stone~\cite[\S5.1]{GardinerStone05}, the magnetic field loop advection is a challenging test for ideal MHD. On the unit square with periodic boundary conditions, we consider the initial conditions 
\begin{equation*}
\begin{split}
    \u_0(x,y) & = (1,1), \\
    \B_0(x, y)& = (\partial_y A(x,y),\ -\partial_x A(x,y) ),
    \end{split}
\end{equation*}
with 
\begin{equation*}
    A(x, y) = \begin{cases}
        10^{-3}\cdot(0.3 - r(x,y)), \qquad& \text{if $r(x,y)<0.3$},\\
        0, \qquad &\text{otherwise,}
    \end{cases}
\end{equation*}
and $r(x,y) = \sqrt{x^2 + y^2}$. To reproduce the ideal case, we take $\nS = \nM = 10^{-8}$. Since $\semiNorm{\B}{}$ is small with respect to $\semiNorm{\u}{}$, we expect the solution $\B$ to be close to the solution of the magnetic advection problem
\begin{equation*}
\partial_t \B + \curl (\B \times \u)= 0,
\end{equation*}
for the corresponding velocity field~$\u$. 
In particular, at time $T = 1$, it holds $\B(\cdot, 1) \approx\B_0(\cdot)$, since, under periodic boundary conditions, the solution is a pure transport of the initial profile, which returns to its original position after one period.
For this test, we take~$k = 1$, $\Delta t = 10^{-3}$, and 
an unstructured simplicial mesh with~$80$ partitions along both the~$x$ and~$y$ axes. 
The contour lines of the magnetic field strength~$|\Bh|$ are displayed in Figure \ref{fig:loop}, where the methods that are not $\nM$-quasi-robust, namely the \emph{unstabilized} method and \texttt{Method~1}, 
exhibit spurious oscillations, while the 
$\nM$-quasi-robust \texttt{Method~2} and \texttt{Method~3} 
present 
some oscillations only in the vicinity of the discontinuity. 
These residual oscillations can be further reduced by employing a nonlinear limiter; see, e.g.,~\cite[\S4.4]{ZampaBustoDumbser24}.
\begin{figure}[!ht]
\centering
\begin{subfigure}{0.49\textwidth}
    \centering
    \includegraphics[width=\linewidth, trim=100 100 100 100, clip]{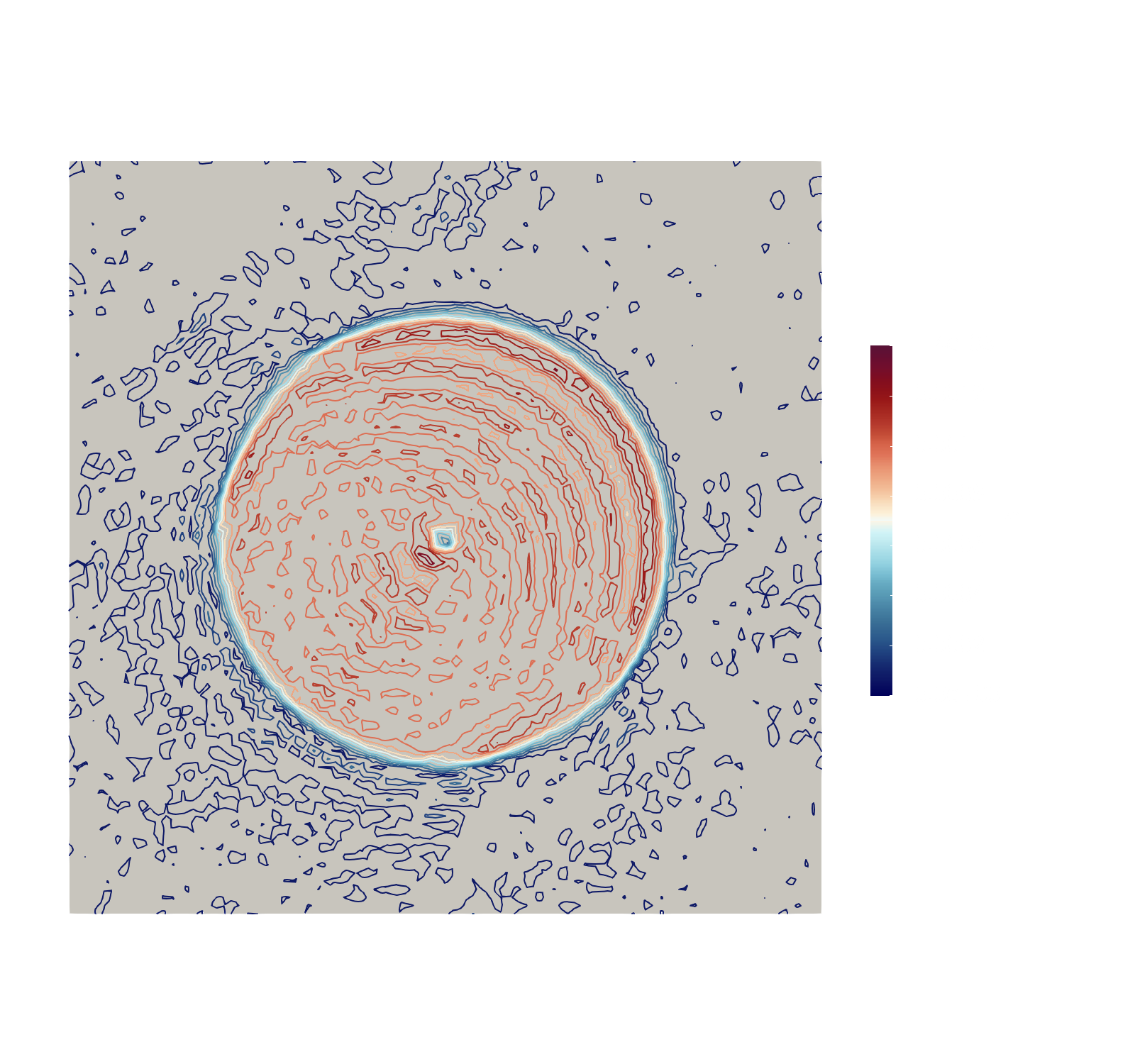}
    \caption{Unstabilized method}
\end{subfigure}
\hfill
\begin{subfigure}{0.49\textwidth}
    \centering
    \includegraphics[width=\linewidth, trim=100 100 100 100, clip]{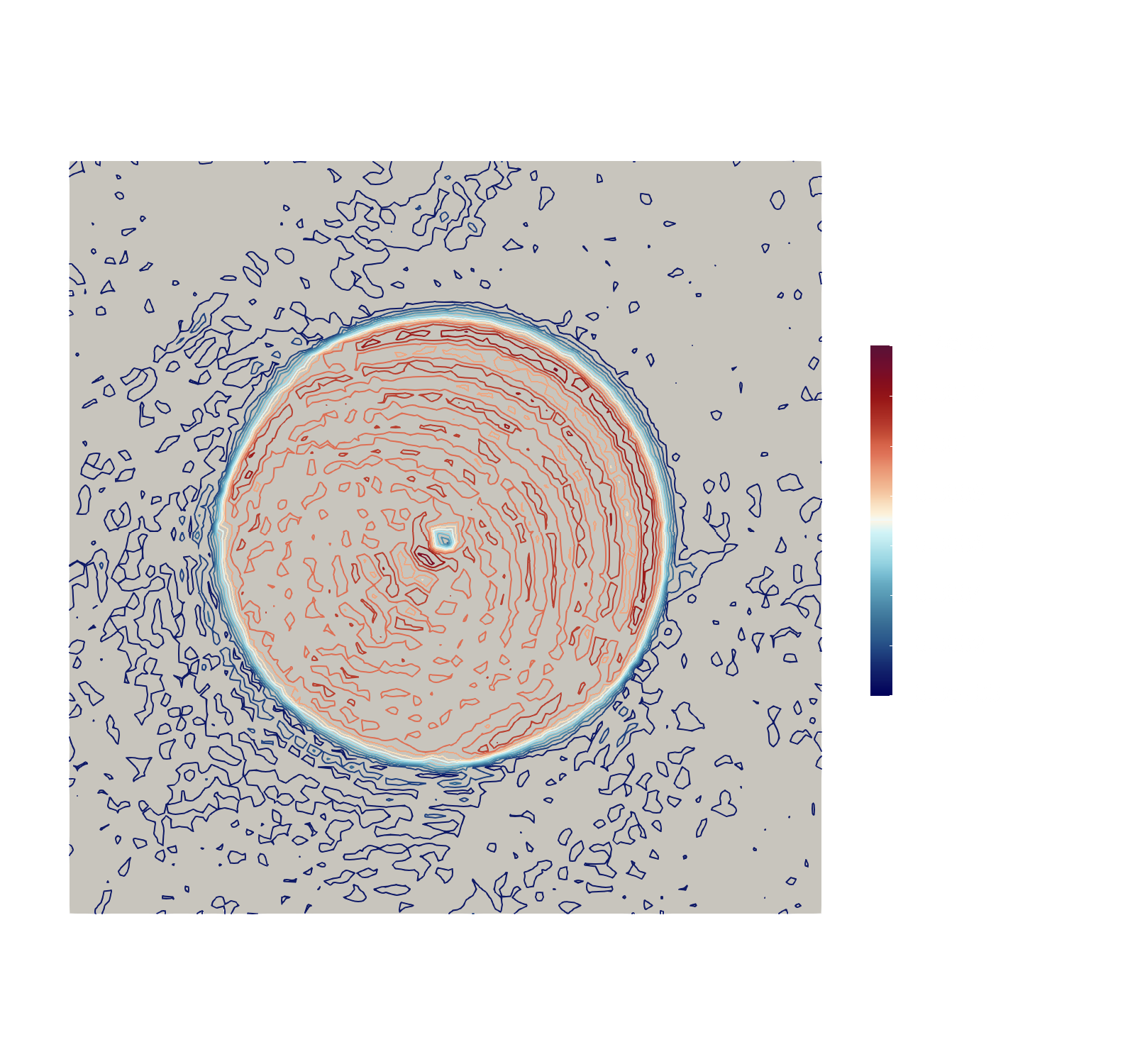}
    \caption{\texttt{Method 1}}
\end{subfigure}
\begin{subfigure}{0.49\textwidth}
    \centering   \includegraphics[width=\linewidth, trim=100 100 100 100, clip]{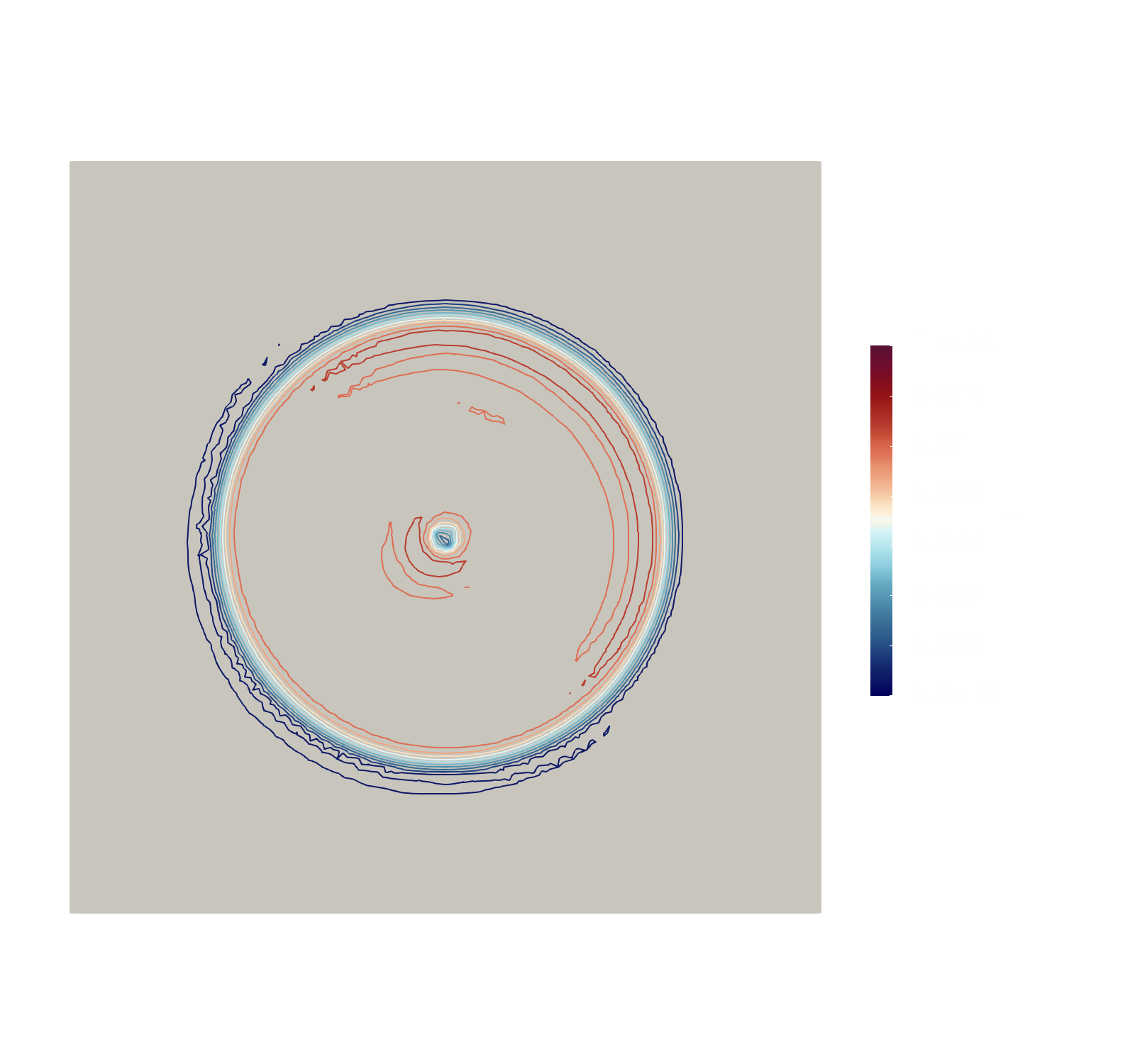}
    \caption{\texttt{Method 2}}
\end{subfigure}
\hfill
\begin{subfigure}{0.49\textwidth}
    \centering
    \includegraphics[width=\linewidth, trim=100 100 100 100, clip]{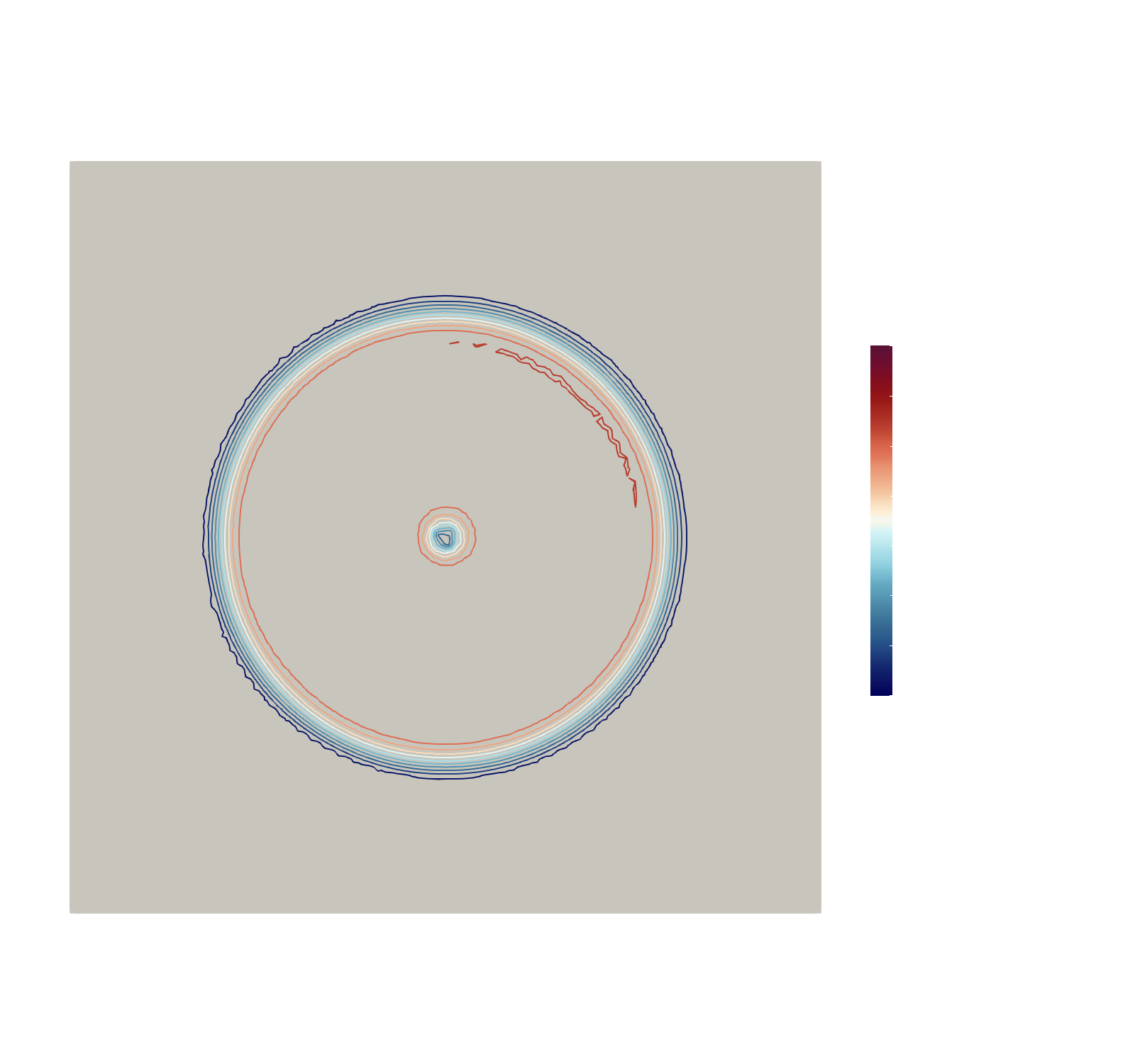}
    \caption{\texttt{Method 3}}
\end{subfigure}
\caption{
Strength~$|\Bh|$ for the magnetic field loop advection problem: $17$ equispaced contour lines from~$2.5 \times 10^{-5}$ to $1.325\times 10^{-3}$.}
\label{fig:loop}
\end{figure}

\subsection{Orszag--Tang vortex}
\label{sec:Orszag-Tang-vortex}
Finally, we consider the vortex solution proposed by Orszag and Tang in \cite{OrszagTang79}. 
We consider again the unit square domain with  periodic boundary conditions, and specify the 
initial conditions 
as
\begin{equation*}
    \begin{split}
         \u_0(x, y) &= (- \sin(2\pi y),\ \sin(2\pi x)),
\\
\B_0(x,y)&= (-\sin(2 \pi y),\ \sin(4 \pi x) ).
    \end{split}
\end{equation*}
We take $\nS = \nM = 10^{-14}$. The spatial domain~$\Omega$ is discretized with an unstructured simplicial mesh with~$50$ 
partitions along both the~$x$ and~$y$ axes, and the time-step is set to~$\Delta t = 10^{-2}$.

The plot of the pressure at the final time $T = 0.4$ is shown in Figure \ref{fig:OT} for all methods tested in Section~\ref{sec:magenetic-field-loop}. 
As in the numerical experiment in Section~\ref{sec:magenetic-field-loop}, the $\nM$-quasi-robust 
\texttt{Method~2} and \texttt{Method~3} 
remain free from spurious oscillations, with the best performance delivered by \texttt{Method~3}. 
The Orszag--Tang vortex problem is a particularly challenging test case. 
Notably, even without a 
direct pressure stabilization, 
the~$\nM$-quasi-robust methods 
effectively suppress nonphysical oscillations in the pressure approximation.
\begin{figure}[!ht]
\centering
\begin{subfigure}{0.49\textwidth}
    \centering
    \includegraphics[width=\linewidth, trim=700 100 700 100, clip]{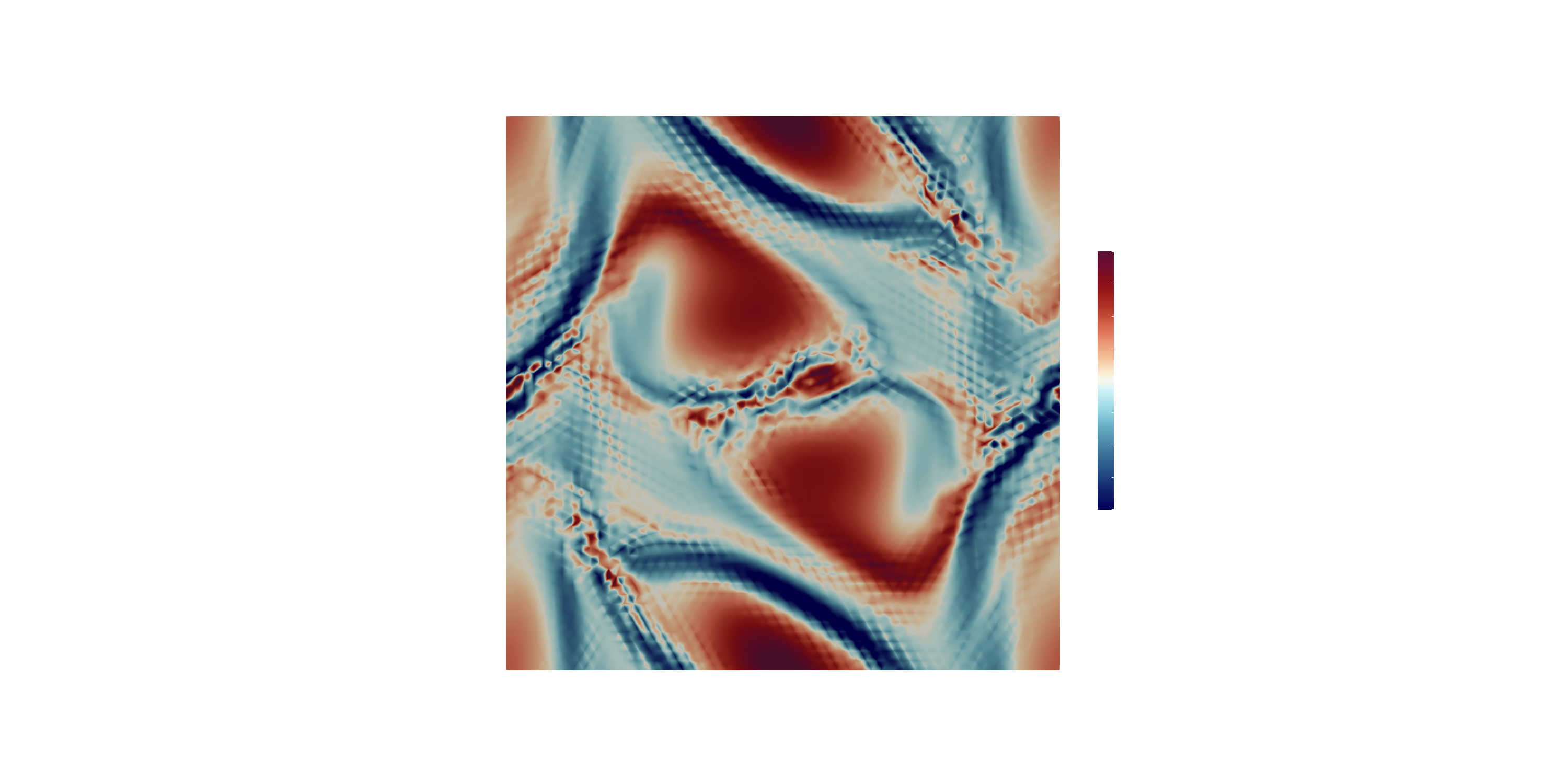}
    \caption{Unstabilized method}
\end{subfigure}
\hfill
\begin{subfigure}{0.49\textwidth}
    \centering
    \includegraphics[width=\linewidth, trim=700 100 700 100, clip]{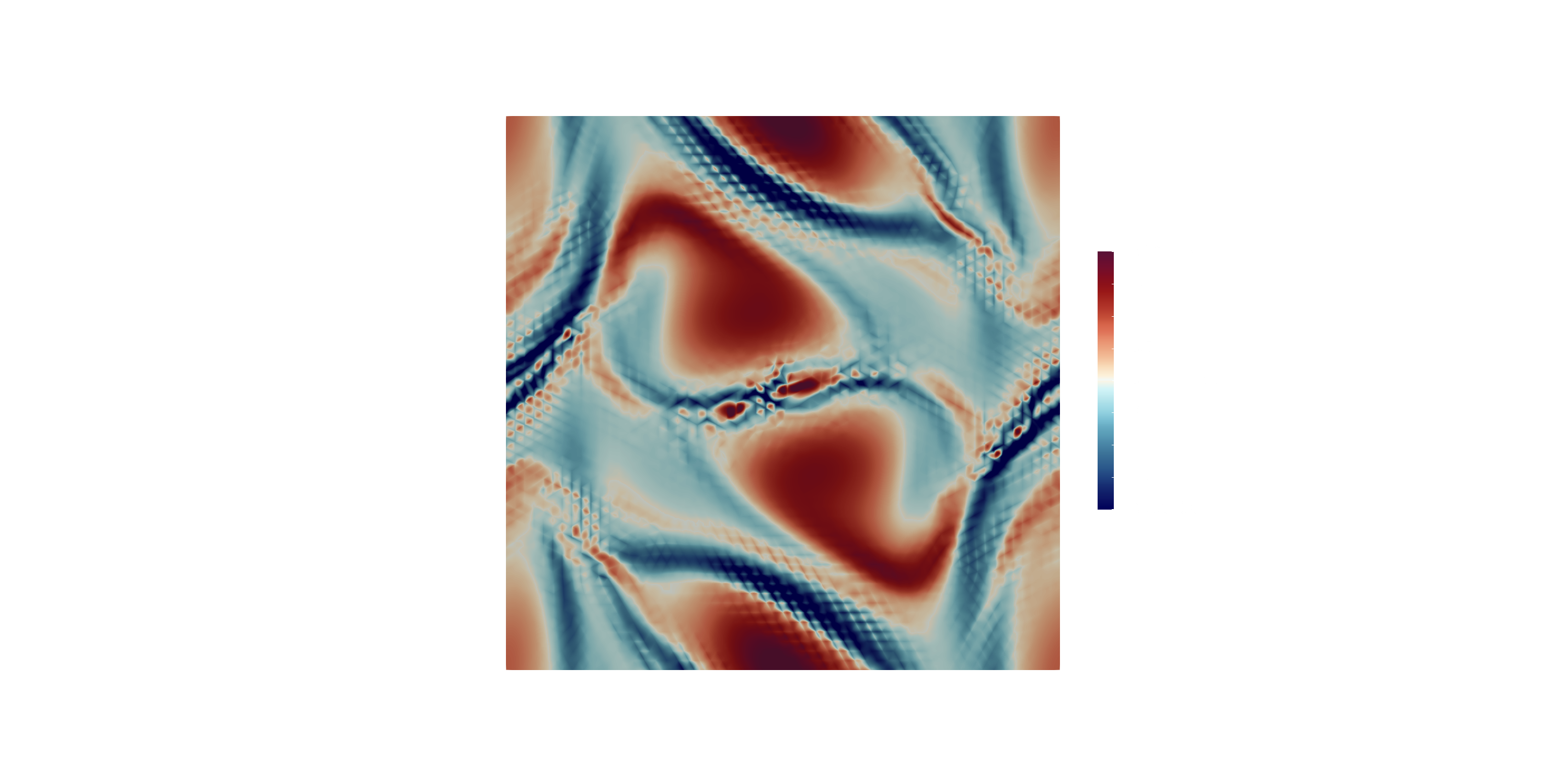}
    \caption{\texttt{Method 1}}
\end{subfigure}
\begin{subfigure}{0.49\textwidth}
    \centering
    \includegraphics[width=\linewidth, trim=700 100 700 100, clip]{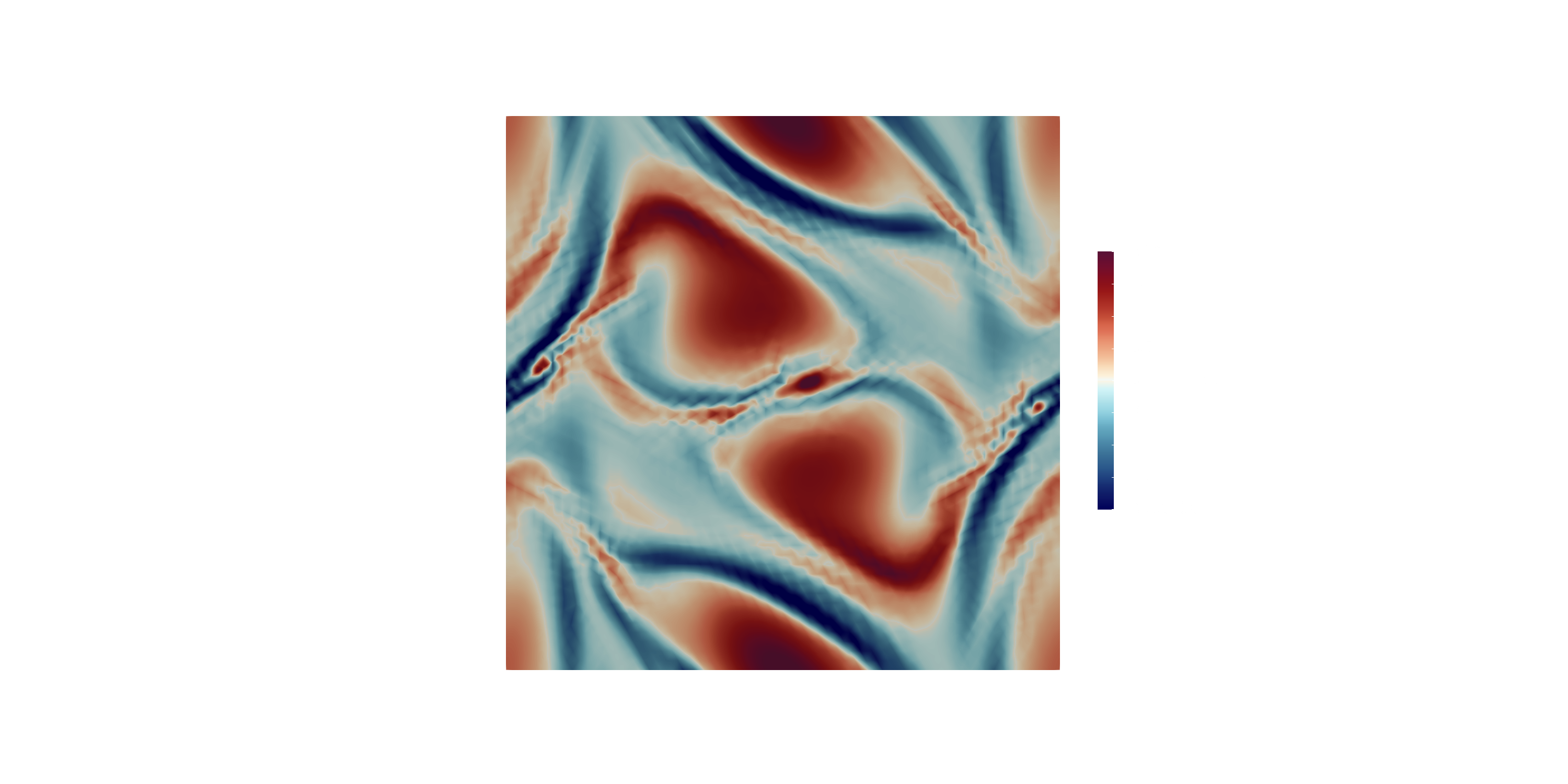}
    \caption{\texttt{Method 2}}
\end{subfigure}
\hfill
\begin{subfigure}{0.49\textwidth}
    \centering
    \includegraphics[width=\linewidth, trim=700 100 700 100, clip]{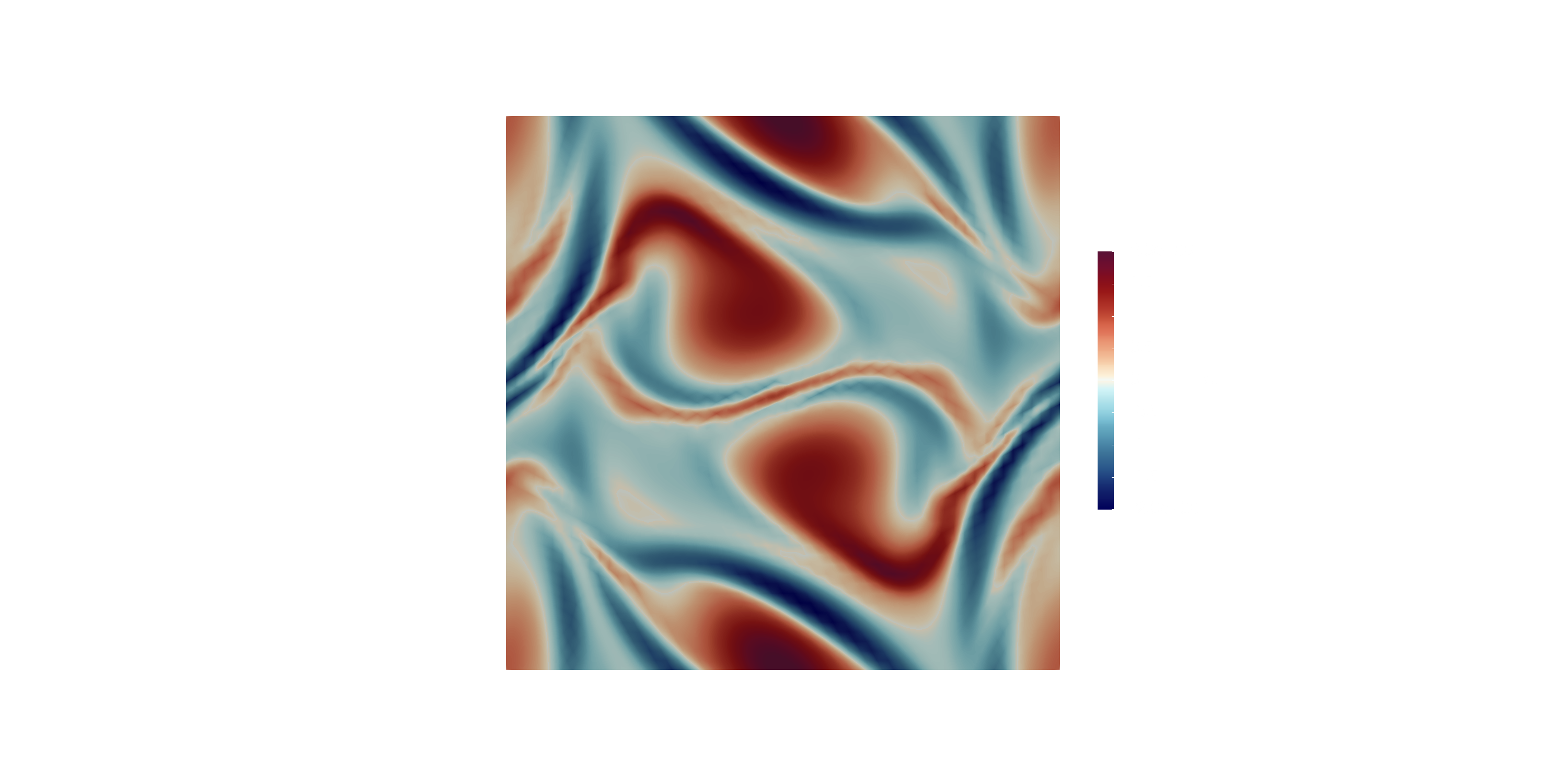}
    \caption{\texttt{Method 3}}
\end{subfigure}
\caption{Pressure approximation for the Orszag--Tang vortex problem at the final time~$T = 0.4$.}
\label{fig:OT}
\end{figure}

Although this is not within the main focus of the present contribution, in order to investigate the effect of stabilization on structure preservation, we also plot the evolution of the energy~$\frac{1}{2} \big(\Norm{\uh}{\L^2(\Omega)}^2+\Norm{\Bh}{\L^2(\Omega)}^2\big)$ and the cross helicity~$(\uh,\Bh)_\Omega$ in Figure \ref{fig:cons}. The results obtained show that only the unstabilized method exactly conserves these 
quantities, while all stabilized variants introduce some energy dissipation and 
fail to preserve cross helicity. 
\begin{figure}[!ht]
\centering
\begin{subfigure}{0.49\textwidth}
    \centering
    \includegraphics[width=\linewidth]{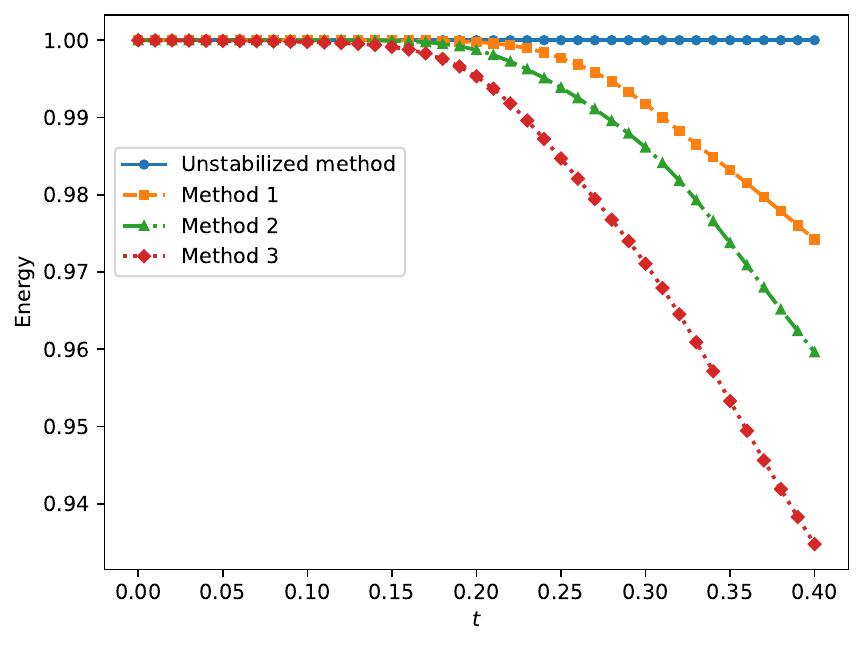}
    \caption{Evolution of the energy}
\end{subfigure}
\hfill
\begin{subfigure}{0.49\textwidth}
    \centering
    \includegraphics[width=\linewidth]{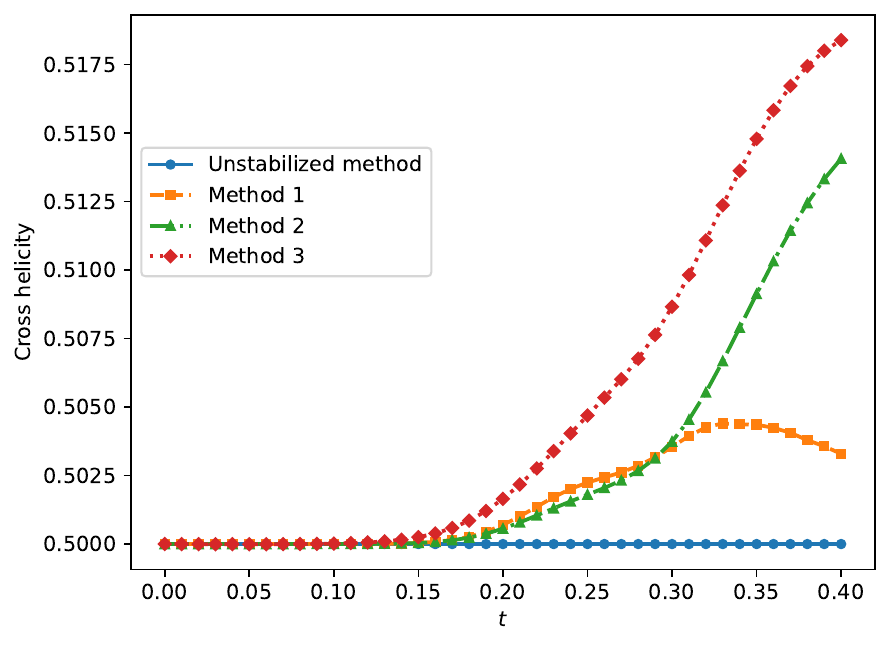}
    \caption{Evolution of cross helicity.}
\end{subfigure}
\caption{Evolution of energy and cross helicity for the Orszag--Tang problem.}
\label{fig:cons}
\end{figure}

\section*{Acknowledgments}
LBDV and SG have been partially funded by the European Union (ERC Synergy, NEMESIS, project number 101115663).
Views and opinions expressed are however those of the authors only and do not necessarily reflect those of the European Union or the ERC Executive Agency. This research was also funded in part by the Austrian Science Fund (FWF) 10.55776/F65. LBDV, SG, and~EZ are members of the INdAM-GNCS group. 

\bibliographystyle{plain}
\bibliography{bibliography}

@article{gerbeau2000stabilized,
    AUTHOR = {Gerbeau, J.-F.},
     TITLE = {A stabilized finite element method for the incompressible
              magnetohydrodynamic equations},
   JOURNAL = {Numer. Math.},
  FJOURNAL = {Numerische Mathematik},
    VOLUME = {87},
      YEAR = {2000},
    NUMBER = {1},
     PAGES = {83--111},
      ISSN = {0029-599X,0945-3245},
   MRCLASS = {76M10 (65M60 76W05)},
  MRNUMBER = {1800155},
       DOI = {10.1007/s002110000193},
       URL = {https://doi.org/10.1007/s002110000193},
}

@article{wacker2016nodal,
    AUTHOR = {Wacker, B. and Arndt, D. and Lube, G.},
     TITLE = {Nodal-based finite element methods with local projection
              stabilization for linearized incompressible
              magnetohydrodynamics},
   JOURNAL = {Comput. Methods Appl. Mech. Engrg.},
  FJOURNAL = {Computer Methods in Applied Mechanics and Engineering},
    VOLUME = {302},
      YEAR = {2016},
     PAGES = {170--192},
      ISSN = {0045-7825,1879-2138},
   MRCLASS = {65M60 (65M15 76W05)},
  MRNUMBER = {3461110},
MRREVIEWER = {De\ Kang\ Mao},
       DOI = {10.1016/j.cma.2016.01.004},
       URL = {https://doi.org/10.1016/j.cma.2016.01.004},
}

@article{douglas1974stability,
    AUTHOR = {Douglas Jr., J. and Dupont, T. and Wahlbin, L.},
     TITLE = {The stability in {$L\sp{q}$} of the {$L\sp{2}$}-projection
              into finite element function spaces},
   JOURNAL = {Numer. Math.},
  FJOURNAL = {Numerische Mathematik},
    VOLUME = {23},
      YEAR = {1974/75},
     PAGES = {193--197},
      ISSN = {0029-599X,0945-3245},
   MRCLASS = {65M30},
  MRNUMBER = {383789},
MRREVIEWER = {Steven\ Pruess},
       DOI = {10.1007/BF01400302},
       URL = {https://doi.org/10.1007/BF01400302},
}

@book{brenner2008mathematical,
 author = {Brenner, S. C. and Scott, L. R.},
 title = {The mathematical theory of finite element methods},
 edition = {3rd ed.},
 fseries = {Texts in Applied Mathematics},
 series = {Texts Appl. Math.},
 issn = {0939-2475},
 volume = {15},
 isbn = {978-0-387-75933-3},
 year = {2008},
 publisher = {New York, NY: Springer},
 language = {English},
 doi = {10.1007/978-0-387-75934-0},
 zbMATH = {5223061},
 Zbl = {1135.65042}
}

@article{chaumont2024stable,
    AUTHOR = {Chaumont-Frelet, T. and Vohral\'ik, M.},
     TITLE = {A stable local commuting projector and optimal {$hp$}
              approximation estimates in {$H({\rm curl})$}},
   JOURNAL = {Numer. Math.},
  FJOURNAL = {Numerische Mathematik},
    VOLUME = {156},
      YEAR = {2024},
    NUMBER = {6},
     PAGES = {2293--2342},
      ISSN = {0029-599X,0945-3245},
   MRCLASS = {65N30},
  MRNUMBER = {4827942},
       DOI = {10.1007/s00211-024-01431-w},
       URL = {https://doi.org/10.1007/s00211-024-01431-w},
}

@article{beirao2024robust,
    AUTHOR = {Beir\~ao da Veiga, L. and Dassi, F. and Vacca, G.},
     TITLE = {Robust finite elements for linearized magnetohydrodynamics},
   JOURNAL = {SIAM J. Numer. Anal.},
  FJOURNAL = {SIAM Journal on Numerical Analysis},
    VOLUME = {62},
      YEAR = {2024},
    NUMBER = {4},
     PAGES = {1539--1564},
      ISSN = {0036-1429,1095-7170},
   MRCLASS = {65N30 (76E06 76E25)},
  MRNUMBER = {4770389},
       DOI = {10.1137/23M1582783},
       URL = {https://doi.org/10.1137/23M1582783},
}

@article {RobustMHD,
    AUTHOR = {Beir\~ao da Veiga, L. and Dassi, F. and Vacca, G.},
     TITLE = {Pressure and convection robust finite elements for
              magnetohydrodynamics},
   JOURNAL = {Numer. Math.},
  FJOURNAL = {Numerische Mathematik},
    VOLUME = {157},
      YEAR = {2025},
    NUMBER = {4},
     PAGES = {1161--1209},
      ISSN = {0029-599X,0945-3245},
   MRCLASS = {76E06 (65N30 76E25)},
  MRNUMBER = {4961427},
       DOI = {10.1007/s00211-025-01476-5},
       URL = {https://doi.org/10.1007/s00211-025-01476-5},
}

@article {HiptmairPagliantini18,
    AUTHOR = {Hiptmair, R. and Pagliantini, C.},
     TITLE = {Splitting-based structure preserving discretizations for
              magnetohydrodynamics},
   JOURNAL = {SMAI J. Comput. Math.},
  FJOURNAL = {SMAI Journal of Computational Mathematics},
    VOLUME = {4},
      YEAR = {2018},
     PAGES = {225--257},
      ISSN = {2426-8399},
   MRCLASS = {65M60 (65M08 76W05)},
  MRNUMBER = {3813097},
MRREVIEWER = {JiChun\ Li},
       DOI = {10.5802/smai-jcm.34},
       URL = {https://doi.org/10.5802/smai-jcm.34},
}

@article{AndrewsFarrell25,
    AUTHOR = {He, M. and Farrell, P. E. and Hu, K. and
              Andrews, B. D.},
     TITLE = {Helicity-{P}reserving {F}inite {E}lement {D}iscretization for
              {M}agnetic {R}elaxation},
   JOURNAL = {SIAM J. Sci. Comput.},
  FJOURNAL = {SIAM Journal on Scientific Computing},
    VOLUME = {48},
      YEAR = {2026},
    NUMBER = {2},
     PAGES = {B165--B183},
      ISSN = {1064-8275,1095-7197},
   MRCLASS = {76W05 (65M60 82D10 85A30)},
  MRNUMBER = {5042378},
       DOI = {10.1137/25M1727540},
       URL = {https://doi.org/10.1137/25M1727540},
}

@article {Costabel91,
    AUTHOR = {Costabel, M.},
     TITLE = {A coercive bilinear form for {M}axwell's equations},
   JOURNAL = {J. Math. Anal. Appl.},
  FJOURNAL = {Journal of Mathematical Analysis and Applications},
    VOLUME = {157},
      YEAR = {1991},
    NUMBER = {2},
     PAGES = {527--541},
      ISSN = {0022-247X,1096-0813},
   MRCLASS = {35Q60 (78A25)},
  MRNUMBER = {1112332},
       DOI = {10.1016/0022-247X(91)90104-8},
       URL = {https://doi.org/10.1016/0022-247X(91)90104-8},
}

@article {BalsaraDumbser2015,
    AUTHOR = {Balsara, D. S. and Dumbser, M.},
     TITLE = {Divergence-free {MHD} on unstructured meshes using high order
              finite volume schemes based on multidimensional {R}iemann
              solvers},
   JOURNAL = {J. Comput. Phys.},
  FJOURNAL = {Journal of Computational Physics},
    VOLUME = {299},
      YEAR = {2015},
     PAGES = {687--715},
      ISSN = {0021-9991,1090-2716},
   MRCLASS = {65M08 (76M12 76W05)},
  MRNUMBER = {3384747},
       DOI = {10.1016/j.jcp.2015.07.012},
       URL = {https://doi.org/10.1016/j.jcp.2015.07.012},
}

@article {FambriEtAl23,
    AUTHOR = {Fambri, F. and Zampa, E. and Busto, S. and R\'io-Mart\'in, L.
              and Hindenlang, F. and Sonnendr\"ucker, E. and Dumbser, M.},
     TITLE = {A well-balanced and exactly divergence-free staggered
              semi-implicit hybrid finite volume / finite element scheme for
              the incompressible {MHD} equations},
   JOURNAL = {J. Comput. Phys.},
  FJOURNAL = {Journal of Computational Physics},
    VOLUME = {493},
      YEAR = {2023},
     PAGES = {Paper No. 112493, 48},
      ISSN = {0021-9991,1090-2716},
   MRCLASS = {76W05 (65M08 65M60 76M10 76M12)},
  MRNUMBER = {4645615},
       DOI = {10.1016/j.jcp.2023.112493},
       URL = {https://doi.org/10.1016/j.jcp.2023.112493},
}

@article {Fu19,
    AUTHOR = {Fu, G.},
     TITLE = {An explicit divergence-free {DG} method for incompressible
              magnetohydrodynamics},
   JOURNAL = {J. Sci. Comput.},
  FJOURNAL = {Journal of Scientific Computing},
    VOLUME = {79},
      YEAR = {2019},
    NUMBER = {3},
     PAGES = {1737--1752},
      ISSN = {0885-7474,1573-7691},
   MRCLASS = {65M60 (65M12 76W05)},
  MRNUMBER = {3946474},
MRREVIEWER = {Zhongyi\ Huang},
       DOI = {10.1007/s10915-019-00909-2},
       URL = {https://doi.org/10.1007/s10915-019-00909-2},
}

@article {ZampaBustoDumbser24,
    AUTHOR = {Zampa, E. and Busto, S. and Dumbser, M.},
     TITLE = {A divergence-free hybrid finite volume / finite element scheme
              for the incompressible {MHD} equations based on compatible
              finite element spaces with a posteriori limiting},
   JOURNAL = {Appl. Numer. Math.},
  FJOURNAL = {Applied Numerical Mathematics. An IMACS Journal},
    VOLUME = {198},
      YEAR = {2024},
     PAGES = {346--374},
      ISSN = {0168-9274,1873-5460},
   MRCLASS = {65M08 (65M60 76M10 76M12 76W05)},
  MRNUMBER = {4697942},
       DOI = {10.1016/j.apnum.2024.01.014},
       URL = {https://doi.org/10.1016/j.apnum.2024.01.014},
}

@article {WimmerTang24,
    AUTHOR = {Wimmer, G. A. and Tang, X.-Z.},
     TITLE = {Structure preserving transport stabilized compatible finite
              element methods for magnetohydrodynamics},
   JOURNAL = {J. Comput. Phys.},
  FJOURNAL = {Journal of Computational Physics},
    VOLUME = {501},
      YEAR = {2024},
     PAGES = {Paper No. 112777, 25},
      ISSN = {0021-9991,1090-2716},
   MRCLASS = {76M10 (65M60 76W05)},
  MRNUMBER = {4692488},
MRREVIEWER = {Xiaojing\ Dong},
       DOI = {10.1016/j.jcp.2024.112777},
       URL = {https://doi.org/10.1016/j.jcp.2024.112777},
}

@book{Gerbeau-etal-book:2006,
 author = {Gerbeau, J.-F. and Le Bris, C. and Leli{\`e}vre, T.},
 title = {Mathematical methods for the magnetohydrodynamics of liquid metals.},
 fseries = {Numerical Mathematics and Scientific Computation},
 series = {Numer. Math. Sci. Comput.},
 isbn = {0-19-856665-4},
 year = {2006},
 publisher = {Oxford: Oxford University Press},
 language = {English},
 doi = {10.1093/acprof:oso/9780198566656.001.0001},
 url = {semanticscholar.org/paper/f496bb82279d7929bec7b73ea143134f8c8c2b05},
 zbMATH = {5076263},
 Zbl = {1107.76001}
}

@article {Nedelec:1980,
    AUTHOR = {N\'ed\'elec, J.-C.},
     TITLE = {Mixed finite elements in {${\bf R}\sp{3}$}},
   JOURNAL = {Numer. Math.},
  FJOURNAL = {Numerische Mathematik},
    VOLUME = {35},
      YEAR = {1980},
    NUMBER = {3},
     PAGES = {315--341},
      ISSN = {0029-599X,0945-3245},
   MRCLASS = {65N30},
  MRNUMBER = {592160},
MRREVIEWER = {P.\ G.\ Ciarlet},
       DOI = {10.1007/BF01396415},
       URL = {https://doi.org/10.1007/BF01396415},
}

@article {Brezzi_Douglas_Duran_Fortin:1987,
    AUTHOR = {Brezzi, F. and Douglas, J. and Dur\'an, R. and
              Fortin, M.},
     TITLE = {Mixed finite elements for second order elliptic problems in
              three variables},
   JOURNAL = {Numer. Math.},
  FJOURNAL = {Numerische Mathematik},
    VOLUME = {51},
      YEAR = {1987},
    NUMBER = {2},
     PAGES = {237--250},
      ISSN = {0029-599X,0945-3245},
   MRCLASS = {65N30},
  MRNUMBER = {890035},
MRREVIEWER = {Jaroslaw\ Jelen},
       DOI = {10.1007/BF01396752},
       URL = {https://doi.org/10.1007/BF01396752},
}

@book{Ern_Guermond-I:2020,
 author = {Ern, A. and Guermond, J.-L.},
 title = {Finite elements {I}. {Approximation} and interpolation},
 fseries = {Texts in Applied Mathematics},
 series = {Texts Appl. Math.},
 issn = {0939-2475},
 volume = {72},
 isbn = {978-3-030-56340-0; 978-3-030-56342-4; 978-3-030-56341-7},
 year = {2020},
 publisher = {Cham: Springer},
 language = {English},
 doi = {10.1007/978-3-030-56341-7},
 zbMATH = {7258486},
 Zbl = {1476.65003}
}

@article {John_etal:2017,
    AUTHOR = {John, V. and Linke, A. and Merdon, C. and
              Neilan, M. and Rebholz, L. G.},
     TITLE = {On the divergence constraint in mixed finite element methods
              for incompressible flows},
   JOURNAL = {SIAM Rev.},
  FJOURNAL = {SIAM Review},
    VOLUME = {59},
      YEAR = {2017},
    NUMBER = {3},
     PAGES = {492--544},
      ISSN = {0036-1445,1095-7200},
   MRCLASS = {65N30 (76M10)},
  MRNUMBER = {3683678},
MRREVIEWER = {Stephan\ Schmidt},
       DOI = {10.1137/15M1047696},
       URL = {https://doi.org/10.1137/15M1047696},
}

@article {MaoXi25,
    AUTHOR = {Mao, S. and Xi, R.},
     TITLE = {An incompressibility, {${\rm div}\,B = 0$} preserving, current
              density, helicity, energy-conserving finite element method for
              incompressible {MHD} systems},
   JOURNAL = {J. Comput. Phys.},
  FJOURNAL = {Journal of Computational Physics},
    VOLUME = {538},
      YEAR = {2025},
     PAGES = {Paper No. 114130, 31},
      ISSN = {0021-9991,1090-2716},
   MRCLASS = {65M60 (65F08 65M06 65M12 76D05 76W05)},
  MRNUMBER = {4917731},
MRREVIEWER = {Xiaodi\ Zhang},
       DOI = {10.1016/j.jcp.2025.114130},
       URL = {https://doi.org/10.1016/j.jcp.2025.114130},
}

@article {HuLeeXu21,
    AUTHOR = {Hu, K. and Lee, Y.-J. and Xu, J.},
     TITLE = {Helicity-conservative finite element discretization for
              incompressible {MHD} systems},
   JOURNAL = {J. Comput. Phys.},
  FJOURNAL = {Journal of Computational Physics},
    VOLUME = {436},
      YEAR = {2021},
     PAGES = {Paper No. 110284, 17},
      ISSN = {0021-9991,1090-2716},
   MRCLASS = {65M60 (76W05)},
  MRNUMBER = {4236017},
       DOI = {10.1016/j.jcp.2021.110284},
       URL = {https://doi.org/10.1016/j.jcp.2021.110284},
}

@article {DualFieldNS,
    AUTHOR = {Zhang, Y. and Palha, A. and Gerritsma, M. and Rebholz,
              L. G.},
     TITLE = {A mass-, kinetic energy- and helicity-conserving mimetic
              dual-field discretization for three-dimensional incompressible
              {N}avier-{S}tokes equations, part {I}: periodic domains},
   JOURNAL = {J. Comput. Phys.},
  FJOURNAL = {Journal of Computational Physics},
    VOLUME = {451},
      YEAR = {2022},
     PAGES = {Paper No. 110868, 23},
      ISSN = {0021-9991,1090-2716},
   MRCLASS = {65M06 (65M12 76D05)},
  MRNUMBER = {4354366},
       DOI = {10.1016/j.jcp.2021.110868},
       URL = {https://doi.org/10.1016/j.jcp.2021.110868},
}

@article {Girault88,
    AUTHOR = {Girault, V.},
     TITLE = {Incompressible finite element methods for {N}avier-{S}tokes
              equations with nonstandard boundary conditions in {${\bf
              R}^3$}},
   JOURNAL = {Math. Comp.},
  FJOURNAL = {Mathematics of Computation},
    VOLUME = {51},
      YEAR = {1988},
    NUMBER = {183},
     PAGES = {55--74},
      ISSN = {0025-5718,1088-6842},
   MRCLASS = {65N30 (76-08 76D05)},
  MRNUMBER = {942143},
       DOI = {10.2307/2008579},
       URL = {https://doi.org/10.2307/2008579},
}

@incollection {Girault90,
    AUTHOR = {Girault, V.},
     TITLE = {Curl-conforming finite element methods for {N}avier-{S}tokes
              equations with nonstandard boundary conditions in {${\bf
              R}^3$}},
 BOOKTITLE = {The {N}avier-{S}tokes equations ({O}berwolfach, 1988)},
    SERIES = {Lecture Notes in Math.},
    VOLUME = {1431},
     PAGES = {201--218},
 PUBLISHER = {Springer, Berlin},
      YEAR = {1990},
      ISBN = {3-540-52770-2},
   MRCLASS = {65N30 (35Q30 76D05 76M10)},
  MRNUMBER = {1072191},
MRREVIEWER = {Long\ An\ Ying},
       DOI = {10.1007/BFb0086071},
       URL = {https://doi.org/10.1007/BFb0086071},
}

@article {LaakmannHuFarrell23,
    AUTHOR = {Laakmann, F. and Hu, K. and Farrell, P. E.},
     TITLE = {Structure-preserving and helicity-conserving finite element
              approximations and preconditioning for the {H}all {MHD}
              equations},
   JOURNAL = {J. Comput. Phys.},
  FJOURNAL = {Journal of Computational Physics},
    VOLUME = {492},
      YEAR = {2023},
     PAGES = {Paper No. 112410, 25},
      ISSN = {0021-9991,1090-2716},
   MRCLASS = {76M10 (65M60 76W05)},
  MRNUMBER = {4632710},
MRREVIEWER = {Hailong\ Qiu},
       DOI = {10.1016/j.jcp.2023.112410},
       URL = {https://doi.org/10.1016/j.jcp.2023.112410},
}

@article {Prohl08,
    AUTHOR = {Prohl, A.},
     TITLE = {Convergent finite element discretizations of the nonstationary
              incompressible magnetohydrodynamics system},
   JOURNAL = {M2AN Math. Model. Numer. Anal.},
  FJOURNAL = {M2AN. Mathematical Modelling and Numerical Analysis},
    VOLUME = {42},
      YEAR = {2008},
    NUMBER = {6},
     PAGES = {1065--1087},
      ISSN = {0764-583X,1290-3841},
   MRCLASS = {65M60 (76M10 76W05 78A25)},
  MRNUMBER = {2473320},
MRREVIEWER = {S\'ebastien\ J.\ Boyaval},
       DOI = {10.1051/m2an:2008034},
       URL = {https://doi.org/10.1051/m2an:2008034},
}

@article {Schotzau04,
    AUTHOR = {Sch\"otzau, D.},
     TITLE = {Mixed finite element methods for stationary incompressible
              magneto-hydrodynamics},
   JOURNAL = {Numer. Math.},
  FJOURNAL = {Numerische Mathematik},
    VOLUME = {96},
      YEAR = {2004},
    NUMBER = {4},
     PAGES = {771--800},
      ISSN = {0029-599X,0945-3245},
   MRCLASS = {76M10 (65N30 76W05)},
  MRNUMBER = {2036365},
MRREVIEWER = {Christian\ Rohde},
       DOI = {10.1007/s00211-003-0487-4},
       URL = {https://doi.org/10.1007/s00211-003-0487-4},
}

@article {BoonTonnonZampa26,
    AUTHOR = {Boon, W. M. and Tonnon, W. and Zampa, E.},
     TITLE = {H(curl)-based approximation of the {S}tokes problem with
              weakly enforced no-slip boundary conditions},
   JOURNAL = {Comput. Methods Appl. Mech. Engrg.},
  FJOURNAL = {Computer Methods in Applied Mechanics and Engineering},
    VOLUME = {448},
      YEAR = {2026},
     PAGES = {Paper No. 118484},
      ISSN = {0045-7825,1879-2138},
   MRCLASS = {65N12 (65N30 76D07)},
  MRNUMBER = {4973498},
       DOI = {10.1016/j.cma.2025.118484},
       URL = {https://doi.org/10.1016/j.cma.2025.118484},
}

@misc{BoonHiptmairTonnonZampa24,
 author = {Boon, W. M. and Hiptmair, R. and Tonnon, W. and Zampa, E.},
 title = {H(curl)-based approximation of the {Stokes} problem with slip boundary conditions},
 year = {2024},
 howpublished = {\href{https://doi.org/10.48550/arXiv.2407.13353}{arXiv:2407.13353}}
}

@article {BdVDassiDiPietroDroniou22,
    AUTHOR = {Beir\~ao da Veiga, L. and Dassi, F. and Di
              Pietro, D. A. and Droniou, J.},
     TITLE = {Arbitrary-order pressure-robust {DDR} and {VEM} methods for
              the {S}tokes problem on polyhedral meshes},
   JOURNAL = {Comput. Methods Appl. Mech. Engrg.},
  FJOURNAL = {Computer Methods in Applied Mechanics and Engineering},
    VOLUME = {397},
      YEAR = {2022},
     PAGES = {Paper No. 115061, 31},
      ISSN = {0045-7825,1879-2138},
   MRCLASS = {65N12 (65N30 65N99 76D07)},
  MRNUMBER = {4436363},
       DOI = {10.1016/j.cma.2022.115061},
       URL = {https://doi.org/10.1016/j.cma.2022.115061},
}

@article {DiPietroDroniuQian24,
    AUTHOR = {Di Pietro, D. A. and Droniou, J. and Qian, J.
              J.},
     TITLE = {A pressure-robust discrete de {R}ham scheme for the
              {N}avier-{S}tokes equations},
   JOURNAL = {Comput. Methods Appl. Mech. Engrg.},
  FJOURNAL = {Computer Methods in Applied Mechanics and Engineering},
    VOLUME = {421},
      YEAR = {2024},
     PAGES = {Paper No. 116765, 21},
      ISSN = {0045-7825,1879-2138},
   MRCLASS = {65N12 (14F40 65N30 76D05)},
  MRNUMBER = {4690519},
MRREVIEWER = {Jeonghun\ J.\ Lee},
       DOI = {10.1016/j.cma.2024.116765},
       URL = {https://doi.org/10.1016/j.cma.2024.116765},
}

@article {BdVHuMascotto25,
    AUTHOR = {Beir\~ao da Veiga, L. and Hu, K. and Mascotto,
              L.},
     TITLE = {Error estimates for a helicity-preserving finite element
              discretisation of an incompressible magnetohydrodynamics
              system},
   JOURNAL = {ESAIM Math. Model. Numer. Anal.},
  FJOURNAL = {ESAIM. Mathematical Modelling and Numerical Analysis},
    VOLUME = {59},
      YEAR = {2025},
    NUMBER = {2},
     PAGES = {1075--1094},
      ISSN = {2822-7840,2804-7214},
   MRCLASS = {65M15 (65M60 76W05)},
  MRNUMBER = {4886367},
       DOI = {10.1051/m2an/2025015},
       URL = {https://doi.org/10.1051/m2an/2025015},
}

@article {GawlikGB22,
    AUTHOR = {Gawlik, E. S. and Gay-Balmaz, F.},
     TITLE = {A finite element method for {MHD} that preserves energy,
              cross-helicity, magnetic helicity, incompressibility, and
              {${\rm div}\, B = 0$}},
   JOURNAL = {J. Comput. Phys.},
  FJOURNAL = {Journal of Computational Physics},
    VOLUME = {450},
      YEAR = {2022},
     PAGES = {Paper No. 110847, 20},
      ISSN = {0021-9991,1090-2716},
   MRCLASS = {65M60 (76W05)},
  MRNUMBER = {4346694},
       DOI = {10.1016/j.jcp.2021.110847},
       URL = {https://doi.org/10.1016/j.jcp.2021.110847},
}

@article {MitreaMonniaux2009,
 author = {Mitrea, M. and Monniaux, S.},
 title = {The nonlinear {Hodge}-{Navier}-{Stokes} equations in {Lipschitz} domains.},
 fjournal = {Differential and Integral Equations},
 journal = {Differ. Integral Equ.},
 issn = {0893-4983},
 volume = {22},
 number = {3-4},
 pages = {339--356},
 year = {2009},
 language = {English},
 zbMATH = {5944820},
 Zbl = {1240.35412}
}

@book {Boffi_Brezzi_Fortin:2013,
    AUTHOR = {Boffi, D. and Brezzi, F. and Fortin, M.},
     TITLE = {Mixed finite element methods and applications},
    SERIES = {Springer Series in Computational Mathematics},
    VOLUME = {44},
 PUBLISHER = {Springer, Heidelberg},
      YEAR = {2013},
     PAGES = {xiv+685},
      ISBN = {978-3-642-36518-8; 978-3-642-36519-5},
   MRCLASS = {65-02 (65M60 65N30)},
  MRNUMBER = {3097958},
MRREVIEWER = {Beny\ Neta},
       DOI = {10.1007/978-3-642-36519-5},
       URL = {https://doi.org/10.1007/978-3-642-36519-5},
}

@article {ErnGuermond:2017,
    AUTHOR = {Ern, A. and Guermond, J.-L.},
     TITLE = {Finite element quasi-interpolation and best approximation},
   JOURNAL = {ESAIM Math. Model. Numer. Anal.},
  FJOURNAL = {ESAIM. Mathematical Modelling and Numerical Analysis},
    VOLUME = {51},
      YEAR = {2017},
    NUMBER = {4},
     PAGES = {1367--1385},
      ISSN = {2822-7840,2804-7214},
   MRCLASS = {65N30 (65D05)},
  MRNUMBER = {3702417},
MRREVIEWER = {Hans-Peter\ Helfrich},
       DOI = {10.1051/m2an/2016066},
       URL = {https://doi.org/10.1051/m2an/2016066},
}

@misc{DiPietroDroniuPatierno26,
 author = {Di Pietro, D. A. and Droniou, J. and Patierno, V.},
 title = {A {Reynolds}- and {Hartmann}-semirobust hybrid method for magnetohydrodynamics},
 year = {2026},
 howpublished = {\href{https://doi.org/10.48550/arXiv.2602.09626}{ {arXiv}:2602.09626}},
 url = {https://arxiv.org/abs/2602.09626},
 arXiv = {arXiv:2602.09626}
}

@article {GardinerStone05,
    AUTHOR = {Gardiner, T. A. and Stone, J. M.},
     TITLE = {An unsplit {G}odunov method for ideal {MHD} via constrained
              transport},
   JOURNAL = {J. Comput. Phys.},
  FJOURNAL = {Journal of Computational Physics},
    VOLUME = {205},
      YEAR = {2005},
    NUMBER = {2},
     PAGES = {509--539},
      ISSN = {0021-9991,1090-2716},
   MRCLASS = {65M06 (76M20 76W05)},
  MRNUMBER = {2134992},
       DOI = {10.1016/j.jcp.2004.11.016},
       URL = {https://doi.org/10.1016/j.jcp.2004.11.016},
}

@article{OrszagTang79, title={Small-scale structure of two-dimensional magnetohydrodynamic turbulence}, 
volume={90}, 
DOI={10.1017/S002211207900210X}, number={1}, 
journal={J. Fluid Mech.}, 
author={Orszag, S. A. and Tang, C.-M.}, 
year={1979}, 
pages={129--143}}

@article{ngSolve,
  author={J. Sch{\"o}berl},
  title= {C++11 Implementation of Finite Elements in {NGS}olve},
  journal ={Technical Report ASC-2014-30, Institute for Analysis and Scientific Computing},
  year = {2014},
  month = {September}
}

@book {NIST-Handbook:2010,
 editor = {Olver, F. W. J. and Lozier, D. W. and Boisvert, R. F. and Clark, C. W.},
 title = {{NIST} handbook of mathematical functions},
 isbn = {978-0-521-19225-5; 978-0-521-14063-8},
 year = {2010},
 publisher = {Cambridge: Cambridge University Press},
 language = {English},
 zbMATH = {5765058},
 Zbl = {1198.00002}
}

@article {Amrouche,
    AUTHOR = {Amrouche, C. and Bernardi, C. and Dauge, M. and Girault, V.},
     TITLE = {Vector potentials in three-dimensional non-smooth domains},
   JOURNAL = {Math. Methods Appl. Sci.},
  FJOURNAL = {Mathematical Methods in the Applied Sciences},
    VOLUME = {21},
      YEAR = {1998},
    NUMBER = {9},
     PAGES = {823--864},
      ISSN = {0170-4214,1099-1476},
   MRCLASS = {35J25 (35Q30 65N30 76D07)},
  MRNUMBER = {1626990},
MRREVIEWER = {Juha\ H.\ Videman},
       DOI =
              {10.1002/(SICI)1099-1476(199806)21:9<823::AID-MMA976>3.0.CO;2-B},
       URL =
              {https://doi.org/10.1002/(SICI)1099-1476(199806)21:9<823::AID-MMA976>3.0.CO;2-B},
}

\end{document}